\documentclass[11pt]{amsart}
\usepackage{amscd,amsfonts,amsmath,amssymb,amstext,amsthm, mathrsfs}
\usepackage{setspace}
\usepackage{ytableau}
\date{}
\def\dim{\operatorname{dim}}

\def\Aut{\operatorname{Aut}}
\def\Ker{\operatorname{Ker}}

\def\Hom{\operatorname{Hom}}
\def\cHom{\operatorname{\check Hom}}

\def\gr{\operatorname{gr}}

\def\gr{{\rm gr}\,}

\input xy
\xyoption{all}
\theoremstyle{plain}
\newtheorem{theorem} {Theorem} [section]
\newtheorem{proposition} {Proposition} [section]
\newtheorem{lemma}   {Lemma} [section]
\newtheorem{corollary}   {Corollary} [section]
\theoremstyle{definition}
\newtheorem{definition}   {Definition} [section]

\newtheorem{remark}  {Remark} [section]

\newtheorem{thmy}{Theorem}

  \title[Prolongations and invariants  of geometric structures and fundamental identities ]{Prolongations, invariants, and fundamental identities of geometric structures}

 \author[J. Hong]{Jaehyun Hong}
   \address{Center for Complex Geometry,  Institute for Basic Science (IBS), Daejeon 34126, Republic of Korea}
    \email{jhhong00@ibs.re.kr}

   \author[T. Morimoto]{Tohru Morimoto}
 \address{Seki Kowa Institute of Mathematics, Yokkaichi University, Yokkaichi 512-8045,
 Institute Kiyoshi Oka de Math\'ematiques, Nara Women's University, Nara 630-8506, Japan}
   \email{morimoto@cc-nara-wu.ac.jp, morimoto.thr@gmail.com}

\begin{document}

\begin{abstract}
 Working in the framework of nilpotent geometry,
we give a unified scheme for the equivalence
problem of geometric structures
which extends and integrates the earlier works by Cartan, Singer-Sternberg, Tanaka, and Morimoto.

By giving a new formulation of the higher order geometric structures and the universal frame bundles, we reconstruct the step prolongation of Singer-Sternberg and Tanaka.
We then investigate the structure function $\gamma$ of the complete step prolongation of a   proper  geometric structure by expanding it into components
$\gamma = \kappa + \tau + \sigma$
and establish the fundamental identities
for $\kappa$, $\tau$, $\sigma$.
This then enables us to study the equivalence problem of geometric structures
in full generality and to extend applications largely to the geometric structures which have not necessarily Cartan connections.

Among all we give an algorithm to construct a complete system
of invariants for any higher order  proper  geometric structure of
constant symbol by making use of generalized Spencer cohomology
group  associated to
the symbol of the geometric structure. We then discuss thoroughly the equivalence problem for
geometric structure in both cases of infinite and finite type.

We also give a characterization of the Cartan connections by means of the structure function $\tau$
and make clear where the Cartan connections are placed in the perspective of the step prolongations.
\end{abstract}

 \maketitle

  \setcounter{tocdepth}{1}
\tableofcontents


\section*{Introduction}


The equivalence problem of  geometric structures is to find  criteria  and to determine whether two
arbitrarily given geometric structures are equivalent or not.
The  theory for the equivalence problem has been developed
over more than a hundred years
and has played an important role in differential geometry.

Recently motivated partly by a necessity to improve the theory for applications to
concrete problems, in particular,
 in complex geometry and in subriemannian geometry,
 and  partly for the unity and  completeness of the theory itself,
 we are led to reconsider
 the equivalence problem of geometric structures.

\medskip

 It was S. Lie  who first posed this problem in a general form  under his theory of continuous transformation groups, and developed a method to obtain the differential invariants of a geometric structure or a system of differential equations by integration of  the completely integrable systems on certain jet spaces defined by the action of a continuous transformation group (\cite{L84}).

From 1904 to 1909 in his series of papers (\cite{C04}, \cite{C05}, \cite{C08}) \' E. Cartan developed
the theory of continuous infinite groups, in the course of which he
 invented a heuristic method for general equivalence problems,
by using the method of bundles of moving frames and the theory of Pfaff systems  in involution, nowadays  called the Cartan-K\"ahler theory.

This ingenious idea of Cartan then found rigorous foundations
through the modern theories of equivalence problem as developed, in particular,
  in  I. M. Singer and S. Sternberg
  (\cite{SS65}),
  N. Tanaka
   (\cite{T70}, \cite{T79})
  and  the second-named author
   (\cite{M83},
   \cite{M93}).

Fundamental notions such as principal fiber bundles, $G$-structure, etc., were first used by Ehresmann and Chern, and the algebraic nature of the systems in involution was maid clear by M. Kuranishi, D. C. Spencer, S. Sternberg, V. Guillemin, J. P. Serre, D. Quillen,  H. Goldschmidt, and others (\cite{K57}, \cite{S62}, \cite{SS65}, \cite{GS64},  \cite {Q64},  \cite{G67a}).

Singer-Sternberg treated the equivalence problem of geometric structures as that of $G$-structures and
gave a 
foundation of an  aspect of Cartan's prolongation procedure, which we call here the step prolongation.

Tanaka  introduced a method of nilpotent approximation by making use of differential systems and
 began to develop   nilpotent geometry (\cite{T70}). First he extended the step prolongation of Singer-Sternberg
to a nilpotent version to prove the finite dimensionality of the automorphism group of a geometric structure of finite type in the  sense of nilpotent geometry.
Next he developed extensively   Cartan's  espace g\'en\'eralis\'e  (\cite{C22}) by constructing Cartan connections to  geometric structures
associated with simple graded Lie algebras (\cite{T79}), which plays a fundamental role in 
parabolic geometry (\cite{CS09}).

 A general theoretical method to solve completely the equivalence problem of geometric structures  was given
by the second-named author  (\cite{M83}),
by introducing  the notion of higher order non-commutative frame bundles.
This, in turn,
applied to nilpotent geometry, gave a unified method to construct Cartan connections as well as
a best possible criterion for constructing Cartan connections (\cite{M93}). It was also given a nilpotent version of the notion of
  involutive geometric structure  in terms of generalized Spencer cohomology groups.

\medskip

In the present paper, working in the framework of nilpotent geometry,
we give a unified scheme for the equivalence
problem of geometric structures
which extends and integrates the earlier works
(\cite{SS65}, \cite{T70}, \cite{T79}, and \cite{M83}),
and which is well adapted to applications.

\medskip

The base spaces of the geometric structures that we consider in this paper are filtered manifolds
of constant symbols:
A {\it filtered manifold} is a differentiable manifold $M$ equipped with a tangential filtration
$ \{ F^p \} _{ p \in   \mathbf Z} $ of depth $\mu$,
$$
TM = F^{-\mu } \supset F^{-\mu +1} \supset \cdots\supset F^{-1} \supset F^0 = 0 $$
satisfying
$ [ \underline{ F}^p, \underline{ F}^q ]
\subset
\underline{ F}^{p + q}
$, where $ \underline F $ denotes the sheaf of sections of $F$ (Section 1.1.1).

To every point of a filtered manifold $(M, F)$
there is attached a nilpotent graded Lie algebra
$\gr F_x$ called the {\it symbol (algebra)} of the filtered manifold at the point.
If the symbol algebra $ \gr F_x$ is isomorphic to a graded Lie algebra
$\mathfrak g _- = \oplus _{p < 0} \mathfrak g_p $ for all $x \in M$, we say that the filtered manifold $(M, F)$ has {\it constant symbol of type} $\mathfrak g_-$.
A usual differentiable manifold $M^n$ is regarded as a trivial filtered manifold $(M, F)$ with the trivial filtration $F^{-1} =TM$  of depth 1, of which the symbol algebra is isomorphic to the abelian Lie algebra  $\mathbb R^n$.

The {\it frame bundle} $\mathscr S^{(0)}(M,F)$ of order 0
of a filtered manifold $(M,F)$ of constant symbol $\mathfrak g_-$
is the set of all graded Lie algebra isomorphisms
 $$ z: \mathfrak g_- \to \gr F_x \quad (x \in M). $$
It is a principal fiber bundle over $M$
with structure group $G_0(\mathfrak g_-)$,
the group of all graded Lie algebra  automorphisms  of $\frak g_-$.

A $G$-structure on a filtered manifold $(M,F)$ (in the sense of nilpotent geometry) is a $G_0$-principal subbundle of
$\mathscr S^{(0)}(M,F)$, where $G_0$ is a Lie subgroup of $G_0(\mathfrak g_-)$.


\medskip

We know many important examples of $G$-structures in the usual sense.
And there are also many interesting $G$-structures in the nilpotent sense, such as contact structures, CR-structures, and various geometric structures linked to differential systems. Indeed, most of the geometric structures studied in differential geometry are $G$-structures in the extended nilpotent sense, which we define in this paper as geometric structures of order 0.

There are as well geometric structures which should be regarded as order 1 such as projective structures and contact projective structures.

Our first task in this paper is to give a general definition of higher order geometric structures, which is indispensable to
find higher order differential invariants  of a given geometric structure that may be even of lower order.
How to define the geometric structures  depends  on  how  to derive the differential invariants.

In order to well formulate the step prolongation in a clearer setting, we present a new category of higher order geometric structures by broadening the category of towers introduced in \cite{M93}.

  \medskip

A {\it geometric structure  $Q^{(k)}$ of order $k\ge -1$ of type} (or {\it of symbol})
$(\mathfrak g_-,G_0, \dots, G_k) $ is defined to be a series of principal fiber bundles
$$Q^{(k)} \stackrel{G_k}{\longrightarrow} Q^{(k-1)} \stackrel{G_{k-1}}{\longrightarrow} \dots \stackrel{G_1}{\longrightarrow} Q^{(0)} \stackrel{G_0}{\longrightarrow} Q^{(-1)}$$
satisfying  that (1) $Q^{(-1)}$ is a filtered manifold $( M,F)$ of type $\mathfrak g_-$, and
(2) for $0 \le i \le k$, $Q^{(i )}$ is a geometric structure of order $i $ of type $(\frak g_-, G_0, \dots, G_{i })$ and is a principal subbundle of the universal frame bundle $ \mathscr
S^{(i)}Q^{(i-1)}$ of order $i$ of $Q^{(i-1)}$,
while the universal frame bundle is defined by the properties of naturalness
and universality (Definition \ref{def:universal frame bundle}). 
%
  Therefore, geometric structures of order 0 are $G$-structures in the sense of nilpotent geometry as mentioned above.
   However, the class of higher-order geometric structures is quite large and includes ``virtual" geometric structures that do not appear in ``real" geometry but are theoretically necessary.

Setting
 $\mathscr S^{(\ell)} Q^{(k)} =
 \mathscr S^{(\ell)}\mathscr S^{(\ell-1)} Q^{(k)}$
and passing to the projective limit,
we obtain  the {\it completed universal frame bundle}
 $$ \mathscr S Q^{(k)} =\lim_{\ell} \,\mathscr S^{(l)} Q^{(k)}, $$
 which proves to be a principal fibre bundle over $Q^{(k-1)}$
 equipped with a canonical 1-form $\theta=\theta_{\mathscr SQ^{(k)}}$
giving an isomorphism
$$\theta_z: T_z\mathscr SQ^{(k)} \to
 E(\mathfrak g _-,
 \mathfrak g_0
, \dots ,
\mathfrak g_{k}) \quad \text{ for all } z \in \mathscr SQ^{(k)}
 $$
and then defining an absolute parallelism on
$\mathscr SQ^{(k)}$, where $ E(\mathfrak g _-,
 \mathfrak g_0
, \dots ,
\mathfrak g_{k}) $ is the universal graded vector space determined by
$  (\mathfrak g _-,
 \mathfrak g_0
, \dots ,
\mathfrak g_{k}) $, which we denote simply by $E$.

The first fundamental observation is that
%
the equivalence problem of the geometric structures $Q^{(k)}$
reduces to the equivalence  problem of the absolute parallelism
$(\mathscr S Q^{(k)},\theta_{\mathscr SQ^{(k)}})$ of the completed universal frame bundles (Theorem \ref{thm:isomorphism via canonical class of any order}).

The invariants of the absolute parallelism  are  given by its structure function $\gamma$, a
$\Hom(\wedge^2E,E)$
-valued function on
$\mathscr SQ^{(k)}$
defined by:
$$d\theta+\frac12\gamma(\theta \wedge \theta)=0. $$


 All information on  $Q^{(k)}$ is encoded in the structure function of $\mathscr S Q^{(k)}$. However, the completed universal frame bundle $\mathscr S Q^{(k)}$ is   infinite dimensional and of large magnitude.
   Following the idea of Cartan, we make a reduction of the universal frame bundle
  by using the structure function to obtain a smaller subbundle   representing
  the invariants more effectively, which is carried out in Section 3.

 To do that, we study basic properties of structure function. In particular, we have the following decomposition:
 $$ \gamma =\gamma _{I} +\gamma _{II}+\gamma _{III}
 =\kappa + \tau + \sigma ,
 \quad
 \text{ with }
 \gamma _{I}=\kappa ,\,\, \gamma _{II}=\tau,\,\, \gamma _{III}=\sigma$$
according to
 the direct sum decomposition
   $$\Hom(\wedge ^2E,E)=
 \Hom(E_-\wedge E_-,E) \oplus \Hom(E_+\wedge E_-,E)\oplus \Hom(E_+\wedge E_+,E),$$
 where
 $ E_- = \bigoplus _{p < 0} E_p ,
  E_+ = \bigoplus _{p \ge 0} E_p
  $.
 We also decompose $\gamma$ as
 $$\gamma =\sum \gamma _{ p },\ \gamma ^{(l)}=\sum_{p \leq l}\gamma _p,
 \quad \text{ and } \quad \gamma =\sum \gamma _{[p]},\ \gamma ^{[l]}=\sum_{p \leq l}\gamma _{[p]},$$
 according to the homogeneous degree and
 modified homogeneous degree of $ \Hom (\wedge ^2 E, E)$, respectively. For the definition, see Section \ref{sect: structure functions of various degree}.
 We then show that
 the truncated structure function
 $\gamma ^{[l]}$
 on $ \mathscr SQ^{(k)}$
 is a function on $\mathscr S ^{(l)} Q^{(k)}$.

\medskip

In Section 3
 we  first introduce the class of  proper  geometric structures.
 We say that a geometric structure $Q^{(k)}$ of type
  $( \mathfrak g_- , G_0, \dots,G_k )$
  is  proper  if  $( \mathfrak g_- , \mathfrak g_0, \dots, \mathfrak g_k )$
  forms a truncated transitive graded Lie algebra $\frak g[k]$.
  It is the   proper  geometric structure that has moderate magnitude and
  appears 
  actually in real geometry.  See the remark after Definition \ref{def:normal geometric structure}.

         Now let us make the key procedure of $W$-normal reduction.
    Given a truncated transitive graded Lie algebra
$\mathfrak g [k] =
\bigoplus_{p \le k}
\mathfrak g_p$
 and Lie groups $G_0, \dots, G_k$ with Lie algebras
 $\mathfrak g_0, \dots, \mathfrak g_k$,
let $\mathfrak g = \bigoplus \mathfrak g_p $ be the prolongation of $\mathfrak g[k]$,
 $\mathfrak g [l] =
\bigoplus_{p \le l}
\mathfrak g_p$,
and $G_{\ell} $ be the vector group $\mathfrak g_{\ell}$
for $\ell >k$.    Here, the vector group $\mathfrak g_{\ell}$ means the same vector space viewed as a commutative Lie group.
 Fix complementary subspaces
 $W=\{W^1_{\ell}, W^2_{\ell+1}\}_{\ell \geq k}$ such that
\begin{eqnarray*}
 \Hom(\wedge^2\frak g_-, \frak g)_{\ell+1} &= &W_{\ell+1}^2 \oplus \partial \Hom(\frak g_-, \frak g)_{\ell+1} \\
 \Hom (\frak g_-, \frak g)_{\ell} &= &W_{\ell}^1 \oplus \partial \frak g_{\ell}.
 \end{eqnarray*}

For any    proper  geometric structure
$Q^{(k)}$ of type       
$(\mathfrak g_-, G_0. \dots, G_k)$,
set
$\mathscr S_W ^{(k)} Q^{(k)} = Q^{(k)}$ and
$$ \mathscr S_W ^{(\ell+1)} Q^{(k)}
 =\{  z \in \mathscr S ^{(\ell+1)}
  \mathscr S_W ^{(\ell)} Q^{(k)} :
  \kappa^{[\ell+1]}(z) \in W^2_{l+1} ,
 \tau^{[\ell+1]}(z) \in   \Hom(\oplus_{i=0}^{l-1}\mathfrak g_i, W^1_{l})
 \}.
 $$
 Then
 $\mathscr S_W ^{(\ell)} Q^{(k)} $
is a  proper  geometric structure of type
  $\mathfrak g[\ell]$
and the projective limit
$\mathscr S_W  Q^{(k)}
  =\lim_{\ell} \mathscr S_W ^{(\ell)} Q^{(k)}   $
 is endowed with a canonical  $\mathfrak g$-valued 1-form
 $ \theta = \theta _{\mathscr S_W} $
 which gives an absolute parallelism on
 $\mathscr S_W  Q^{(k)}$.

  We call
 $\mathscr S_W  Q^{(k)}  $
 the {\it $W$-normal step prolongation} of $Q^{(k)} $
 or the {\it $W$-normal reduction} of
 $\mathscr S  Q^{(k)} $.
  We then have

 \begin{thmy} [Theorem \ref{thm:W normal prolongation}]  \label{theorem II}
 The equivalence problem of the  proper  geometric structures $Q^{(k)}$
of type
$(\mathfrak g_-, G_0. \cdots, G_k)$,
reduces to the equivalence  problem of the absolute parallelisms
$(\mathscr S_W Q^{(k)}, \theta)$ of the step prolongations.
Moreover if $\mathfrak g[k]$ is of finite type, that is, the prolongation
$\mathfrak g $
is finite dimensional, then so is
$\mathscr S_WQ^{(k)}$
and
$\dim \mathscr S_WQ^{(k)}= \dim \mathfrak g$.

\end{thmy}


This is one of the main theorems of
\cite{SS65} and \cite{T70}.
The $W$-normal step  prolongation
$\mathscr S^{(\ell)}_WQ^{(0)} $
essentially coincides with what Tanaka constructed.
We thus reformulate the construction of step prolongation in our new scheme of the complete universal frame bundles.

\medskip
By virtue of the conceptional construction of step prolongations, we are naturally led
to the fundamental identities: In Section 4 we prove the following theorem, which plays
a key role in our prolongation scheme.


\begin{thmy} [Theorem \ref{thm:fundamental identities}] \label{theorem III}
Let $Q^{(k)}$ be a  proper  geometric structure of order $k$ and let
$\mathscr S _WQ^{(k)} $ be
 the $W$-normal step prolongation of
 $Q^{(k)}$, and let
$\gamma = \gamma_{I} + \gamma_{II} + \gamma_{III} $ be
 its structure function.
Then we have
$$
\begin{aligned}
1)\,\, & \partial \gamma _{ I [k]} &= & \,\,\,\Psi _{I [k]}( \gamma _{ I [i]},\gamma _{ II[i]};i<k)\\
2)\,\,& \partial \gamma _{ II [k]} &= & \,\,\,  \Psi _{II [k]}( \gamma _{ I [i]},\gamma _{ II[i]},\gamma _{ III [i]};i<k ) \\
3)\,\,& \partial \gamma _{ III [k]}& = & \,\,\,\Psi _{III [k]}( \gamma _{ II[i]},\gamma _{ III [i]};i<k )
\end{aligned}
$$
where $\Psi_{X[k]}$  for $X \in \{I,II,III\}$  is a polynomial in $\gamma_{Y[i]}$   with $Y \in \{I,II,III\}$ and $i<k$, and  their  covariant derivatives.
\end{thmy}

 In the above theorem, we have written the fundamental identities simply by trivially extending the Spencer cochain complex. We provide    more precise and explicit formulas, which give  the identities for the invariants of a geometric structure in the most general form (Section 4).
In the special case of Riemannian structures they are known as the first and second Bianchi identities.
We remark also that if there  is  no $\gamma _{II}$ and $\gamma_{III}$, the above identities   reduce to those known in the geometry of Cartan connection.

From the fundamental identities it follows many important consequences  as described in  Section \ref{sec:Invariants} -- Section \ref{sect:Cartan connection}, each of which is of independent interest.


\medskip

In Section \ref{sec:Invariants}
we show that, as an immediate consequence of Theorem \ref{theorem II} and \ref{theorem III}, and  on account of the vanishing of the generalized Spencer cohomology group for higher degree,  we obtain generators of a fundamental system of invariants after a finite number
of  steps in  prolongation procedure even  in case of   a geometric structure of infinite type.

\begin{thmy}  [Theorem  \ref{thm:complete system of invariants}]  \label{theorem IV}
Let $\frak g[k]$ be a truncated transitive graded Lie algebra and $\frak g$ be its prolongation.
Let
\begin{eqnarray*}
I^1 &=& \{ i \in \mathbb Z_{>0} : H^1_i (\frak g_-,\frak g) \not=0 \} \\
I^2&=& \{ i \in \mathbb Z_{>0} : H^2_i (\frak g_-, \frak g) \not=0 \},
\end{eqnarray*}
 which are finite sets by Theorem \ref{vanishing cohomology of higher degree}.

Let $Q^{(k)}$  be  a  proper  geometric structure of type $\frak g[k]$.
Then the following invariants, called the set of  {\it essential invariants}  of  $Q^{(k)}$,
$$\{  D^m  \kappa_{[i]}: i \in I^2, m\geq 0\} \text{ and } \{D^m\tau_{[i]}: i-1 \in I^1, m \geq 0\}$$
  form  a fundamental system of invariants of $Q^{(k)}$.

 \end{thmy}

Thus we have given an algorithm to find a fundamental system of invariants for any
 proper  geometric structure, which solves the first half of the equivalence problem (Necessity).

\medskip
In Section \ref{sect:equivalence}
we consider the second half of equivalence problem (Sufficiency) and pose
the following question: \\

\noindent
[Equivalence Problem (Sufficiency)] Let $P^{(k)}$ and $Q^{(k)}$ be two proper geometric structures of order $k$.
Assume that there is an isomorphism between
the (fundamental system of, or essential) invariants of   $P^{(k)}$
and those of $Q^{(k)}$ (in a suitable sense). Then, is there an isomorphism
 between $P^{(k)}$ and $Q^{(k)}$? \\




We say that a proper  geometric structure $P^{(k)}$ is {\it involutive} if it is quasi-involutive (i.e., $H^1_{\ell}
(\mathfrak g_-,\mathfrak g)=H^2_{\ell+1}(\mathfrak g_-,\mathfrak g)=0 $
for $\ell\ge k$) and the structure function $\gamma ^{[k]}$ 
of $P^{(k)}$ is constant.
%
%
We then prove

\begin{thmy}  [Theorem \ref{thm:involutive local equivalence}] \label{theorem VI} The answer to the  equivalence problem (Sufficiency)   is affirmative for the involutive geometric structures
in the analytic category,
 and  for those of finite type in the $C^{\infty}$-category.
\end{thmy}

For the proof we   rely on the Cartan-K\"ahler theorem or its generalization to a nilpotent version.
For the geometric structure of finite type the theorem holds in the $C^{\infty}$-category, which is easily proved
by the Frobenius theorem.

In particular, we obtain as a corollary that  if the fundamental invariants vanish then the geometric structure is equivalent to a standard one, under the assumption of analyticity if it is of infinite type.
%
For the finite type analytic geometric structures we have a
more general theorem worth noting:
%
 The
 equivalence problem (Sufficiency)  is affirmative for
the analytic geometric structures of finite type (Theorem \ref{thm: analytic and of finite type}).

\medskip

In Section \ref{sect:Cartan connection}
we study relations between
 the $W$-normal step prolongations and the  Cartan connections
by using  fundamental identities,
and make clear where the Cartan connections are placed in the perspective of step prolongations.

For this, we revisit the notion of tower and reformulate it to define the category of principal geometric structures, which is parallel and in a sharp contrast to that of geometric structures
introduced in Section 2. This then allows us to characterize the proper  principal geometric structure among the  proper  geometric structures, and then the Cartan connection
among the $W$-normal step prolongations in terms of the structure function $\tau$.

We first show that   Cartan connections are   almost equal to   complete proper
geometric structures with constant structure function $\tau$ (Theorem 7.1).
More precisely and specifically we prove:

\begin{thmy} [Theorem \ref{thm:vanishing of tau implies}]   \label{theorem VIII} Let $Q^{(0)}$ be a geometric structure of order 0 with
a connected structure group $G_0$
and $\mathscr S _WQ^{(0)}$
the $W$-normal step prolongation. If $\tau^{[k+1]}$ is flat, then $\mathscr S_W^{(k)}Q^{(0)} \rightarrow M$ is a proper principal geometric structure of order $k$.
If the structure function $\tau$ of $\mathscr S _W Q^{(0)}$
is flat, then $\mathscr S _W Q^{(0)}$ is a Cartan connection.
\end{thmy}

 Here, we say that $\tau^{[k+1]}$ is {\it flat}  if $\tau_{(\ell)[m]}$ is zero for any $\ell>0$ and $m \leq k+1$ and that $\tau$ is {\it flat} if $\tau_{(\ell)[m]}$ is zero for any $\ell>0$ and $m$.

From this we obtain  a more   precise result
well adapted to applications in complex geometry:
 the fundamental invariants can be regarded as  sections of bundles on the base manifold inductively as in the case when there exists a Cartan connection associated
with a given geometric structure
(Theorem \ref{thm: sections all vanishing}).

 We also prove:

\begin{thmy} [Theorem \ref{thm:Cartan connection}] \label{theorme IX}
 If the structure group $G_0$ of a geometric structure $Q^{(0)}$ of order 0 satisfies the condition (C), then $W$-normal step prolongation $\mathscr S_W Q^{(0)}$ coincides with the $W$-normal Cartan connection  $\mathscr R_WQ^{(0)}$ constructed in \cite{M93}.
 \end{thmy}

Thus the $W$-normal step prolongations contain, as special cases, the Cartan connections, and therefore have vast applications
 to the geometric structures which do not necessarily admit Cartan connections.

\medskip
In Section 8 we explain how our method is applied to subriemannian geometry
and   complex geometry.



We give a mention to the works \cite{AD17}, \cite{Z09} which provide alternative proofs of Tanaka's step prolongation.
In our opinion, without the fundamental identities, the prolongation scheme
would not be as clear and as effective  as shown
 in the present paper.

 It is to intrinsic geometry that our present paper is concerned.
   In the paper \cite{DMM}
 a general theory on extrinsic geometry
 is developed
 and it is made clear the similarity existing between
intrinsic and extrinsic geometry.
The present paper will make more
visible and highlight    this similarity.

In this paper we have confined ourselves to geometric structures of constant symbols.
But if we use the method in \cite{M83}
to treat prolongations of $G$-structures with varying structure group,  we can extend the prolongation scheme of the present paper to the geometric structures of non-constant symbol, which will be an object of our forthcoming paper. \\

\noindent{ \bf Acknowledgements.}

We would like to thank  Yoshio Agaoka who carefully read through
the manuscript and gave us many valuable comments.

The first named author was supported by the Institute for Basic Science IBS-R032-D1. The second named author was partially supported by JSPS KAK-ENHI Grant Number 17K05232.



\section{Preliminaries}

\subsection{Geometric preliminaries}

 \subsubsection  
A {\it filtered manifold} is  a differentiable manifold $M$ equipped with a filtration $F=\{F^p\}_{p \in \mathbb Z}$ of the tangent bundle $TM$ of $M$ satisfying:
\begin{enumerate}
\item $F^p$ is a subbundle of $TM$ and $F^p \supset F^{p+1}$ for all $p \in \mathbb Z$;
\item $\cup_{p \in \mathbb Z} F^p = TM$ and $F^0=0$;
\item $[\underline{F}^p, \underline{F}^q] \subset \underline{F}^{p+q}$ for all $p,q \in \mathbb bZ$, where $\underline{F}^{\bullet}$ is the sheaf of sections of $F^{\bullet}$.
\end{enumerate}
The minimal integer $\mu$ such that $F^{-\mu} = TM$ is called the {\it depth} of $F$.
A filtered manifold will be denoted as $(M, F)$, or $(M, F_M)$,
or simply as $M$.
The filtration  is written  as $\{F^p\}$,  or  alternatively, as
$\{F^p_M\}$,
$\{F^pTM\}$ or $\{T^pM\}$.

Let $(M,F)$ be a filtered manifold. For each $x \in M$ we set
$$\gr F_x=\oplus_p \gr_pF_x, \quad \gr_p F_x = F^p_x/F^{p+1}_x$$
where $F_x^p$ denotes the fiber of $F^p$ over $x$.
The bracket operation of vector fields induces a bracket
$$[\,,\,]: \gr_pF_x \times \gr_qF_x \rightarrow \gr_{p+q}F_x.$$
Then $\gr F_x$ becomes a nilpotent graded Lie algebra and is called the {\it symbol (algebra) of} $(M,F)$ at $x$.

An {\it isomorphism} of a filtered manifold $M$ onto another $N$
is a diffeomorphism
$\varphi : M \to N $ which preserves the filtrations, that is,
$\varphi_* F^p_M = F^p_N $.
It should be noted that if $\varphi : (M, F_M) \to (N, F_N) $ is an isomorphism of filtered manifolds, then
$\varphi$ induces a graded Lie algebra isomorphism
$ gr \varphi _* : \gr F_{M, x} \to \gr F_{N, \varphi (x)}$.

We say that a filtered manifold $(M,F)$ is {\it of (constant symbol  of) type} $\frak g_-$ if $\frak g_-=\oplus_{p<0} \frak g_p$ is a graded Lie algebra and $\gr F_x$ is isomorphic to $\frak g_-$ as graded Lie algebras for all $x \in  M$.


In this paper, we will be  almost exclusively concerned  with filtered manifolds which
are of constant symbol.   Unless otherwise stated, this will always be assumed.

\subsubsection{ } \label{sect:reduced frame bundel of order 0}

Let $\frak g_- = \oplus_{p<0}\frak g_p$ be a graded Lie algebra and  $(M,F)$ be a filtered manifold of type $\frak g_-$. For each $x \in M$ we set
$$\mathscr S_x^{(0)}(M):=\{z:\frak g_- \rightarrow \gr F_x: \text{ graded Lie algebra isomorphisms }\}$$
and $\mathscr S^{(0)}(M)=\cup_{x \in M} \mathscr S^{(0)}_x(M)$.
We denote by $G_0(\frak g_-)$ the Lie group consisting of all graded Lie algebra automorphisms of $\frak g_-$.
Then
$\mathscr S^{(0)}(M)$ is a principal fiber bundle over $M$ with structure group $G_0(\frak g_-)$.
%
%
The bundle $\mathscr S^{(0)}(M)$ is called the {\it (reduced)  frame bundle of} $(M,F)$ {\it (of order $0$)} in the sense of nilpotent geometry.

It is a quotient space  of {\it the {(reduced) frame bundle} $\widehat{\mathscr S}^{(0)}(M)$ of $(M,F)$ (of order $0$)} in the usual sense  defined as follows.  We regard the graded Lie algebra $\frak g_-$ as a filtered vector space with the filtration $\{F^p\frak g_-\}$ given by $F^p\frak g_- =\oplus_{i \geq p}\frak g_i$. For $x \in M$, let $\widehat{\mathscr S}^{(0)}_x(M)$ be the set of all linear isomorphisms $\zeta: \frak g_- \rightarrow T_xM$ such that $\zeta$ preserves the filtration and
$$[\zeta]:\gr \frak g_-(=\frak g_-) \rightarrow \gr F_x$$
is an isomorphism of graded Lie algebras, and let $\widehat{\mathscr S}^{(0)}(M) = \cup_{x \in M}\widehat{\mathscr S}^{(0)}_x(M)$. Then $\widehat{\mathscr S}^{(0)}(M)$ is a principal fiber bundle over $M$ with structure group $\widehat{G}_0(\frak g_-)$ consisting of all filtration preserving  automorphisms  $\alpha$ of $\frak g_-$ such that the induced map $[\alpha]$ is contained in $G_0(\frak g_-)$.

  For a   vector space $V$ with a filtration $ \{F^pV\}_p$, $\Hom(V,V)$ and $GL(V)$ are filtered as follows:
\begin{eqnarray*}
F^{p}\Hom(V,V)&=&\{A \in \Hom(V,V): A(F^qV) \subset F^{p+q}V \text{ for all } q \}\\
F^p GL(V) &=& \{ a \in GL(V): a - id_V \in F^p \Hom(V,V)\}.
\end{eqnarray*}
 We see that
$$G_0(\frak g_-) =\widehat{G}_0(\frak g_-)/F^1, \quad \mathscr S^{(0)}(M) = \widehat{\mathscr S}^{(0)}(M)/F^1,$$
where $F^1 (= F^1\widehat{G}_0 (\frak g_-))$ is the  subgroup of
$\widehat{G}_0(\frak g_-)$
defined by the induced filtration.

 We remark that the definition of frame bundles $\mathscr S^{(0)}(M)$ and
 $ \widehat{\mathscr S} ^{(0)} (M)$ depends on the choice
 of a graded Lie algebra $\frak g _-$. To avoid the ambiguity
 we shall fix, once and for all,  one representative $\frak g_- $
  for each equivalence class of   graded Lie algebras.

  Then for each isomorphism
  $\varphi: M \to N $ of filtered manifolds, we have
  the associated bundle isomorphisms, called {\it the lift or the prolongation} of
  $\varphi$:
   $$
     \widehat{\mathscr S} ^{(0)}\varphi:
 \widehat{\mathscr S} ^{(0)}M
 \to
 \widehat{ \mathscr S} ^{(0)}N,
 \quad
 \mathscr S ^{(0)}\varphi:
 \mathscr S ^{(0)}M
 \to
  \mathscr S ^{(0)}N
  $$
 where
 $
 \widehat{\mathscr S}
 ^{(0)}
 \varphi (\zeta)
 = \varphi _* \circ \zeta, \,
 \mathscr S^{(0)}\varphi(z)
 = gr \varphi_* \circ z $
 for
 $\zeta \in
 \widehat{\mathscr S}
 ^{(0)} M ,
 z \in
\mathscr S^{(0)} M$.

 \medskip

 \subsubsection{ }
 \label{sect:geometric structure of order 0}

We define a {\it geometric structure of order 0} on a filtered manifold $(M, F)$
{\it of type} $\frak g_-$ to be  a reduction of the frame bundle
$\mathscr S ^{(0)} (M)$
to a subgroup
$G_0 \subset G_0(\frak g _-)$,
that is, a $G_0$-principal subbundle
$Q^{(0)}$
of $\mathscr S ^{(0)} (M)$.
Two geometric structures
 $P_0 \rightarrow (M, F_M)$
  and $Q_0 \rightarrow (N, F_N) $
  are said to be {\it isomorphic}, or {\it equivalent}
   if there is an isomorphism $\varphi:(M, F_M)\rightarrow
   (N, F_N) $ such that the induced map
   $\mathscr S^{(0)}\varphi:
   \mathscr S^{(0)}(M) \rightarrow \mathscr S^{(0)}(N )$
   maps $P_0$ onto $Q_0$.

Thus a  geometric structure of order 0 may be
called a $G$-structure in the sense of nilpotent geometry:
if the underlying filtered manifold is of depth 1,
 i.e., a trivial filtered manifold,
it is  just  what is usually called  a $ G $-structure.

 We know various important examples of $G$-structures,
and almost all geometric structures that we treat in differential geometry are geometric structures
of order 0 of the above definition or their generalizations to higher orders.
 In the proceeding sections we shall study the equivalence problem of geometry structures  of order 0 or higher order.


\subsection{Algebraic preliminaries}

\subsubsection{ }
An infinitesimal algebraic model of a transitive geometric structure
on a filtered manifold is represented by a transitive filtered Lie algebra
defined as follows:
%
 A {\it transitive filtered Lie algebra}   is a Lie algebra
$ L$
equipped with
a decreasing filtration  $\{F^p L\}$
satisfying the following conditions:
\begin{enumerate}
\item $[ F^p L, F^qL] \subset F^{p+q}L $
\item $\dim L/F^{0}L < \infty$,
 \item $\cap_p F^pL =0$
 \item (transitivity)
 If $i \geq 0, X \in F^iL$
 and if $[X, F^aL] \subset F^{i+ a + 1} L$ for all $ a < 0$,
 then $ X \in F^{i+1} L $.

\item $L$ is complete with respect to the topology which makes $
\{ F^p\}$ as a fundamental system of neighbourhoods of the origin.

\end{enumerate}

Passing to the graded object
$gr L = \bigoplus F^pL/ F^{p+1}L $, we have
%
a {\it transitive graded Lie algebra}, that is,  a Lie algebra
$\frak g$
equipped with
a grading  $\frak g=\oplus_i \frak g_i$
satisfying the following conditions:
\begin{enumerate}
\item $[ \frak g _i, \frak g_j] \subset \frak g_{i+ j} $
\item $\dim \frak g_- < \infty$,
where we put  $\frak g_- = \oplus_{p<0} \frak g_p$.
\item (transitivity)\ If $X  \in \frak g_i$  ($i \geq 0$) and $[X , \frak g_-]=0$, then $X =0$.
\end{enumerate}

This (or its completion)
   gives  an infinitesimal algebraic model of a {\it flat } geometric structure.

\subsubsection{ }

 A {\it truncated transitive graded Lie algebra of order $k$}
$(k \geq -1)$ is a graded vector space
 $\frak g[k]=\oplus _{i \leq k}\frak g_i$
  endowed with a bracket operation $[\,\,,\,\,]: \frak g_i \otimes \frak g_j \rightarrow \frak g_{i+j}$
  which is defined only for $i,j,i+j \leq k$
  and satisfies the Jacobi identity whenever it makes  sense,
  and moreover satisfies the conditions (2) and (3) in the definition of a transitive graded
  Lie algebra.

  Given a transitive graded Lie algebra $\frak g = \oplus_{p}\frak g_p$ with the Lie bracket $[\,\,,\,\,]$, the graded vector space $\oplus_{p\leq k}\mathfrak g_p$ with the bracket operation $[\,\,,\,\,]^{(k)}$ defined by
  \begin{eqnarray*}
  [X,Y]^{(k)}=\left\{\begin{array}{ll}
  [X,Y]& \text{ if } X\in \frak g_i, Y \in \frak g_j, \text{ and } i+j\leq k \\
  0 & \text{ otherwise,}
  \end{array}\right.
  \end{eqnarray*}
  is a truncated transitive graded Lie algebra of order $k$, denoted by $Trun ^{(k)} \, (\frak g)$. We call $Trun ^{(k)} \, (\frak g)$ a {\it truncation} of $\frak g$.

Note that $ \frak g_- $ and $\frak g_- \oplus \frak g_0$
 are truncated transitive graded Lie algebras,
 where $\frak g_-$ is a graded Lie algebra  negatively concentrated and of finite dimension,
 and $\frak g_0$ is a Lie subalgebra of the Lie algebra
 $Der_0(\frak g_- ) $ of all derivations of degree 0 of $\frak g_-$.
  Fundamental is the following:

\begin{proposition} [\cite{SS65}, \cite{T70}]
For a truncated transitive graded Lie algebra
$\frak g [k]=\oplus _{p \le k} \frak g_p $ there exists, uniquely up to isomorphisms,
 a maximal transitive graded Lie algebra
$\bar {\frak g} = \oplus \bar{\frak g} _p$
such that $\frak g [k]$ is a truncation of $\bar{\frak g}$.
\end{proposition}

In fact, we can construct the transitive graded Lie algebra
$\bar {\frak g} = \oplus \bar{\frak g} _p$
by setting $\bar{\frak g}_p = \frak g_p$
for $p \le k$ and then inductively so as to satisfy
$$ \bar{\frak g} _{p+1} = \{ \alpha \in Hom( \frak g _- , \oplus_{i \leq p} \bar{\frak g}_i )_{p+1}:
\alpha([u, v]) = [\alpha(u), v] + [u, \alpha (v)], \  u, v \in \frak g_-\}
$$
as well as the Jacobi identity,  whenever it is defined.

The transitive graded Lie algebra $\bar {\frak g}$ is called the {\it prolongation} of $\frak g [k] $ and denoted by $Prol (\frak g [k])$.


 \subsubsection{ }

Let $\frak g$ be a transitive graded Lie algebra
and $\mathfrak g_-=\oplus_{p<0} \mathfrak g_p$
be its negative part. The adjoint representation
of $\frak g_-$ on $\frak g$ gives rise to the cohomology
group $ H(\frak g _- , \frak g)
=\oplus H^q (\frak g _- , \frak g)$, called {\it the generalized Spencer cohomology
group}. It is the cohomology group associated to the
following chain complex:
$$0 \longrightarrow \frak g  \stackrel{\partial}{\longrightarrow} \cdots
 \stackrel{\partial}{\longrightarrow}
  \Hom(\wedge^q \frak g_-, \frak g)
 \stackrel{\partial}{\longrightarrow}
 \Hom(\wedge^{q+1} \frak g_-, \frak g)
 \stackrel{\partial}{\longrightarrow} \cdots . $$

The coboundary operator is defined as follows:
 For $c \in \Hom(\wedge^q \frak g_-, \frak g)$, define  $\partial c \in \Hom(\wedge^{q+1}\frak g_-, \frak g)$  by
\begin{eqnarray*}
\partial c(v_1, \dots, v_{q+1}) &=& \sum_i (-1)^{i+1}[v_i, c(v_1, \dots, \hat{v}_i, \dots, v_{q+1})] \\ && \qquad \quad  + \sum_{i<j} (-1)^{i+j} c([v_i, v_j], v_1, \dots, \hat{v}_i, \dots, \hat{v}_j, \dots, v_{q+1}).
\end{eqnarray*}

Let $\Hom(\wedge \frak g_-,  \frak g)_r$
denote the set of all maps of degree $r$.
Then,
since the coboundary operator preserves the degree,
we have the subcomplex:
$$0 \longrightarrow \frak g_r  \stackrel{\partial}{\longrightarrow} \cdots
 \stackrel{\partial}{\longrightarrow}
  \Hom(\wedge^q \frak g_-, \frak g)_r
 \stackrel{\partial}{\longrightarrow}
 \Hom(\wedge^{q+1} \frak g_-, \frak g)_r
 \stackrel{\partial}{\longrightarrow} \cdots . $$
The associated cohomology group is denoted by $H^q_r(\frak g_-, \frak g)$. Then we have
$$ H(\frak g _- , \frak g)
=\oplus H^q_r(\frak g _- , \frak g).$$
It is
$ H^q_+(\frak g _- , \frak g)
:=\oplus _{r >0} H^q_r( \frak g _- , \frak g)$ where $q = 1, 2$
that plays an important role
in our intrinsic geometry.
Note that $\frak g$ is the prolongation of its truncation $\frak g [k]$
if and only if
$H^1_r( \frak g _- , \frak g) = 0 $ for $ r > k$.

We have the   following finitude of the cohomology group.

\begin{theorem}[Theorem of finitude, \cite{M88}]  \label{vanishing cohomology of higher degree}
For every transitive graded Lie algebra $\frak g$
there exists an positive integer $r_0$
such that
$$H^q_r( \frak g _- , \frak g) = 0
\quad  {\text for\,  all\, } q \geq 1
{\,\text  and \, }
  r > r_0.$$
\end{theorem}

A finitude theorem concerning the prolongation of exterior differential systems
proved by Kuranish opened a way to
modern theory of infinite Lie pseudo-groups. The finitude theorem in this form was
proved in the case of depth 1 by Singer-Sternberg. The general case is due to the second-named author.


\section{Geometric structures of higher order and universal frame bundles} \label{sect: geometric structures}
 In this section we introduce a category of
higher order geometric structures which represents the most general geometric structures of constant symbols
and which will turn out to be well adapted to the step prolongations to be studied in the next section.

\subsection{Geometric structures of order $k$} $\,$ \label{sect:geometric structure of order k}





\begin{definition} \label{def:universal frame bundle}
A {\it geometric structure  ${\bf Q}^{(k)}$ of order $k$}
and its {\it universal frame bundle $\mathscr S^{(k+1)}{\bf Q}^{(k)}$ of order $k+1$} are defined inductively
for $k \geq -1$  by the following properties.

\begin{enumerate}


\item 
  A geometric structure ${\bf Q}^{(k)}$
 of order $k \geq -1$ on a filtered manifold $(M,F)$
 (of type $( \frak g_-,G_0,\cdots, G_k)$)
 is a step-wise
principal fiber bundle over $(M,F)$ of type $\frak g_-$
$$
Q^{(k)} \stackrel{G_k}{\rightarrow} Q^{(k-1)}
\rightarrow \cdots \rightarrow
Q^{(0)} \stackrel{G_0}{\rightarrow} Q^{(-1)}=M,
$$ (that is, each
$
Q^{(i)} \stackrel{G_i}{\rightarrow} Q^{(i-1)}
$ is a principal fiber bundle with structure group
$G_i$ for $0 \leq i \leq k$) satisfying:

\begin{enumerate}
\item                      

 A geometric structure ${\bf Q}^{(-1)}$ of order $-1$ on $(M,F)$
 is the filtered manifold $(M,F)$ itself.

\item 
If $k\geq0$, the truncated sequence ${\bf Q}^{(i-1)}=(Q^{(i-1)} \rightarrow \dots Q^{(0)} \rightarrow Q^{(-1)}=M)$ is a geometric structure of order $ i-1$
for $i=k$  (and then consequently for $0 \leq i \leq k$ by induction).

\item 
 If $k\geq0$, $Q^{(i)} \rightarrow Q^{(i-1)}$ is a principal subbundle of the universal frame bundle
 $\mathscr S^{(i)}{\bf Q}^{(i-1)}\rightarrow Q^{(i-1)} $ of $ {\bf Q}^{(i-1)}$ of order $i$
for $i=k$  and then consequently for
$0 \leq i \leq k$ by induction).

\end{enumerate}

\item
To every geometric structure ${\bf Q}^{(k)}$ of order $ k \geq -1$ there is associated a principal fiber bundle
$\widehat{\mathscr S}^{(k+1)}{\bf Q}^{(k)}
\stackrel{\widehat  {\bf G}_{k+1}  }{\longrightarrow} Q^{(k)}$
 which is called the universal linear frame bundle of ${\bf Q}^{(k)}$ of order $k+1$. Its structure group ${\widehat  {\bf G} _{k+1} }$ is endowed with a filtration $\{F^{\ell}\widehat {\bf G}_{k+1} \}$.
The quotient
\begin{eqnarray*}
{\mathscr S}^{(k+1)}{\bf Q}^{(k)}: =\widehat{\mathscr S}^{\,(k+1)}{\bf Q}^{(k)} /F^{k+2}\widehat {\bf G}_{k+1}
\end{eqnarray*}
  is a principal fiber bundle over $Q^k$ with structure group
\begin{eqnarray*}
{\bf G}_{k+1} : =\widehat {\bf G}_{k+1}   /F^{k+2}\widehat {\bf G}_{k+1},
\end{eqnarray*}
which is regarded as a geometric structure of order $k+1$ and called the universal frame bundle of ${\bf Q}^{(k)}$ of order $ k+1 $.

  \item
The  universal linear frame bundle $\widehat{\mathscr S}^{(k+1)}{\bf Q}^{(k)}$ for a geometric structure ${\bf Q}^{(k)}$ of order $k \geq -1$ is defined as follows:
\begin{enumerate}
\item
The bundle $\widehat{\mathscr S}^{(k+1)}{\bf Q}^{(k)}$.
 We set
    $$               E^{(-1)} :=\frak g_- \text{ and }
                                        E  ^{(k)} :=\frak g_- \oplus \frak g_0 \oplus \dots \oplus \frak g_k \text{ if }k\geq 0 , $$
  where $ \frak g_i$ is the Lie algebra of the structure
  group $G_i$ for $0 \leq i \leq k$.
We regard $ E^{(k)} $ not only as a graded vector space but also as a filtered vector space in the standard way.
 We note also that the tangent space $T_{z^i}Q^{(i)}\ (-1\leq i\leq k)$ has a canonical filtration
 $\{F^{\ell}T_{z^i}Q^{(i)}\}_{\ell\in\mathbf Z}$ defined inductively
by the following conditions for $0 \leq i\leq k$:
$
F^{i+1}T_{z^i}Q^{(i)}=0$  and
\begin{eqnarray*}
0 \rightarrow F^{\ell}T_{z^{i}}
Q^{(i)}
\rightarrow T_{z^{i}}
Q^{(i)}
\rightarrow T_{z^{i-1}}
Q^{(i-1)}/ F^{\ell}T_{z^{i-1}}
Q^{(i-1)}\rightarrow 0
 \quad \text {(exact for }
 \ell \leq i\text{)},
\end{eqnarray*}
 where $z^{i-1} \in Q^{(i-1)}$ is the projection of $z^{k} \in Q^{(k)}$. \\



 It being prepared, we define
   the fiber $\widehat{\mathscr S}^{(k+1)}_{z^k}{\bf Q}^{(k)}$ over
    $z^k \in Q^{(k)}$ to be  the set of all filtration preserving isomorphisms
$$\zeta^{k+1}: E ^{(k)} \rightarrow T_{z^k}Q^{(k)}$$
which satisfy the following conditions:

If $k=-1$, $ [\zeta ^0] = \gr \zeta^0:\frak g_- \to \gr T_{z^{(-1)}} Q^{(-1)}$ is a graded Lie algebra isomorphism.

If $k\geq 0$, the following diagram is commutative:
\begin{eqnarray*}
 \xymatrix{
 0 \ar[r] & \frak g_k \ar[d]^{\widetilde{\cdot}} \ar[r] &  E ^{(k)} \ar[d]^{\zeta^{k+1}} \ar[r] &   E ^{(k-1)} \ar[d]^{\zeta^k} \ar[r] &0 \\
 0 \ar[r] & (\widetilde{\frak g_k})_{z^k}   \ar[r] &T_{z^k}Q^{(k)}   \ar[r] & T_{z^{k-1}}Q^{(k-1)} \ar[r] &0
 }
 \end{eqnarray*}
 where
  $\widetilde{\cdot}$ denote the map which sends $A \in \frak g_k$ to $\frac{d}{dt}|_{t=0}(z^k \exp \, tA) \in (\widetilde{\frak g_k})_{z^k}$, and $\zeta^k$ is the truncation   of $\zeta^{k+1}$, and
  $$z^{k } =[\zeta^k] =\zeta^k/F^{k+1}GL(  E ^{(k-1)}) $$
and $z^{k-1}$ is the image of $z^k$ by the projection map $Q^{(k)} \rightarrow Q^{(k-1)}$.




\item
The group $\widehat{\bf G}_{k+1}$ consists of all filtration preserving linear isomorphisms $\alpha^{k+1}:  E ^{(k)} \rightarrow   E ^{(k)}$ which satisfy the following conditions:

If $k=-1$, $ [\alpha ^0] =\gr \alpha^0:\frak g_- \to \frak g_-$ is a graded Lie algebra isomorphism.

If $k\geq 0$, the following diagram is commutative:

\begin{eqnarray*}
 \xymatrix{
 0 \ar[r] & \frak g_k \ar[d]^{=} \ar[r] &  E ^{(k)} \ar[d]^{\alpha^{k+1}} \ar[r] &  E ^{(k-1)} \ar[d]^{\alpha^k} \ar[r] &0 \\
 0 \ar[r] & \frak g_k   \ar[r] &  E ^{(k)}   \ar[r] &  E ^{(k-1)} \ar[r] &0
 }
 \end{eqnarray*}
and $[\alpha^k] =\alpha^k/F^kGL( E ^{(k-1)})=1$.
%
The group $\widehat {\bf G}_{k+1}$ is endowed  with a natural filtration induced from that of $GL(  E ^{(k)})$. 

\item 
The action of $\widehat{\bf G}_{k+1}$ on $\widehat{\mathscr S}^{(k+1)}{\bf Q}^{(k)}$.
For $\alpha^{k+1} \in \widehat{\bf G}_{ k+1 }$ and $\zeta^{k+1} \in \widehat{\mathscr S}^{(k+1)}_{z^k}{\bf Q}^{(k)}$,
 $\zeta^{k+1}\cdot \alpha^{k+1}$ is given by the following commutative diagram:
\begin{eqnarray*}
 \xymatrix{
  E ^{(k)} \ar[d]^{\alpha^{k+1}}\ar[r]^{\zeta^{k+1}\cdot \alpha^{k+1}} & T_{z^k }Q^{(k)} \\
   E ^{(k)} \ar[r]^{\zeta^{k+1}}& T_{z^{k}}Q^{(k)} \ar[u]_{identity}
 }
 \end{eqnarray*}
where $z^{k} =[\zeta^{k}]$ is the projection  of $\zeta^{k+1}$   to $Q^{(k)}$. 
\end{enumerate}

 \end{enumerate}
\end{definition}

This completes the definition. Indeed, the three properties (1) (2) (3) above being bound each to each, the inductive definition is well carried out in the order
  of:
   $${\bf Q}^{(-1)},\widehat{\mathscr S}^{(0)}{\bf Q}^{(-1)},\mathscr S^{(0)}{\bf Q}^{(-1)},
{\bf Q}^{(0)},\widehat{\mathscr S}^{(1)}{\bf Q}^{(0)},\mathscr S^{(1)}{\bf Q}^{(0)},
{\bf Q}^{(1)},\cdots. $$


 We remark that for a geometric structure of order $-1$, ${\bf Q}^{(-1)} =(M,F) $, the universal frame bundle $\mathscr S^{(0)}
 {\bf Q}^{(-1)} $ is clearly identified with the (reduced) frame bundle $\mathscr S^{(0)} (M,F)$ of $(M,F)$ introduced in Section 1, so that the geometric structures
 of order 0 are the $G$-structures in the sense of nilpotent geometry.
 We remark also that
 $\widehat{\mathscr S}^{(k+1)}{\bf Q}^{(k)} =
\mathscr S^{(k+1)}{\bf Q}^{(k)}$
if the filtration of the filtered manifold $Q^{(-1)}$ is trivial.  \\

 In  Definition \ref{def:universal frame bundle}, the filtered vector space $   E ^{(k)}$ and the filtered group $  {\bf G}_{k+1}$ depend on the graded Lie algebra $\frak g_-$ and  the sequence of structure groups $ (G_0, \dots, G_k)$ of ${\bf Q}^{(k)}$, and may be written more precisely as
 $$  E ^{(k)}(\frak g_-, \frak g_0, \dots, \frak g_k) \text{  and  } {\bf G}_{k+1}(\frak g_-, G_0, \dots, G_k)   .$$

Let us see more concretely what ${\bf G}_{k+1}(\frak g_-, G_0, \dots, G_k)$ is. Recall that $G_0(\frak g_-)$ is the group of automorphisms of $\frak g_-$ as graded Lie algebra, and $\frak g_0(\frak g_-)$ is the Lie algebra of all derivation of $\frak g_-$ preserving the degree. If we take a Lie subgroup $G_0 \subset G_0(\frak g_-)$, we see that the Lie algebra $\frak g_1(\frak g_-,\frak g_0)$ of ${\bf G}_1(\frak g_-, \frak g_0)$ is abelian and can be identified as:
$$\frak g_1(\frak g_-,\frak g_0) = \oplus_{p<0}\Hom(\frak g_p, \frak g_{p+1})  $$
and ${\bf G}_1(\frak g_-, G_0) =\exp \frak g_1(\frak g_-,\frak g_0)$, where we view an element of $\frak g_1(\frak g_-, \frak g_0)$ and ${\bf G}_1(\frak g_-, G_0)$ as a matrix in $\Hom(\frak g_- \oplus \frak g_0, \frak g_-\oplus \frak g_0)$ modulo the filtration $F^2\Hom(\frak g_- \oplus \frak g_0, \frak g_- \oplus \frak g_0)$ and $\exp$ denote the usual exponential map of matrix.

Taking successively subgroups
$$G_i \subset {\bf G}_i(\frak g_-, G_0, \dots, G_{i-1}) \text{ for } i=0,1, \dots, k$$
we get a ``symbol" $(\frak g_-, G_0, \dots, G_k)$ of a geometric structure of order $k$. For $k \geq 1$, the Lie algebra $\frak g_{k+1}(\frak g_-, \frak g_0, \dots, \frak g_{k })$ of ${\bf G}_{k+1}(\frak g_-, G_0, \dots, G_k)$ can be identifies as:
$$\frak g_{k+1}(\frak g_-, \frak g_0, \dots, \frak g_{k }) =\left(\oplus_{p<0} \Hom(\frak g_p, \frak g_{p+k+1}) \right) \oplus \left(\oplus_{i=0}^{k-1}\Hom(\frak g_i, \frak g_{k } )\right) $$
and ${\bf G}_{k+1}(\frak g_-, G_0, \dots, G_k)$ is ......
Then $\frak g_{k+1}(\frak g_-, \frak g_0, \dots, \frak g_k)$ and ${\bf G}_{k+1}(\frak g_-, G_0, \dots, G_k)$ can be represented in matrix as illustrated below (Figure \ref{G3} for $\frak g_- =\frak g_{-2} \oplus \frak g_{-1}$ and $k=2$).

\begin{figure}[h]
\begin{eqnarray*}
\frak g_3(\frak g_-, \frak g_0,\frak g_1, \frak g_2) \qquad \qquad \quad \quad &\qquad& {\bf G}_3(\frak g_-, G_0, G_1, G_2) \\[10 pt]
\begin{array}{c|c|c|c|c|c}
 & \frak g_{-2} & \frak g_{-1} & \frak g_0 & \frak g_1 &\frak g_2 \\ \hline
 \frak g_{-2} &&&&& \\ \hline
 \frak g_{-1} &&&&&  \\ \hline
 \frak g_0 &&&&&  \\ \hline
 \frak g_1 &\alpha^{-2}_1 &&&& \\ \hline
 \frak g_2 & & \alpha^{-1}_2&\beta^0_2 & \beta^1_2 & \\ \hline
\end{array}
&\qquad&
\begin{array}{c|c|c|c|c|c}
& \frak g_{-2} & \frak g_{-1} & \frak g_0 & \frak g_1 &\frak g_2 \\ \hline
 \frak g_{-2} &{\rm Id}&&&& \\ \hline
 \frak g_{-1} && {\rm Id}&&&  \\ \hline
 \frak g_0 &&& {\rm Id}&&  \\ \hline
 \frak g_1 &\alpha^{-2}_1 &&& {\rm Id}& \\ \hline
 \frak g_2 & & \alpha^{-1}_2&\beta^0 _2 & \beta^1_2 & {\rm Id} \\ \hline
\end{array}
\end{eqnarray*}
\smallskip
\noindent
where $\alpha = \alpha^{-2}_1 + \alpha^{-1}_2 \in \oplus_{p<0} \Hom(\frak g_p, \frak g_{p+k+1})$ and $\beta =\beta^0_2 + \beta^1_2 \in \oplus_{i=0}^{k-1}\Hom(\frak g_i, \frak g_{k } )$.

\caption{$\frak g_3(\frak g_-, \frak g_0,\frak g_1, \frak g_2)$ and ${\bf G}_3(\frak g_-, G_0, G_1, G_2)$} \label{G3}
\end{figure}

\medskip
\noindent{\bf  Notations.}
  For simplicity of notation,  we will write a geometric structure ${\bf Q}^{(k)}=(Q^{(k)} \rightarrow Q^{(k-1)} \rightarrow \dots \rightarrow Q^{(-1)})$ (respectively, $  E ^{(k)}(\frak g_-, \frak g_0, \dots, \frak g_k)$ and  ${\bf G}_{k+1}(\frak g_-, G_0, \dots, G_k)$)  simply as  $Q^{(k)}\stackrel{G_k}{\longrightarrow} Q^{(k-1)}$ (respectively,  $E^k(G_k)$ and $G_{k+1}(G_k)$)    or  $Q^{(k)}$ (respectively, $E^k$ and $G_{k+1}$)     when no confusion can arise.
  The Lie algebra $\frak g_{k+1}(\frak g_-, \frak g_0, \dots, \frak g_k)$  of ${\bf G}_{k+1}(\frak g_-, G_0, \dots, G_k)$ will be also written as  $\frak g_{k+1}(\frak g_k)$  or $\frak g_{k+1}$.

  Similarly, we will write $\widehat{\mathscr S}^{(k+1)}{\bf Q}^{(k)}$ and $\mathscr S ^{(k+1)}{\bf Q}^{(k)}$ simply as $\widehat{\mathscr S}^{(k+1)}{ Q}^{(k)}$ and $\mathscr S ^{(k+1)}{ Q}^{(k)}$ when no confusion can arise.  \\
%


\subsection{Completed universal frame bundles} $\,$

For a geometric structure $Q^{(k)} \stackrel{G_k}{\longrightarrow} Q^{(k-1)}$ we define $\mathscr S^{(\ell+1)}Q^{(k)}$ and $G_{\ell+1}(G_k)$ for $\ell \geq k$ inductively as
$$\mathscr S^{(\ell+1)}Q^{(k)} = \mathscr S^{(\ell+1)}\mathscr S^{(\ell  )}Q^{(k)} \text{ and } G_{\ell+1}(G_k) = G_{\ell+1}(G_{\ell }(G_k)).$$
Then $\mathscr S^{(\ell+1)}Q^{(k)} \stackrel{G_{\ell+1}(G_k)}{\longrightarrow} \mathscr S^{(\ell )}Q^{(k)} \longrightarrow \dots \longrightarrow \mathscr S^{(k+1)}Q^{(k)} \longrightarrow Q^{(k)} \longrightarrow \dots \longrightarrow Q^{(-1)}$
is a geometric structure of order $\ell$ for $\ell \geq k$.
In this subsection we shall show that
$\mathscr S ^{(\ell+1)}Q^{(k)} \rightarrow  \dots \rightarrow Q^{(k)} \rightarrow  Q^{(k-1)}$
is not only a step-wise principal fiber bundle but also a principal fiber bundle over $Q^{(k-1)}$ for $\ell \geq k$.

Note that for $\ell \geq k+1$, the Lie algebra $ \frak g _{\ell}(\frak g_k) $ of $G_{\ell}(G_k)$ is inductively given by
$$ \frak g _{\ell}(\frak g_k)  = \Hom \left(\frak g_-, \frak g_- \oplus \left( \oplus_{i=0}^{\ell-1} \overline{\frak g}_i\right)\right)_{\ell} \oplus \Hom\left( \oplus_{i=0}^{i = \ell-2} \overline{\frak g}_i, \overline{\frak g}_{\ell-1} \right).$$
where  $\overline{\frak g}_i $ denotes $ \frak g_i$ for $0 \leq i \leq k$ and $\frak g_i(\frak g_k)$ for $k+1 \leq i \leq \ell-1$.
Set $$ E^{(\ell)}(\frak g_k):=\frak g_- \oplus \frak g_0 \oplus \dots \oplus \frak g_k \oplus \frak g_{k+1}(\frak g_k) \oplus \dots \oplus \frak g_{\ell}(\frak g_k). $$
%

First, let us define the structure group $G^{(\ell+1)}(G_k)$ of $\mathscr S ^{(\ell+1)}Q^{(k)} \rightarrow     Q^{(k-1)}$
  with an exact sequence
$$ 1 \rightarrow G_{\ell+1}(G_k) \rightarrow G^{(\ell+1)}(G_k) \rightarrow G^{(\ell)}(G_k) \rightarrow 1$$
  by induction on $\ell \geq k$.
 We set $G^{(k)}(G_k) =G_k$. Let $\ell \geq k$ and suppose that we have defined the groups  $G^{(i)}(G_k)$ ($k \leq i \leq \ell $)
  with an exact sequence
$$ 1 \rightarrow G_{i}(G_k) \rightarrow G^{(i)}(G_k) \rightarrow G^{(i-1)}(G_k) \rightarrow 1$$
Define  $\widehat{G}^{(\ell+1)}(G_k)$   by  the group of all filtration preserving isomorphisms $\alpha^{\ell+1}:E^{(\ell)}(\frak g_k) \rightarrow E^{(\ell)}(\frak g_k)$ which makes the following diagram commutative:
\begin{eqnarray*}
 \xymatrix{
 0 \ar[r] &  \frak g _{\ell}(\frak g_k) \ar[d]^{Ad(a^{\ell})} \ar[r] &E^{(\ell)}(\frak g_k) \ar[d]^{\alpha^{\ell+1}} \ar[r] & E^{(\ell-1)} (\frak g_k) \ar[d]^{\alpha^{\ell}} \ar[r] &0 \\
 0 \ar[r] &  \frak g _{\ell}(\frak g_k)   \ar[r] &E^{(\ell)} (\frak g_k)  \ar[r] & E^{(\ell-1)} (\frak g_k) \ar[r] &0
 }
 \end{eqnarray*}
with $a^{\ell}=[\alpha^{\ell}] \in G^{(\ell)}(G_k)$.
Here,  $\alpha^{\ell} $ is the induced map from $\alpha^{\ell+1}$ and $ [\alpha^{\ell}]$ denotes $\alpha^{\ell}/F^{\ell+1}GL(E^{(\ell-1)})$.
%
 %
%
 %
Note also   that by the inductive assumption we have the following exact sequence:
$$1 \rightarrow G_{\ell}(G_k) \rightarrow G^{(\ell)}(G_k) \rightarrow G^{(\ell-1)}(G_k) \rightarrow 1$$
so that $a^{\ell} \in G^{(\ell)}(G_k)$ acts on $\frak g_{\ell}(\frak g_k)$, denote it by ${\rm Ad}(a^{\ell})$.
The group $G^{(\ell+1)}(G_k)$  is defined by the quotient $$G^{(\ell+1)}(G_k) =\widehat{G}^{(\ell+1)}(G_k)/F^{\ell+2}GL(E^{(\ell)}).$$
  %
Then we get an exact sequence
 $$1 \rightarrow G_{\ell+1}(G_k) \rightarrow G^{(\ell+1)}(G_k) \rightarrow G^{(\ell )}(G_k) \rightarrow 1.$$
 This completes the inductive definition of $G^{(\ell)}(G_k)$.

  Next, let us show that the group $G^{(\ell)}(G_k)$ acts on $\mathscr S^{(\ell+1)}Q^{(k)}$. Supposing that the action of $G^{(\ell)}(G_k)$ on $\mathscr S^{(\ell)}Q^{(k)}$ is defined, we will show that we can define it for $\ell+1$.  For $\alpha^{\ell+1} \in \widehat{G}^{(\ell+1)}(G_k)$ and $\zeta^{\ell+1} \in \widehat{\mathscr S}^{(\ell+1)}Q^{(k)}$ we define the right action $\zeta^{\ell+1}\cdot \alpha^{\ell+1}$ by the following commutative diagram:
\begin{eqnarray*}
 \xymatrix{
 E^{(\ell)} \ar[d]^{\alpha^{\ell+1}}\ar[r]^{\zeta^{\ell+1}\cdot \alpha^{\ell+1}} & T_{z^{\ell}\cdot a^{\ell}}Q^{(\ell)} \\
 E^{(\ell)} \ar[r]^{\zeta^{\ell+1}}& T_{z^{\ell}}Q^{(\ell)} \ar[u]_{{R_{a^{\ell}}}_*}
 }
 \end{eqnarray*}
where $z^{\ell} =[\zeta^{\ell}]$ and $a^{\ell} =[\alpha^{\ell}]$ are projections of $\zeta^{\ell+1}$ and $\alpha^{\ell+1}$ to $\mathscr S^{(\ell)}Q^{(k)}$ and $G^{(\ell)}(G_k)$ respectively, and $z^{\ell}\cdot a^{\ell}$ is the associated action of $G^{(\ell)}(G_k)$ on $\mathscr S^{(\ell)}Q^{(k)}$.

  Since $G^{(\ell)}(G_k)$ acts on $\mathscr S^{(\ell)}Q^{(k)}$, we have
$$(R_{a^{\ell}})_* \widetilde{A}_{z^{\ell}} = \widetilde{{\rm Ad}(a^{\ell})^{-1} A}_{z^{\ell}\cdot a^{\ell}}$$
for $A \in \frak g^{(\ell)}(\frak g_k)$. Therefore, we see that $\zeta^{\ell+1}\cdot \alpha^{\ell+1} $ belongs to $\widehat{\mathscr S}^{(\ell+1)}Q^{(k)}$. It then follows that $\widehat{G}^{(\ell+1)}(G_k)$ acts on $\widehat{\mathscr S}^{(\ell+1)}$. Passing to the quotients, we see that $G^{(\ell+1)}(G_k)$ acts on $\mathscr S^{(\ell+1)}Q^{(k)}$.

 Hence we have proven the following:
\begin{proposition}
For a geometric structure $Q^{(k)} \stackrel{G_k}{\longrightarrow} Q^{(k-1)}$,
$$\mathscr S^{(\ell+1)}Q^{(k)} \rightarrow Q^{(k-1)}$$
is a principal fiber bundle over $Q^{(k-1)}$ with structure group $G^{(\ell+1)}(G_k)$.

\end{proposition}

\begin{figure}[t]
 \begin{eqnarray*}
 \xymatrix{
    &  & && &  \mathscr S(Q^{(k)})\ar@/^3pc/[dddddd]^{G(G_k)} \\
    &  & && &  \\
    &  & && &   \vdots \ar[d]_{G_{k+3}(G_k)} \\
    &  & && &  \mathscr S^{(k+2)}(Q^{(k)})   \ar[d]_{G_{k+2}(G_k)} \\
    &  & && &  \mathscr S^{(k+1)}(Q^{(k)}) \ar[d]_{G_{k+1}(G_k)}\ar@/^2pc/[dd]^{G^{(k+1)}(G_k)} \\
    &  & && \mathscr S^{(k)} (Q^{(k-1)}) \ar[d]_{G_{k }(G_{k-1}) } & \ar@{_{(}->}[l] Q^{(k) } \ar[d]_{G_{k }} \\
    &  &   & &  Q^{(k-1)}\ar[d]_{G_{k-1}} & \ar@{=}[l]  Q^{(k-1)}     \\
    &  & & \vdots &   & \\
    &   &\mathscr S^{(2)}(Q^{(1)}) \ar[d]_{G_{2}(G_1)}  && &  \\
    & \mathscr S^{(1)} (Q^{(0)}) \ar[d]_{G_{1}(G_0) }    & \ar@{_{(}->}[l] Q^{(1)} \ar[d]_{G_1} && & \\
   \mathscr S^{(0)}(M) \ar[d]   & \ar@{_{(}->}[l] Q_0\ar[d]_{G_0} & Q_0  \ar@{=}[l] && &  \\
 M  &M \ar@{=}[l]   & && & .
}
 \end{eqnarray*}
\caption{Completed universal frame bundles} \label{fig complete}
\end{figure}
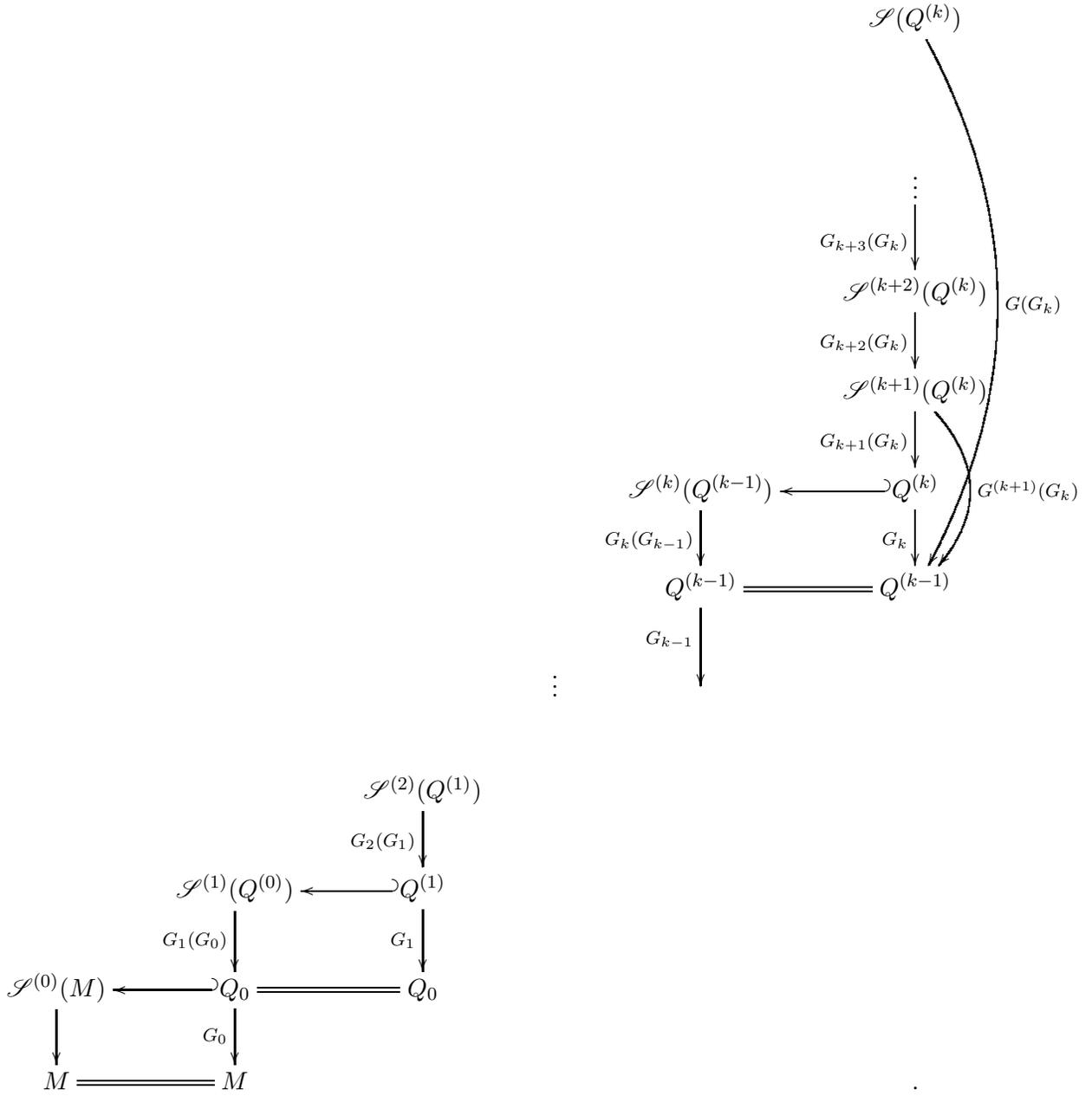
 \vskip 10 pt

\begin{definition}
Passing to the projective limit, we set
$$\mathscr S Q^{(k)}:=\lim_{\leftarrow_{\ell}}\mathscr S^{(\ell)}Q^{(k)} \text{ and } G(G_k):=\lim_{\leftarrow_{\ell}} G^{(\ell)}(G_k) \text{ and } E(\frak g_k):=\lim_{\leftarrow_{\ell}} E^{(\ell)}(\frak g_k) .$$
Then $\mathscr SQ^{(k)} \rightarrow Q^{(k-1)}$ is a principal bundle with structure group $G(G_k)$. We call $\mathscr SQ^{(k)}$ the {\it completed universal frame bundle of} $Q^{(k)}$ and $\mathscr S^{(\ell)}Q^{(k)}$ the {\it universal frame bundle of} $Q^{(k)}$ {\it of order} $\ell$ (Figure \ref{fig complete}).
\end{definition}

We remark that the completed frame bundle $\mathscr SQ^{(k)}$ is of infinite dimension.
However, since it is the projective limit of a sequence of finite dimensional  manifolds, we can deal with it almost equally
as in the finite dimensional case.
In this paper we will freely use standard terminologies concerning   finite dimensional manifolds also for these infinite dimensional objects, such as Lie groups, fiber bundles, tangent spaces, smooth or analytic differential forms, absolute parallelism etc.
The reader may refer to Section 2.1 (p.275--p.279) of \cite{M93}  for a brief summary of this convention.

\subsection{Canonical Pfaff class and canonical Pfaff form}  $\,$

\subsubsection{ }

Here we give the definition of an isomorphism of geometric structures.

 \begin{definition} \label{def:isomorphism of geometric structure of order k}
 An isomorphism of geometric structures is defined inductively by the following conditions:
 \begin{enumerate}
 \item An isomorphism $\Phi:Q^{(-1)}(=(M,F))\rightarrow Q'^{(-1)}(=(M',F'))$ of geometric structures
 of order $-1$ is an isomorphism of filtered manifolds.
 \item  If $ Q ^{(k)} $
and $ Q'^{(k)} $ are geometric structures of order $ k \geq -1$ of same type $ (\frak g_-,G_0,\cdots G_k )        $ and
 if $\Phi:Q^{(k)}\to Q'^{(k)}$
  is an isomorphism of the geometric structures,
 then there is induced a bundle isomorphism
 $ \mathscr S^{(k+1)} \Phi:  \mathscr S^{(k+1)}Q^{(k)}\to\mathscr S^{(k+1)}Q'^{(k)}$ which lifts
  $\Phi $.
  \item Let $ Q^{(k+1)} $ and $ Q'^{(k+1)} $ be geometric structures of order $ k+1\geq 0 $ of type $(\frak g_-,G_0,\cdots,G_{k+1})$.
  A bijection $ \Phi ^{(k+1)}:Q^{(k+1)} \to Q'{(k+1)} $ is an isomorphism of the geometric structures
  if and only if there exists an isomorphism
   $ \Phi ^{(k)}:Q^{(k)} \to Q'{(k)} $ such that the lift $\mathscr S^{( k+1)} \Phi ^{(k)}$ maps  $ Q^{(k+1)}
   $ to $ Q'^{(k+1)}$
    and the restriction of  $\mathscr S^{( k+1)} \Phi ^{(k)}$ to $ Q^{(k+1)}$ coincides with $\Phi ^{(k+1)}$.

    \end{enumerate}

   \end{definition}


 To complete the definition, we just indicate how to define the canonical lift
$\mathscr S^{(k+1)}\Phi^{(k)} : \mathscr S^{(k+1)}Q^{(k)} \rightarrow \mathscr S^{(k+1)} Q'^{(k)} $
for an isomorphism
$\Phi^{(k)} : Q^{(k)} \rightarrow Q'^{(k)} $:
For $ \zeta^{(k+1)}\in\widehat{\mathscr S}^{(k+1)}Q^{(k)} $, setting
$$ (\widehat {\mathscr S}^{(k+1)}\Phi^{(k)} )(\zeta^{(k+1)})=\Phi^{(k)}_*\circ\zeta^{(k+1)}.
$$
we get  a well-defined map
$\widehat {\mathscr S}^{(k+1)}\Phi^{(k)} : \widehat {\mathscr S}^{(k+1)}Q^{(k)} \rightarrow \widehat {\mathscr S}^{(k+1)} Q'^{(k)} $.
Then, passing to the quotient, we have the canonical lift
$\mathscr S^{(k+1)}\Phi^{(k)} : \mathscr S^{(k+1)}Q^{(k)} \rightarrow \mathscr S^{(k+1)} Q'^{(k)} $. \\


   Now we proceed to the definition of the canonical Pfaff class.

  The  {\it canonical Pfaff class} $[\theta^{(k-1)}]$ of a geometric structure $Q^{(k)} \rightarrow Q^{(k-1)}$  is the equivalence class of the Pfaff form under the action of the group $F^{k+1}GL(E^{(k-1)})$. Recall that  for $z^k \in Q^{(k)}$, there is a filtration preserving isomorphism $\zeta^k: E^{(k-1)} \rightarrow T_{z^{k-1}}Q^{(k-1)}$ with $z^k=\zeta^k/F^{k+1}GL(E^{(k-1)})$, where $z^{k-1}$ is the image of $z^k$ by the projection map $\pi:Q^{(k)} \rightarrow Q^{(k)}$. The composition
$$  \theta^{(k-1)}: T_{z^k}Q^{(k)} \stackrel{\pi_*}{\longrightarrow} T_{z^{k-1}}Q^{(k-1)} \stackrel{(\zeta^k)^{-1}}{\longrightarrow} E^{(k-1)}, 
$$
depends on a choice of $\zeta^k $. The canonical Pfaff class $[\theta^{(k-1)}]$ at $z^k$ is the equivalent class of $\theta^{(k-1)}$ by the action of $F^{k+1}GL(E^{(k-1)})$.

\begin{proposition} \label{prop:isomorphism via canonical class}
Let $Q^{(k)}$ and $\overline{Q}^{(k)}$ be geometric structures of type $(\frak g_-; G_0, \dots, G_k)$. If $\Phi:Q^{(k)} \rightarrow \overline{Q}^{(k)}$ is an isomorphism, then $\Phi$ preserves the canonical Pfaff classes $[\theta^{(k-1)}]$ and $[\overline{\theta}^{(k-1)}]$ of $Q^{(k)}$ and $\overline{Q}^{(k)}$, respectively, that is $\Phi^*[\overline{\theta}^{(k-1)}] =[\theta^{(k-1)}]$. Conversely, if $G_0, \dots, G_k$ are connected, then a diffeomorphism $\Phi:Q^{(k)} \rightarrow \overline{Q}^{(k)}$ such that $\Phi^*[\overline{\theta}^{(k-1)}] =[\theta^{(k-1)}]$ is an isomorphism of the geometric structures.
\end{proposition}

\begin{proof}
The first statement is clear, i.e., if $\Phi :Q^{(k)} \rightarrow \overline{Q}^{(k)}$ is an isomorphism, then $\Phi $ preserves the canonical classes.
For the second statement, we will   use the induction on $\ell$ ($0 \leq \ell \leq k$) and show  that
if $\Phi^{(\ell)}: Q^{(\ell)} \rightarrow \overline{Q}^{(\ell)}$ satisfies ${\Phi^{(\ell)}}^* [\overline{\theta}^{(\ell-1)}] = [\theta^{(\ell-1)} ]$, then $\Phi^{(\ell)}$ is an isomorphism.

First, let us show that it holds for $\ell=0$. In fact, since the fibers of $\overline{Q}^{(0)} \rightarrow \overline{Q}^{(-1)}$ are leaves of $\overline{\theta}^{(-1)}=0$ and the fibers are connected, there exists $\Phi^{(-1)}: Q^{(-1)} \rightarrow \overline{Q}^{(-1)}$ such that the following diagram is commutative:
\begin{eqnarray*}
\xymatrix{
Q^{(0)} \ar[d]\ar[r]^{\Phi^{(0)}} & \overline{Q}^{(0)} \ar[d] \\
Q^{(-1)} \ar[r]^{\Phi^{(-1)}}  & \overline{Q}^{(-1)}.
}
\end{eqnarray*}
Let $z^{(0)} \in \overline{Q}^{(0)}$, which may be viewed as an isomorphism $z^{(0)}: \frak g_- \rightarrow  T_{x} \overline{Q}^{(-1)}$ of graded Lie algebras. We see that the inverse map $(z^{(0)})^{-1}$ can be identified with the canonical class $[\overline{\theta}^{(-1)} ]_{z^{(0)}}$. Then the assumption that ${\Phi^{(0)}}^*[\overline{\theta}^{(-1)} ] = [\theta^{(-1)} ]$ implies that $\Phi^{(-1)}$ is a filtration preserving map. Moreover, we see that the lift $\mathscr S^{(0)}\Phi^{(-1)}: \mathscr S^{(0)}Q^{(-1)} \rightarrow \mathscr S^{(0)} \overline{Q}^{(-1)}$ coincides with $\Phi^{(0)}$ on $Q^{(0)}$.

Now assuming the statement valid for $\ell-1$, we prove it for $\ell$. By the assumption ${\Phi^{(\ell)}}^*[\overline{\theta}^{(\ell-1)}] =[\theta ^{(\ell-1)}]$, we see again that there exists $\Phi^{(\ell-1)}$ which makes the following diagram commutative:
\begin{eqnarray*}
\xymatrix{
Q^{(\ell)} \ar[d]\ar[r]^{\Phi^{(\ell)}} & \overline{Q}^{(\ell)} \ar[d] \\
Q^{(\ell-1)} \ar[r]^{\Phi^{(\ell-1)}}  & \overline{Q}^{(\ell-1)}.
}
\end{eqnarray*}
Moreover, we see that ${\Phi^{(\ell-1)}}^*[\overline{\theta}^{(\ell-2)}] =[\theta ^{(\ell-2)}]$. Therefore, $\Phi^{(\ell-1)}$ is an isomorphism by induction assumption.

Now take $z^{(\ell)} \in Q^{(\ell)}$ and $\overline{z}^{(\ell)} \in \overline{Q}^{(\ell)}$ with $\overline{z}^{(\ell)} = \Phi^{(\ell)}(z^{(\ell)})$, and let $z^{(\ell-1)}$ and $\overline{z}^{(\ell-1)}$ be their projections on $Q^{(\ell-1)}$ and $\overline{Q}^{(\ell-1)}$, respectively. Consider the following commutative diagram:
\begin{eqnarray*}
\xymatrix{
T_{z^{(\ell-1)}}Q^{(\ell-1)}  \ar[r]^{\Phi^{(\ell-1)}_*} &T_{\overline{z}^{(\ell-1)}}\overline{Q}^{(\ell-1)} \\
E^{(\ell-1)} \ar[u]^{\eta^{(\ell)}} \ar@{=}[r] & E^{(\ell-1)} \ar[u]^{\zeta^{(\ell)}},
}
\end{eqnarray*}
where we choose an $\eta^{(\ell)}$ so that $[\eta^{(\ell)}] =z^{(\ell)}$ and define $\zeta^{(\ell)}$ by the commutative diagram. Then the assumption that ${\Phi^{(\ell)}}^*[\overline{\theta}^{(\ell-1)}] =[\theta ^{(\ell-1)}]$ implies that $[\zeta^{(\ell)}] = \overline{z}^{(\ell)}$. This means that $\overline{z}^{(\ell)} =  \mathscr S^{(\ell)} \Phi^{(\ell-1)} (z^{(\ell)})$. Hence the restriction of  $\mathscr S ^{(\ell)} \Phi^{(\ell-1)}$ to $Q^{(\ell)}$ is $ \Phi^{(\ell)}$, which completes the induction.
\end{proof}




\subsubsection{ }
Now  we will show that there is an $E(\frak g_k)$-valued  1-form $\theta$
canonically defined on the completed frame bundle $\mathscr S Q^{(k)}$  of $Q^{(k)} \stackrel{G_k}{\longrightarrow} Q^{(k-1)}$,   called the {\it canonical Pfaff form} on $\mathscr SQ^{(k)}$.

For each $\ell$ define the canonical Pfaff class $[\theta^{(\ell-1)}]$ of $\mathscr S^{(\ell)}Q^{(k)} \rightarrow \mathscr S^{(\ell-1)}Q^{(k)}$ in a similar way  as we define $[\theta^{(k-1)}]$.
For a tangent vector $X \in T_zQ$, we can assign a sequence $\{v^{(i)} \in E^{(i)}\}$ such that
$$\langle \pi_{E^{(i)}} \theta ^{(j)}, {\pi_{Q^{(j+1)}}}_*X \rangle = v^{(i)}$$
for any representative $\theta^{(j)}$ of the canonical Pfaff class $[\theta^{(j)}]$ of $Q^{(j+1)}$ for any $j$ large enough, where $\pi_{E^{(i)}}$, $\pi_{Q^{(j)}}$ denote the projections on $E^{(i)}$ and $Q^{(j)}$, respectively, and $\pi_{E^{(i)}}v^{(j)} = v^{(i)}$ for all $j \geq i$.
We then set
$$\langle \theta, X \rangle := \lim_i v^{(i)}$$
and thus define a Pfaff form $\theta$ on $\mathscr SQ^{(k)}$ taking values in $E$.

In fact, $\theta$ can be also defined in the following way.
The universal frame bundle   $\mathscr S^{(\ell)}Q^{(k)}$  of order $\ell$ has a tangential filtration regular of type $E^{(\ell)}(\frak g_k)$ and
  $\widehat{\mathscr S}^{(\ell+1)}Q^{(k)}$ is a principal subbundle of the  frame bundle of $\mathscr S^{(\ell)}Q^{(k)}$,  so that  $\widehat{\mathscr S}^{(\ell+1)}Q^{(k)}$ has an $E^{(\ell)}(\frak g_k)$-valued one form, the restriction of the canonical one-form of the  frame bundle on $\mathscr S^{(\ell)}Q^{(k)}$,  
  denoted by $\widehat{\theta}^{(\ell)}(\frak g_k)$ and called the {\it canonical form} of $\widehat{\mathscr S}^{(\ell+1)} Q^{(k)}$.  

For a given $\ell$,  there is a projection map $\mathscr S^{(\ell+m)}Q^{(k)} \rightarrow \widehat{\mathscr S}^{(\ell+1)}Q^{(k)}$ for sufficiently large $m$, so that we have a projection map $\mathscr S Q^{(k)}  \rightarrow \mathscr S^{(\ell+m)} Q^{(k)} \rightarrow \widehat{\mathscr S}^{(\ell+1)}Q^{(k)}$.
 Let $\theta^{(\ell)}(\frak g_k)$ be the pull-back of the canonical form on $\widehat{\mathscr S}^{(\ell+1)}Q^{(k)}$ under the projection map $\mathscr S Q^{(k)}\rightarrow \widehat{\mathscr S}^{(\ell+1)}Q^{(k)}$.

Passing to the projective limit, we set
 $$\theta(\frak g_k) :=\lim_{\leftarrow_{\ell}}  \theta^{(\ell)}(\frak g_k).
  $$
Then $\theta(\frak g_k)$ is an $E(\frak g_k)$-valued one-form on $\mathscr S Q^{(k)}$.

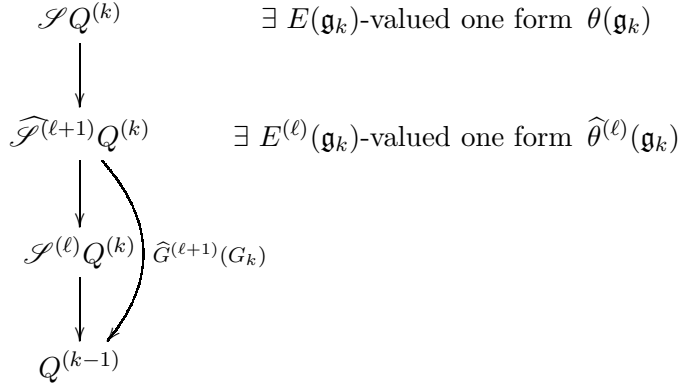
\begin{eqnarray*}
\xymatrix{
\mathscr S Q^{(k)} \ar[d] & \exists \,\, E (\frak g_k)\text{-valued one form }\,  \theta (\frak g_k)\\
\widehat{\mathscr S}^{(\ell+1)}Q^{(k)}  \ar[d] \ar@/^2pc/[dd]^{\widehat{G}^{(\ell+1)}(G_k)}&   \exists \,\, E^{(\ell)}(\frak g_k)\text{-valued one form }\, \widehat{\theta}^{(\ell)}(\frak g_k) \\
\mathscr S^{(\ell)}Q^{(k)} \ar[d]&\\
Q^{(k-1)}&
}
\end{eqnarray*}

\begin{proposition}   \label{prop:canonical forms} 
We use the notations $\mathscr S $, $\theta$, $E$,  $  \overline{G}$, for  $\mathscr SQ^{(k)} $, $\theta(\frak g_k)$, $E(\frak g_k)$,  $G(G_k)$, respectively.
Then the canonical Pfaff form $\theta$ on $\mathscr S$ satisfies the following properties.

\begin{enumerate}
\item $\theta:T_z \mathscr S  \rightarrow E $ is a filtration preserving isomorphism for every $z \in \mathscr S $.
\item $\gr \theta_z:\gr T_{z}M \rightarrow \frak g_-$ is a graded Lie algebra isomorphism for every $z \in M$
 and $\gr \theta_z: \gr V _z  \mathscr S   \rightarrow \gr F^0E  $ is the canonical isomorphism $\widetilde{\cdot}$ for every $z \in \mathscr S $, where $V  \mathscr S  $ denotes the vertical tangent space of $\mathscr S \rightarrow M$.
\item $R_a^* \theta = a^{-1}\theta$ for $a \in \overline{G} $.
\end{enumerate}

\end{proposition}

 \begin{proof}

By the same argument as in the proof of Theorem 2.3.1 of \cite{M93},
together with the properties:
 \begin{enumerate}
 \item [(i)] $\widehat{\theta}^{(\ell)}(\widetilde A) = A/F^{\ell+1}$ for $A \in \widehat{\frak g}^{(\ell+1)}(\frak g_k)$;
 \item [(ii)] $R_a^* \widehat{\theta}^{(\ell )} =a^{-1}\widehat{\theta}^{(\ell)}$ for $a \in  \widehat{G}^ {(\ell+1)}(G_k)$,
 \end{enumerate}
 we get the desired results. 
 \end{proof}

\begin{theorem} \label{thm:isomorphism via canonical class of any order}
Let $P^{(k)}$ and $Q^{(k)}$ be geometric structures of order $k$ of type $(\frak g_-,  G_0, G_1, \dots, G_k)$. Set $P^{(\ell)} = \mathscr S^{(\ell)}P^{(k)}$ and $Q^{(\ell)}=\mathscr S^{(\ell)}Q^{(k)}$ for $\ell \geq k+1$.
\begin{enumerate}

\item If there is an isomorphism $\varphi^{(k)}:P^{(k)} \rightarrow Q^{(k)}$ of geometric structures, then it induces
an isomorphism $\mathscr S^{(\ell)}\varphi^{(k)}: \mathscr S^{(\ell)}P^{(k)} \rightarrow \mathscr S^{(\ell)}Q^{(k)}$ of geometric structures, i.e. a diffeomorphism satisfying  $(\mathscr S^{(\ell)}\varphi^{(k)})^*[\theta^{(\ell-1)}]_Q =[\theta^{(\ell-1)}]_P$  for all $\ell =k+1, k+2, \dots, \infty$.
\item  Conversely, assuming that $G_0, G_1, \dots, G_k$ are connected, if 
 there is a diffeomorphism $\Phi^{(i)}: P^{(i)} \rightarrow Q^{(i)}$ for some $i \geq 0$ 
 such that  $(\Phi^{(i)})^* [\theta^{(i-1)}]_Q = [\theta^{(i-1)}]_P$,
then $\Phi^{(i)}$ induces an isomorphism $\varphi^{(i-1)}:P^{(i-1)} \rightarrow Q^{(i-1)}$  such that the canonical lift $\mathscr S^{(i)}\varphi^{(i-1)}: \mathscr S^{(i)}P^{(i-1)} \rightarrow \mathscr S^{(i)} Q^{(i-1)}$ sends $P^{(i)}$ onto $Q^{(i)}$ and its restriction $\mathscr S^{(i)}\varphi^{(i-1)}|_{P^{(i)}}$ to $P^{(i)}$ is $ \Phi^{(i)}$.
 \end{enumerate}

In particular, the equivalence problem of geometric structure $Q^{(k)}$ reduces to the equivalence problem of the absolute parallelism $(\mathscr S Q^{(k)}, \theta)$ on the complete  universal frame bundle $\mathscr S Q^{(k)}$.
\end{theorem}

\begin{proof} Use Proposition \ref{prop:isomorphism via canonical class}.
\end{proof}

\subsection{Semi-canonical embeddings}

\begin{proposition} [$\mathcal W$-canonical embedding]  \label{prop: embedding} Let $Q^{(k+1)} \stackrel{G_{k+1}}{\longrightarrow} Q^{(k)}$ be a geometric structure of order $k+1$. A choice of a direct sum decomposition
$$\frak g_{k+1}(\frak g_k) = \frak g_{k+1} \oplus \mathcal W $$
determines 
(filtration preserving) injective homomorphisms
$$G (G_{k+1}) \stackrel{\iota_{\mathcal W} }{\longrightarrow} G (G_k), \qquad E (G_{k+1}) \stackrel{\iota_{\mathcal W} }{\longrightarrow} E (G_k),  $$
and a principal fiber bundle embedding
\begin{eqnarray*}
 \xymatrix{
\mathscr S Q^{(k+1)} \ar[d] \ar[r]^{\iota_{\mathcal W} }   & \mathscr S Q^{(k)}\ar[d]\\
 Q^{(k)} \ar@{=}[r] & Q^{(k)}
 }
 \end{eqnarray*}
satisfying  $\iota_{\mathcal W}^* \theta_{\mathscr SQ^{(k)}} = \iota_{\mathcal W} \circ \theta_{\mathscr S Q^{(k+1)}}$.

\end{proposition}

\begin{proof} Fix a complementary subspace $\mathcal W$ such that $\frak g_{k+1}(\frak g_k) = \frak g_{k+1} \oplus \mathcal W $. Note that $\frak g_{k+1}(\frak g_k) = \Hom(\frak g_-, E^{(k)}(\frak g_k))_{k+1} \oplus (\oplus_{i=0}^{k-1} \Hom(\frak g_i, \frak g_k))$.

We claim that for each $\ell \geq k $ there is an embedding $\mathscr S^{(\ell+1)}  Q^{(k +1)}  \stackrel{{\iota^{(\ell+1)}}}{\longrightarrow} \mathscr S^{(\ell+1)} Q^{(k)}$
 such that  the following diagram
\begin{eqnarray*}
 \xymatrix{
      \mathscr S^{(\ell+1)}  Q^{(k +1)} \ar[d] \ar[r]^{\iota^{(\ell+1)}} &\mathscr S^{(\ell+1)}( Q^{(k)} \ar[d]
   \\
  	 \mathscr S^{(\ell)}  Q^{(k +1)}  \ar[r]^{\iota^{(\ell)}} & \mathscr S^{(\ell)} Q^{(k)}
}
 \end{eqnarray*}
is commutative, as well as   splittings:
$$\frak g_{\ell+1}(\frak g_{k }) = \frak g_{\ell+1}(\frak g_{k+1}) \oplus \mathcal W_{\ell+1} \text{ and } E^{(\ell+1)}(\frak g_k) = E^{(\ell+1)}(\frak g_{k+1}) \oplus \mathcal V_{\ell+1}, $$
where $\mathcal V_{\ell+1}= \oplus_{i=k+1}^{\ell+1} \mathcal W_i$,
such that for $\xi^{\ell+1} \in \widehat{\mathscr S}^{(\ell+1)}Q^{(k+1)}$, $\iota^{(\ell)}([\xi^{\ell+1}])(\mathcal V_{\ell+1})$ is vertical with respect to $\mathscr S^{(\ell+1)} Q^{(k)}  \rightarrow Q^{(k)}$. \\

For $\ell =k$, we have  already had
\begin{eqnarray*}
 \xymatrix{
       Q^{(k +1)} \ar[d]^{\frak g_{k+1}} \ar[r]  &\mathscr S^{(k+1)}(Q^{(k)})\ar[d]^{\frak g_{k+1}(\frak g_{k})}
   \\
  	  Q^{(k )}  \ar[r]^{ identity}  &    Q^{(k)}.
}
 \end{eqnarray*}
 and $\frak g_{k+1}(\frak g_k) = \frak g_{k+1} \oplus \mathcal W $, so that we can take $\mathcal W_{k+1}=\mathcal W$.

 Suppose that we have   embeddings  as in the above clam for $\ell  =m \geq k$. We will show that there are such embeddings for $\ell=m +1$.

 \begin{eqnarray*}
 \xymatrix{
      \mathscr S^{(m+2)}  Q^{(k +1)}   \ar[d]^{\frak g_{m+2}(\frak g_{k+1})}  \ar[r]^{\iota^{(m+2)}}  &\mathscr S^{(m+2)} Q^{(k)} \ar[d]^{\frak g_{m+2}(\frak g_{k})}
   \\
  	      \mathscr S^{(m+1)}  Q^{(k +1)}  \ar[d]^{\frak g_{m+1}(\frak g_{k+1})}  \ar[r]^{\iota^{(m+1)}}  &\mathscr S^{(m+1)} Q^{(k)}  \ar[d]^{\frak g_{m+1}(\frak g_{k})}
   \\
   \mathscr S^{(m)} Q^{(k+1 )}   \ar[r]^{\iota^{(m )}}   &    \mathscr S^{(m)} Q^{(k)}.
}
 \end{eqnarray*}

 By the induction assumption there is an embedding $\mathscr S^{(m+1)}  Q^{(k +1)}     \stackrel{\iota^{(m+1)}}{\longrightarrow}  \mathscr S^{(m+1)} Q^{(k)}$ and decompositions $\frak g_{m+1}(\frak g_{k}) = \frak g_{m+1}(\frak g_{k+1}) \oplus \mathcal W_{m+1}$ and $E^{(m+1)}(\frak g_k) =E^{(m+1)}(\frak g_{k+1}) \oplus \mathcal V_{m+1}$, where $\mathcal V_{m+1} =  \oplus_{i=k+1}^{m+1} \mathcal W_i$.
 Note that
 \begin{eqnarray*}
 \frak g_{m+2}(\frak g_{k+1}) &=& \Hom(\frak g_-, E^{(m+1)}(\frak g_{k+1}))_{m+2} \oplus \left(\oplus_{i=0}^{m} \Hom(\frak g_i, \frak g_{m+1}(\frak g_{k+1}) \right) \\
 \frak g_{m+2}(\frak g_{k}) &=&  \Hom(\frak g_-, E^{(m+1)}(\frak g_{k }))_{m+2} \oplus \left(\oplus_{i=0}^{m} \Hom(\frak g_i, \frak g_{m+1}(\frak g_{k }) \right).
 \end{eqnarray*}
 Therefore, there is a subspace $\mathcal W_{m+2}$ such that $$\frak g_{m+2}(\frak g_k) = \frak g_{m+2}(\frak g_{k+1}) \oplus \mathcal W_{m+2}.$$ Since $E^{(m+2)}(\frak g_{k+1}) = E^{(m+1)}(\frak g_{k+1}) \oplus \frak g_{m+2}(\frak g_{k+1})$ and $E^{(m+2)}(\frak g_{k}) = E^{(m+1)}(\frak g_{k}) \oplus \frak g_{m+2}(\frak g_{k})$, we have $$E^{(m+2)}(\frak g_k) = E^{(m+2)}(\frak g_{k+1}) \oplus \mathcal V_{m+2},$$ where $\mathcal V_{m+2} = \mathcal V_{m+1} \oplus \mathcal W_{m+2} =\oplus_{i=k+1}^{m+2} \mathcal W_i$.

 For $\zeta^{m+2} \in  \widehat{\mathscr S}^{(m+2)}  Q^{(k +1)}$ define $\widehat{\iota} (\zeta^{m+2})$ by
 \begin{eqnarray*}
  \widehat{\iota}(\zeta^{m+2})(X)  =\left\{ \begin{array}{ll}
 \iota^{(m+1)}_* \zeta^{m+2}(X) & \text{ for } X \in E^{(m+1)}(\frak g_{k+1})  \\
 \widetilde X & \text{ for } X \in \mathcal V_{m+1}.
 \end{array} \right.
  \end{eqnarray*}
\end{proof}



\subsection{Structure equations and structure functions} $\,$ \label{sect: structure functions}
%
%

Let $Q^{(k)} \stackrel{G_k}{\rightarrow} Q^{(k-1)}$ is a geometric structure of order $k$ and $\mathscr SQ^{(k)}$ be its completed universal frame bundle. We  write   $E(\frak g_k), \theta(\frak g_k), G(G_k), \frak g(\frak g_k)=\oplus \frak g_i(\frak g_k)$ simply as $E, \theta, \overline{G}, \overline{\frak g} =\oplus \overline{\frak g}_i$.

\subsubsection{ }
By Proposition \ref{prop:canonical forms} (1) there is a unique function
$$\gamma: \mathscr S Q^{(k)} \rightarrow \Hom(\wedge ^2E, E)$$
satisfying that
$$d \theta + \frac{1}{2} \gamma(\theta, \theta)=0.$$
We call $\gamma$ the {\it structure function} of $\mathscr S Q^{(k)}$.

\begin{proposition} \label{prop:structure functions under the action} For $z \in \mathscr S Q^{(k)}$, $a \in \overline{G}$  and $X,Y \in E$,
  $$\gamma(za)(X,Y) = a^{-1}\gamma(z)(aX,aY).$$
In other words,
$R_a^* \gamma= \rho(a)^{-1} \gamma$ for $a \in \overline{G}$. Here, the action of $\overline G$ on $\Hom(\wedge^2 E, E)$ is given by
$(\rho(a)\varphi)(X,Y) = a \varphi(a^{-1}X, a^{-1}Y)$ for $a \in \overline G$ and $X,Y \in E$.

\end{proposition}

\begin{proof} 
It follows from Proposition \ref{prop:canonical forms} (3).
Indeed, $0=d \theta(za) + \frac{1}{2} \gamma(za)(\theta(za), \theta(za))= d(a^{-1}\theta(z )) + \frac{1}{2} \gamma(za)(a^{-1}\theta(z), a^{-1}\theta(z))$ because $\theta(za) = a^{-1}\theta$.
Thus
$$d \theta(z ) + \frac{1}{2} a \gamma(za)(a^{-1}\theta, a^{-1}\theta)=0.$$
Hence $a \gamma(za)(a^{-1} \cdot \,\,,a^{-1}\cdot\,\,) = \gamma(z)(\,\,\,,\,\,\,)$.
\end{proof}

\begin{proposition}[Bianchi identity] \label{prop:properties of structure functions 3} 
$$\gamma(\gamma(\theta, \theta), \theta) = d \gamma (\theta, \theta)$$
\end{proposition}

\begin{proof}

By differentiating $d\theta +\frac{1}{2} \gamma(\theta, \theta)=0$ and replacing $d\theta$ by $-\frac{1}{2}\gamma(\theta, \theta)$, we get
\begin{eqnarray*}\gamma(\gamma(\theta, \theta), \theta) = d\gamma(\theta, \theta).
\end{eqnarray*}
\end{proof}


\subsubsection{ } \label{sect: structure functions of various degree}
Let us consider various pieces of the structure function $\gamma$ of $\mathscr S Q^{(k)}$, by decomposing the vector space $\Hom(\wedge^2 E,E)$ in which $\gamma$ takes values. 

Let $\pi^c_{ab}$ denote the projection $ \Hom(\wedge ^2 E,E) \rightarrow \Hom(\overline{\frak g}_a \wedge \overline{\frak g}_b, \overline{\frak g}_c)$ as well as the image $\pi^c_{ab} \Hom(\wedge^2E,E)$ (we shall assume the latter for a projection and its image).
We have the natural notion of degree: $\alpha \in \Hom(\wedge^2E,E)$ is {\it of  homogeneous  degree} $r$ if $\alpha(\overline{\frak g}_a \wedge \overline{\frak g}_b) \subset \frak g_{a+ b + r}$ for all $a,b \in \mathbb Z$. But we shall also use a modified degree which is well adapted to our setting:  $\alpha \in \Hom(\wedge^2 E, E)$ is {\it of modified degree} $s$ if $\alpha(\overline{\frak g}_a \wedge \overline{\frak g}_b) \subset \overline{\frak g}_{\widetilde{a} + \widetilde{b} + s}$ for all $a,b \in \mathbb Z$, where we set $\widetilde{a} =\min\{a,-1\}$. Then
$$\pi_{(r)}=\sum_{a,b} \pi^{a+b+r}_{ab}, \qquad \pi_{[s]} = \sum_{a,b} \pi^{\widetilde{a} + \widetilde{b} +s}_{ab}$$
are projections to the elements of homogeneous degree $r$, of modified degree $s$, respectively. Then the image of $\pi_{(r)}$ is $\Hom(\wedge ^2 E,E)_r$.

We define the corresponding filtrations $\{F^i\}$ and $\{\widetilde F^i\}$ by
$$F^i = \sum_{r \geq i } \pi_{(r)}, \qquad \qquad \widetilde{F}^j = \sum_{s \geq j} \pi_{[s]}$$
and we identify
\begin{eqnarray*}
\Hom(\wedge^2 E,E)^{(\ell)}:=\Hom(\wedge^2E,E) /F^{\ell+1} &\text{ with }& \pi^{(\ell)} := \sum_{r \leq \ell} \pi_{(r)} \\
\Hom(\wedge^2E,E)^{[m]}:=\Hom(\wedge^2E,E)/\widetilde{F}^{m+1} &\text{ with }& \pi^{[m]} := \sum_{s \leq m }\pi_{[s]}.
\end{eqnarray*}

Another projection is made by the direct sum decomposition:
$$\Hom(\wedge ^2E,E) = \Hom(\wedge ^2 \frak g_-, E) \oplus \Hom(\frak g_- \wedge E_+, E) \oplus \Hom(\wedge ^2 E_+, E) $$
where $E_+=\bigoplus_{a \geq 0} \overline{\frak g}_a $.
We denote the projections to each component by $\pi_I, \pi_{II}, \pi_{III}$, respectively. Noting that all projections commute, we write
$$\pi_{II} \circ \pi^{[m]} = \pi_{II}^{[m]}, \quad \pi_{II} \circ \pi_{(r)} \circ \pi_{[s]} = \pi_{II(r)[s]}, \text{ etc. }$$

Applying various projections, we define
$$\gamma_{a b}^c=\pi_{ab}^c \circ \gamma , \quad \gamma_I^{[m]}=\pi_I^{[m]}\circ \gamma,\quad \gamma_{II (r)[s]} = \pi_{II(r)[s]} \circ \gamma, \text{ etc.}$$

It being said, we have

\begin{proposition} $\,$ \label{prop:degree of structure function}

\begin{enumerate}
\item $\gamma_{I(d)}=0$ for $d<0$ and $\gamma_{I(0)}$ coincides with the Lie bracket of $\frak g_-$.
\item $\gamma_{II(d)}=0$ for $d <0$ and $\gamma_{II(0)}$ coincides with the action of $\frak g_-$ on $E$.
\item $\gamma^c_{IIIab}=0$ if $a,b \geq 0$ and $c < \max\{a,b\}-1$.
\end{enumerate}

\end{proposition}

\begin{proof} 

By Proposition \ref{prop: embedding} there is an embedding
$$ (\mathscr SQ^{(k)}, \theta(\frak g_k)) \stackrel{\iota}{\rightarrow} (\mathscr S Q^{(0)}, \theta(\frak g_0)),$$
and  (filtration  preserving)  injective morphisms
$$ E(\frak g_k) \stackrel{\iota}{\rightarrow} E(\frak g_0), \quad G(G_k) \stackrel{\iota}{\rightarrow} G(G_0)$$
such that $\iota^*  \theta (\frak g_0) = \iota \circ \theta(\frak g_k)$.
Since (1) and (2) hold  for $\mathscr S Q^{(0)}$, so do for $\mathscr S Q^{(k)}$.

For the statement (3),
note that $$\overline{\frak g}_{a}  \subset \Hom \left(\frak g_-, \frak g_- \oplus \left( \oplus_{i=0}^{a-1} \overline{\frak g}_i\right)\right)_{a} \oplus \Hom\left( \oplus_{i=0}^{  a-2} \overline{\frak g}_i, \overline{\frak g}_{a-1} \right).$$
%
%
%
  From $R_a^* \theta =a^{-1} \theta$ for $a \in \overline{G}$, it follows that  $L_A \theta = - \rho(A ) \theta$ for $A \in \overline{\frak g}$. Since $L_A \theta = \widetilde{A} \lrcorner d \theta + d(\widetilde{A} \lrcorner \theta) $, we have $-\widetilde{A} \lrcorner d \theta  =  d (\theta(\widetilde{A})) + \rho(A) \theta  $.  By Proposition \ref{prop:canonical forms} (2), for $A \in \overline{\frak g}_{a}$, $\theta (\widetilde{A}) =A \mod F^{a+1}$ and thus $-\widetilde{A} \lrcorner d \theta  =  \rho(A) \theta  \mod F^{a+1}$.
\end{proof}

\subsubsection{ }
On account of Proposition \ref{prop:degree of structure function}, we define a subspace $\check \pi = \ \cHom (\wedge^2E,E)$ by the condition:
$$\pi^{(-1)}_I = \pi^{(-1)}_{II} =0 \text{ and } \pi^c_{IIIab} =0 \text{ for } a,b,c \text{ such that } a,b \geq 0 \text{ and } c< \max\{a,b\}-1.$$
Recall that our group $\overline{G} $ acts on $E$ and it naturally induces an action on $\Hom(\wedge ^2E,E)$, denoted by $\rho$.
It is easy to see that $\overline{G}$ leaves invariant the filtration $\{\widetilde F^j\}$ and the subspace $\cHom(\wedge^2 E,E)$. Hence $\overline{G}$ acts on $\cHom(\wedge^2E,E)/\widetilde{F}^{\ell+1}  =: \cHom(\wedge^2E,E)^{[\ell]}$.

\begin{lemma}

The action of $F^{\ell+1}\overline{G}$ on $\cHom(\wedge^2E,E)^{[\ell]}$ is trivial for $\ell \geq k+1$.
\end{lemma}

\begin{proof}
We may assume that the group $\overline{G} $ is connected. Then it suffices to prove: For $A \in F^{\ell+1}\overline{\frak g}$,
$$\pi_{[m]}(\rho(A)\gamma) =0 \text{ for } m \leq \ell.$$
Let us verify this by showing that its projections by $\pi_I, \pi_{II}, \pi_{III}$ vanishes.

(1) For $X  \in \overline{\frak g}_p$ and $ Y  \in \overline{\frak g}_q$, where $p,q<0$, we have
\begin{eqnarray*}
\pi_{[m]}(\rho(A)\gamma)(X,Y) &=& \pi_{p+q+m}(\rho(A)\gamma)(X,Y) \\
&=& \pi_{p+q+m}\left(A \gamma(X,Y) - \gamma([A,X],Y) - \gamma(X,[A,Y]) \right) \\
&=& 0 \qquad \text{ because } \pi_{p+q+m} F^{p+q+\ell+1}=0 \text{ for } m \leq \ell.
\end{eqnarray*}

(2) For $X \in \overline{\frak g}_p$ and $Y \in \overline{\frak g}_a$, where $p<0$ and $a \geq 0$, we have
\begin{eqnarray*}
(\pi_{[m]}(\rho(A)\gamma))(X,Y) &=& \pi_{p-1+m}(A \gamma(X,Y) - \gamma(A\cdot X,Y) - \gamma(X, A\cdot Y)) \\
&=& 0 \qquad \text{ because } \pi_{p-1+m}(F^{\ell+p} + F^{\ell+p} + F^{\ell+p})=0.
\end{eqnarray*}

(3) For $X \in \overline{\frak g}_a$ and $Y \in \overline{\frak g}_b$, where $a,b \geq 0$, we have
\begin{eqnarray*}
\pi_{[m]}(\rho(A)\gamma)(X,Y)& =& \pi_{-2+m}(\rho(A)\gamma)(X,Y) \\
&=& 0 \qquad \text{ because } \pi_{-2+m}F^{\ell-1}=0.
\end{eqnarray*}
Here, we use the fact that $F^{\ell+1}\overline{\frak g} \cdot E_+ \subset F^{\ell}$.
\end{proof}

Then we have immediately

\begin{proposition}
The structure function $\gamma$ of $\mathscr SQ^{(k)}$ induces a $\overline{G}$-equivariant map
$$\gamma^{[\ell]}: \mathscr S^{(\ell)}Q^{(k)} \rightarrow \cHom(\wedge^2 E, E)^{[\ell]}$$
for $\ell \geq k$.
\end{proposition}

\begin{proposition} \label{prop:isomorphism preserves structure function in the order ell}
If $ \varphi^{(k)}: P^{(k)} \rightarrow  Q^{(k)}$ preserves the canonical Pfaff classes $[\theta^{(k-1)}]_P$ and $[\theta^{(k-1)}]_Q$, then the canonical lift $\mathscr S^{(\ell)}\varphi^{(k)}: \mathscr S^{(\ell)}P^{(k)} \rightarrow \mathscr S^{(\ell)}Q^{(k)}$ preserves the structure functions $\gamma^{[\ell]}_P$ and $\gamma^{[\ell]}_Q$ for any $\ell \geq k+1$.
\end{proposition}

\begin{proof}
An isomorphism $ \varphi^{(k)}: P^{(k)} \rightarrow  Q^{(k)}$ lifts to an isomorphism $ \mathscr S^{(\ell)}\varphi^{(k)}:  \mathscr S^{(\ell)}P^{(k)} \rightarrow   \mathscr S^{(\ell)}Q^{(k)}$ for any $\ell \geq k+1$, and thus an isomorphism $ \mathscr S \varphi^{(k)}:  \mathscr S P^{(k)} \rightarrow  \mathscr S  Q^{(k)}$. Therefore, $\mathscr S \varphi^{(k)}$ preserves the structure functions, $\gamma_P$ and  $\gamma_Q$, and thus $\mathscr S^{(\ell)}\varphi^{(k)}$ preserves the structures function $\gamma^{[\ell]}_P$ and $\gamma^{[\ell]}_Q$.
\end{proof}

Since $\gamma^{[k+1]}=\sum_{s \leq k+1} \gamma_{[s]}$, the map  $\gamma^{[k+1]}:\mathscr S^{(k+1)}Q^{(k)} \rightarrow \cHom(\wedge^2 E, E)^{[k+1]}$ induces a function
$$ \gamma_{[k+1]} : \mathscr S^{(k+1)}Q^{(k)} \rightarrow \cHom(\wedge^2 E, E)^{[k+1]},$$
which will play an important role in the following section.


\section{$W$-normal Step   prolongations} \label{sect:normal step prolongation}


\subsection{Proper  geometric structures}

\begin{definition} \label{def:normal geometric structure}
A geometric structure $Q^{(k)}$ of type $(\frak g_-, G_0, \dots, G_k)$ is said to be {\it proper} if $E^{(k)}:=\frak  g_- \oplus \frak g_0 \oplus \dots \oplus \frak g_k$ forms  a truncated transitive graded Lie algebra (TTGLA),  which we denote by $\frak g[k]$.
\end{definition}

Let us explain more precisely what the above definition means. Recall that $(\frak g_-, \frak g_0, \dots, \frak g_k)$ satisfies
\begin{enumerate}
 \item[$\bullet$] $\frak g_0 \subset \frak g_0(\frak g_-)$   and
 \item[$\bullet$] $\frak g_i \subset \frak g_i(\frak g_{i-1})=\left(\oplus_{p<0} \Hom(\frak g_p, \frak g_{p+i})\right) \oplus \left( \oplus_{a=0}^{i-2} \Hom(\frak g_a, \frak g_{i-1})\right)$   for   1 $\leq i \leq k$.
\end{enumerate}

Saying the $E^{(k)}=\frak  g_- \oplus \frak g_0 \oplus \dots \oplus \frak g_k$ becomes a truncated transitive graded Lie algebra requires the following   conditions to be satisfied: for $0 \leq i \leq k$,
\begin{enumerate}
\item $\frak g_i \subset \oplus_{p<0} \Hom(\frak g_p, \frak g_{p+i})$;
\item $\frak g_i \subset Prol\,(\frak g[i-1])_i$, that is, $E^{(i-1)}=\frak  g_- \oplus \frak g_0 \oplus \dots \oplus \frak g_{i-1}$ is a truncated transitive graded Lie algebra $\frak g[i-1]$ and $\frak g_i$ is a subspace of the $i$-th component of $Prol\,(\frak g[i-1]) $;
\item $E^{(i)}=\frak  g_- \oplus \frak g_0 \oplus \dots \oplus \frak g_i$ is closed under the truncated Lie bracket defined in $Trun^{(i)}\,(Prol\, \frak g[i-1])$ and then forms a truncated transitive graded Lie algebra $\frak g[i]$.
\end{enumerate}

Note that a geometric structure $Q^{(0)}$ of order 0 ia always proper. If $Q^{(k)}$ is a $W$-normal prolongation of $Q^{(0)}$ (that we are going to study), then $Q^{(k)}$ is proper.

By (3), $\frak g_i$ ($0 \leq i \leq k$) acts trivially on $E_+^{(i-1)}:= \frak g_0 \oplus \dots \oplus \frak g_{i-1}$ and by (2), $\frak g_i$ is contained in the prolongation of $\frak g_- \oplus \frak g_0 \oplus \dots \oplus \frak g_{i-1}$. \\




 By this observation we have the following.

\begin{proposition} \label{prop:degree of gamma on positive set}  If $Q^{(k)}$ is proper, then the structure function $\gamma$ of the universal frame bundle $\mathscr S Q^{(k)}$  satisfies that
$ \gamma^c_{ab}$ vanishes  for $ 0 \leq a,b  \leq k $ and $ c <\max\{a,b\}.$

\end{proposition}

 %

\subsection{$W$-normal reductions} \label{sect: definition of kappa and tau} $\,$

Let $Q^{(k)}$ be a proper geometric structure of type $\frak g[k]$, $\mathscr S Q^{(k)}$ the universal frame bundle, and $\gamma$ its structure function.
    Let  $E$, $ \overline{G} $, $ \overline{\frak g} $, $\theta$, $\gamma$, $\gamma^{[k+1]}$, $\gamma_{[k+1]}$  be given as in Section \ref{sect: structure functions}.
     We will often write alternatively $\kappa, \tau, \sigma$ for $\gamma_{I}, \gamma_{II}, \gamma_{III}$.
Then   $\kappa_{[k+1]}$, $\tau_{[k+1]}$, $\sigma_{[k+1]}$ are functions on  $\mathscr S^{(k+1)}Q^{(k)}$  defined as follows.
 $$ \kappa_{[k+1]} =\gamma_{k+1}|_{\wedge^2 \frak g_-}, \qquad  \tau_{[k+1]} =\sum _{\substack{\ell+a=k \\ a \geq 0, \ell \geq 0 }}\gamma_{\ell}|_{\frak g_a \wedge \frak g_-}, \qquad  \sigma_{[k+1]} =\sum_{\substack{\ell+a+b=k-1\\ a,b \geq 0 }}
\gamma_{\ell}|_{\frak g_a \wedge \frak g_b}. $$

\noindent
Recalling the construction of the universal flame bundles, we see easily that Lie algebra $\overline{\frak g}_{k+1}$ of the structure group of the principal bundle  $\mathscr S^{(k+1)}Q^{(k)} \rightarrow Q^{(k)}$ is $$\frak q_{k+1} + \frak r_{k+1} ,$$ where  $  \frak q_{k+1} =\Hom\left(\frak g_-, \frak g_- \oplus (\oplus_{i=0}^{k}\frak g_i)\right)_{k+1}$ and $ \frak r_{k+1}=\Hom\left(\oplus_{i=0}^{k-1} \frak g_i, \frak g_k\right)$ (Section \ref{sect:geometric structure of order k}).
  For $\varphi \in  \frak q_{k+1}  $ and $\psi \in \frak r_{k+1} $ denote by $a_{\varphi}$ and $a_{\psi}$ the   elements  in $\overline{G}_{k+1} $ corresponding to $\varphi$ and $\psi$. Then for $X \in \frak g_p$ ($p <0)$ and $A \in \frak g_a$ ($a \geq 0$),
  \begin{eqnarray*}
  a_{\varphi } X &\equiv& X + \varphi X \mod F^{p+k+2} \\
  a_{\psi} A&\equiv& A +\psi A \mod F^{k+1}.
  \end{eqnarray*}

\begin{proposition} \label{prop:structure equation along fiber} The structure function $\gamma_{[k+1]}:\mathscr S^{(k+1)}Q^{(k)}  \rightarrow \cHom(\wedge^2 E, E)^{[k+1]}$ satisfies that, for $\varphi \in \frak q_{k+1}$ and   $\psi \in \frak r_{k+1}$, we have

\begin{enumerate}
\item  $\kappa_{[k+1]}(z a_{\varphi }) = \kappa_{[k+1]}(z) + \partial \varphi$

    \item
    \begin{enumerate}
    \item $    \tau_{[k+1]}(z a_{\varphi })(A,\,\cdot\,) = \tau_{[k+1]}(z)(A,\,\cdot\,)  $
    \item
         $\tau_{[k+1]}(z a_{ \psi})(A,\,\cdot\,)= \tau_{[k+1]}(z)(A,\,\cdot\,)+ \partial \psi(A)$,  where $A \in \oplus_{i=0}^{k-1} \frak g_i$.

    \end{enumerate}

\end{enumerate}

\end{proposition}

  \begin{proof}

(1)  Let $X \in \frak g_p$ and $Y \in \frak g_{q}$, where $p<0$ and $q <0$.
  From $\gamma(za_{\varphi })(X,Y) =a_{\varphi}^{-1} \gamma(z)(a_{\varphi } X, a_{\varphi }Y)$ (Proposition \ref{prop:structure functions under the action}) it follows that
 \begin{eqnarray*}
 \gamma(z a_{\varphi })(X,Y) &=&  \gamma(z)(X,Y) + \gamma(z)(X, \varphi Y) + \gamma(z)(\varphi X, Y) + \gamma(z)(\varphi X, \varphi Y) + \dots   \\
  && -\varphi\left( \gamma(z)(X,Y) + \gamma(z)(X, \varphi Y) + \gamma(z)(\varphi X, Y) + \gamma(z)(\varphi X, \varphi Y) + \dots \right) \dots
  \end{eqnarray*}
  By comparing $\frak g_{p+q+k+1}$-component,   we get
  $$ \kappa_{[k+1]}(z a_{\varphi })(X,Y)
   =
  \kappa_{[k+1]}(z)(X,Y) + [X, \varphi Y] + [\varphi X,Y] - \varphi [X,Y].
 $$
  Here, we use that $\gamma_0|_{\wedge^2 \frak g_-} $ and $\gamma_0|_{\frak g_a \wedge \frak g_-}$ are given by the Lie bracket $[\,,\,]$.
  It follows that  $\kappa_{[k+1]}(z a_{\varphi}) =\kappa_{[k+1]}(z) + \partial \varphi$.

  (2) Let   $Y \in \frak g_q$ for some $q <0$ and $A \in \frak g_a$ for some $0 \leq a \leq k-1$.
  %
  As in (1), from $\gamma(z a_{\varphi+\psi})(A,Y) = a_{\varphi+\psi}^{-1} \gamma(z)(a_{\varphi+\psi} A, a_{\varphi+\psi}Y)$ it follows that
  \begin{eqnarray*}
 \gamma(z a_{\varphi })(A,Y) &=&  \gamma(z)(A,Y) + \gamma(z)(A, \varphi Y) + \gamma(z)(\varphi A, Y) + \gamma(z)(\varphi A, \varphi Y) + \dots   \\
   && -\varphi\left( \gamma(z)(A,Y) + \gamma(z)(A, \varphi Y) + \gamma(z)(\varphi A, Y) + \gamma(z)(\varphi A, \varphi Y) + \dots \right) \dots
  \end{eqnarray*}
   By comparing the $\frak g_{q+k}$-component,  
    we get
    \begin{eqnarray*}
    \tau_{[k+1]}(z a_{ \varphi})(A,Y) &=&\tau_{[k+1]}(z)(A,Y)  \\
 \tau_{[k+1]}(z a_{ \psi})(A,Y) &=&\tau_{[k+1]}(z)(A,Y) + [\psi(A),  Y].
 \end{eqnarray*}
 Here, we use Proposition \ref{prop:degree of gamma on positive set} in the first identity. Indeed, $\gamma(z)(A, \varphi Y) $ belongs to $F^{\max\{a, q+k+1\}}$ if $q+k+1 \geq 0$, and to $F^{a+q+k+1}$ if $q+k+1<0$, and thus their $\frak g_{q+k}$-component is zero.
  \end{proof}



Now let us make the key procedure of $W$-normal reduction consisting of $\gamma_{II}$-reduction ($\tau$-reduction) and $\gamma_I$-reduction ($\kappa$-reduction). 

 \begin{proposition}  \label{prop: step prolongation of order k+1}
Let $\frak g[k]=\oplus_{i \leq k} \frak g_i$ be a truncated transitive graded Lie algebra and $ \frak g =\oplus_i \frak g_i$ be its prolongation. Fix subspaces $W^1_k$ and $W^2_{k+1}$ such that
\begin{eqnarray*}
 \Hom (\frak g_-,  \frak g )_{k} &= &W_{k}^1 \oplus \partial  \frak g_{k}\\
 \Hom(\wedge^2\frak g_-,   \frak g )_{k+1} &= &W_{k+1}^2 \oplus \partial \Hom(\frak g_-,  \frak g )_{k+1} .
 \end{eqnarray*}
 Let $Q^{(k)} \stackrel{G_k}{\longrightarrow} Q^{(k-1)}  $ be a proper geometric structure  of type $\frak g [k]$ on a filtered manifold $M$. Let $\mathscr S^{(k+1)}  Q^{(k)}$ be the universal  frame bundle of order $k+1$.  Let $\kappa_{[k+1]}$ and $\tau_{[k+1]}$ be the structure functions of $\mathscr S^{(k+1)} Q^{(k)} $. 
 \begin{enumerate}
 \item {\rm(}$\gamma_{II}$-reduction{\rm)}
 Define a subbundle $\mathscr S_{W^1}^{(k+1)}Q^{(k)}$ of $\mathscr S^{(k+1)} Q^{(k)} $ by
\begin{eqnarray*}
\mathscr S_{W^1}^{(k+1)}Q^{(k)}=\left\{z \in \mathscr S^{(k+1)} Q^{(k)} : \tau_{[k+1]}(z)(A,\,\cdot\,) \in W_k^1 \text{ for any } A \in \oplus_{i=0}^{k-1} \frak g_i \right\}.
\end{eqnarray*}
Then $\mathscr S_{W^1}^{(k+1)}Q^{(k)}\rightarrow Q^{(k)}$ is a principal $\mathfrak Q_{k+1}$-subbundle of $\mathscr S^{(k+1)}Q^{(k)}\rightarrow Q^{(k)}$, where $\mathfrak Q_{k+1}$ is  the maximal    subgroup of $\overline{G}_{k+1}$ whose    Lie algebra is  $\frak q_{k+1}  \subset \overline{\frak g}_{k+1}$.
 \item {\rm(}$\gamma_I$-reduction{\rm)}
 Define a subbundle  $\mathscr S_{W }^{(k+1)}Q^{(k)}$ of  $\mathscr S_{W^1} ^{(k+1)}Q^{(k)}$ by
 \begin{eqnarray*}
\mathscr S_{W}^{(k+1)}Q^{(k)}=\left\{z \in \mathscr S_{W^1} ^{(k+1)}Q^{(k)}: \kappa_{[k+1]}(z) \in W_{k+1}^2  \right\}.
\end{eqnarray*}
Then $\mathscr S_{W}^{(k+1)}Q^{(k)} \rightarrow Q^{(k)}$ is a principal $G_{k+1}$-subbundle of $\mathscr S_{W^1}  ^{(k+1)}Q^{(k)} \rightarrow Q^{(k)}$, where $G_{k+1}$ is the maximal subgroup of $\overline{G}_{k+1}$ whose   Lie algebra is $\frak g_{k+1} \subset \overline{\frak g}_{k+1} $.
 \end{enumerate}
 Consequently, $\mathscr S_W^{(k+1)}Q^{(k)}  \rightarrow Q^{(k)}$ is a proper geometric structure of order $k+1$ of type $\frak g[k+1] = \frak g[k] \oplus \frak g_{k+1}$.
 \end{proposition}

\begin{proof}

(1) For $A \in E_+^{(k-1)}=\oplus_{i=0}^{k-1}\frak g_i$, consider
$$\tau_{[k+1]}(A, \,\cdot\,):\mathscr S^{(k+1)}Q^{(k)} \rightarrow \Hom(   \frak g_-, \frak g)_{k}.$$
Take any $z \in \mathscr S^{(k+1)}Q^{(k)}$ and write
$$\tau_{[k+1]}(z)(A,\,\cdot\,) =\partial \psi(A) + \alpha(A)$$
where $\psi \in \Hom(E_+^{(k-1)}, \frak g_k)$ and $\alpha \in \Hom(E_+^{(k-1)},W^1_k)$.
Then by Proposition \ref{prop:structure equation along fiber} (2)(b), we have
$\tau_{[k+1]}(z a_{-\psi})(A, \,\cdot\,) = \alpha(A) \in W^1_k$. Therefore, $\mathscr S_{W^1}^{(k+1)}Q^{(k)}$ is nonempty.

 By Proposition \ref{prop:structure equation along fiber} (2)(a), if $z $ is an element of $ \mathscr S_{W^1}^{(k+1)}Q^{(k)}$, then so is $z a_{\varphi}$ for any $\varphi \in \frak q_{k+1}$.
 On the other hand,
if both $z$ and $z a_{\psi}$, where $\psi \in \mathfrak r_{k+1}$, are contained in $\mathscr S_{W^1}^{(k+1)}Q^{(k)}$, then,
by Proposition \ref{prop:structure equation along fiber} (2)(b) 
we have $$ -\partial \psi(A)=\tau_{[k+1]}(z)(A, \,\cdot\,)  -\tau_{[k+1]}(z a_{ \psi})(A, \,\cdot\,)\in W_k^1 .$$  
Since $W_k^1 \cap \partial  \frak g_{k}=0$, $\partial \psi(A) $ is zero, and hence $\psi(A)$ is zero. Therefore, $\mathscr S_{W^1}^{(k+1)}Q^{(k)}$ intersects $\{z a_{\psi}: \psi \in \mathfrak r_{k+1}\}$ only at one point.
Therefore, $\mathscr S_{W^1}^{(k+1)}Q^{(k)} \rightarrow Q^{(k)}$ is a principal subbundle of $\mathscr S^{(k+1)}Q^{(k)}\rightarrow Q^{(k)}$ whose structure group has Lie algebra $\frak q_{k+1} $.

(2) Consider
$$\kappa_{[k+1]}:\mathscr S_{W^1}^{(k+1)}Q^{(k)} \rightarrow \Hom(\wedge^2 \frak g_-, \frak g)_{k+1}.$$
Take any $z \in \mathscr S_{W^1}^{(k+1)}Q^{(k)}$ and write
$$\kappa_{[k+1]}(z) = \partial \varphi_{k+1} + \alpha$$
where $\varphi_{k+1} \in \Hom(\frak g_-,\frak g)_{k+1}$ and $\alpha \in W^2_{k+1}$.
Then by Proposition \ref{prop:structure equation along fiber} (1), we have $\kappa_{[k+1]}(z a_{-\varphi_{k+1}}) =\alpha \in W^2_{k+1}$. Therefore, $\mathscr S_{W}^{(k+1)}Q^{(k)}$ is nonempty.

Furthermore, if $z$ and $za_{\varphi}$, where $\varphi \in \mathfrak q_{k+1}$ are contained in $\mathscr S_W^{(k+1)}Q^{(k)}$, then by Proposition \ref{prop:structure equation along fiber} (1), we have   $\partial \varphi =\kappa_{[k+1]}(z a_{\varphi }) - \kappa_{[k+1]}(z)  \in W^2_{k+1} $. Therefore, 
Since $\partial \frak g_{k+1} \cap W^2_{k+1} =0$, $\partial \varphi$ 
is zero. Thus $\varphi $ is contained in the kernel of $\partial :  \frak q_{k+1} \rightarrow \Hom(\wedge^2 \frak g_-, \frak g)_{k+1}$,  which is $\frak g_{k+1}$.
%
%
Consequently,  $\mathscr S_W^{(k+1)}Q^{(k)}$ is a principal  subbundle of $\mathscr S_{W^1}^{(k+1)}(Q^{(k)})$ whose structure group has Lie algebra $\frak g_{k+1} $.
%
%
%
%
\end{proof}

\begin{proposition} \label{prop: step prolongation isomorphism} We use the same notations as in Proposition \ref{prop: step prolongation of order k+1}.

\begin{enumerate}

 \item
 If there is an isomorphism   $\varphi^{(k)}:P^{(k)} \rightarrow Q^{(k)}$ of proper geometric structures $P^{(k)}$ and $Q^{(k)}$, then there is a unique isomorphism $\mathscr S_W^{(k+1)}\varphi^{(k)}:\mathscr S_W^{(k+1)}P^{(k)} \rightarrow \mathscr S^{(k+1)}_WQ^{(k)}$ prolonging $\varphi^{(k)}$.

 \item Conversely, if    $G_0, G_1, \dots, G_k$ are   connected and there is 
     a diffeomorphism $\Phi^{(k+1)}:\mathscr S_W^{(k+1)}P^{(k)} \rightarrow \mathscr S^{(k+1)}_WQ^{(k)}$ such that $(\Phi^{(k+1)})^* [\theta^{(k)}]_Q = [\theta^{(k)}]_P$, then there exists an isomorphism $\varphi^{(k)}:P^{(k)} \rightarrow Q^{(k)}$ such that
     the restriction   of the lift $\mathscr S^{(k+1)}\varphi^{(k)}$ to $\mathscr S_W^{(k+1)}P^{(k)}$ is $\Phi^{(k+1)}$.  

\end{enumerate}

\end{proposition}

\begin{proof}
(1) If  $\varphi^{(k)}:P^{(k)} \rightarrow Q^{(k)}$ is  an isomorphism   of proper geometric structures $P^{(k)}$ and $Q^{(k)}$, then there is an isomorphism $\mathscr S ^{(k+1)}\varphi^{(k)}:\mathscr S ^{(k+1)}P^{(k)} \rightarrow \mathscr S^{(k+1)} Q^{(k)}$ prolonging $\varphi^{(k)}$. Hence $ \mathscr S^{(k+1)}\varphi^{(k)} $ preserves the canonical Pfaff classes $[\theta^{(k)}]_P$ and $[\theta^{(k)}]_Q$ of $\mathscr S ^{(k+1)}P^{(k)}$ and $\mathscr S^{(k+1)} Q^{(k)}$,   and thus preserves their structure functions $\gamma^{[k+1]}_P$ and $\gamma^{[k+1]}_Q$ (Proposition \ref{prop:isomorphism preserves structure function in the order ell}).  
Now that  $\mathscr S_W^{(k+1)}P^{(k)}$ ($\mathscr S_W^{(k+1)}Q^{(k)}$, respectively) is defined by the condition  that $\kappa_{ P, [k+1]}(z) \in W^2_{k+1}$ and $\tau_{ P, [k+1]}(z)(A, \,\cdot\,) \in W^1_{k}$ ($\kappa_{ Q, [k+1]}(z) \in W^2_{k+1}$ and $\tau_{Q, [k+1]}(z)(A, \,\cdot\,) \in W^1_{k}$, respectively),  $\mathscr S^{(k+1)}\varphi^{(k)}$ maps $\mathscr S_W^{(k+1)}P^{(k)}$ onto $\mathscr S_W^{(k+1)}Q^{(k)}$ isomorphically.

(2) follows from Proposition \ref{prop:isomorphism via canonical class}.
\end{proof}

 \begin{theorem} \label{thm:W normal prolongation}
 Let $\frak g[k]$ be a truncated transitive graded Lie algebra and $ \frak g $ be its prolongation.
 Fix a set of  subspaces $W=\{W^1_{\ell}, W^2_{\ell+1}\}_{\ell \geq k}$ such that
\begin{eqnarray*}
 \Hom (\frak g_-,  \frak g )_{\ell} &= &W_{\ell}^1 \oplus \partial  \frak g_{\ell}\\
 \Hom(\wedge^2\frak g_-,   \frak g )_{\ell+1} &= &W_{\ell+1}^2 \oplus \partial \Hom(\frak g_-,  \frak g )_{\ell+1} .
 \end{eqnarray*}
  Then to every proper geometric structure $Q^{(k)}$ of type $\frak g[k]$ there is canonically associated a series of proper geometric structures $\{\mathscr S_W^{(\ell)}Q^{(k)}\}_{\ell = k+1, \dots, \infty}$  of type $\frak g[\ell]$  which is uniquely determined by the condition that
 $\mathscr S_W^{(\ell)}Q^{(k)}$ is a maximal proper geometric structure prolonging $Q^{(k)}$ whose structure function $\gamma^{[\ell]}$ satisfies
 $$\kappa_{[\ell ]}(z) \in W^{2 }_{\ell} \text{ and } \tau_{[\ell ]}(z) \in \Hom(E_+^{(\ell-1)}, W^{1 }_{\ell-1}). $$

 Furthermore, we have the following.

 \begin{enumerate}
 \item 
 If there is an isomorphism   $\varphi^{(k)}:P^{(k)} \rightarrow Q^{(k)}$ of proper geometric structures $P^{(k)}$ and $Q^{(k)}$, then there is a unique isomorphism $\mathscr S_W^{(\ell)}\varphi^{(k)}:\mathscr S_W^{(\ell)}P^{(k)} \rightarrow \mathscr S^{(\ell)}_WQ^{(k)}$ prolonging $\varphi^{(k)}$ for $\ell=k+1, \dots, \infty$.

 \item Conversely, assume that $G_0, G_1, \dots, G_k$ are   connected. For $\ell=k+1, \dots, \infty$, if there is 
     a diffeomorphism $\Phi^{(\ell)}:\mathscr S_W^{(\ell)}P^{(k)} \rightarrow \mathscr S^{(\ell)}_WQ^{(k)}$ such that $(\Phi^{(\ell)})^* [\theta^{(\ell-1)}]_Q = [\theta^{(\ell-1)}]_P$, then there exists an isomorphism $\varphi^{(k)}:P^{(k)} \rightarrow Q^{(k)}$ such that
     the restriction  of the lift $\mathscr S^{(\ell)}\varphi^{(k)}$  to $\mathscr S_W^{(\ell)}P^{(k)}$ is $\Phi^{(\ell)}$.  

 \end{enumerate}

 Therefore, the equivalence problem of the proper geometric structure $Q^{(k)}$ of type $\frak g[k]$ reduces to that of the absolute parallelism of $(\mathscr S_WQ^{(k)}, \theta)$.
 \end{theorem}

 We call $\mathscr S^{(\ell)}_W Q^{(k)} $ the $W$-{\it normal step prolongation} of   $Q^{(k)}$  {\it of order} $\ell$ for $\ell \geq k+1$. We denote the projective limit $ \lim_{\leftarrow \ell} \mathscr S^{(\ell)}_W Q^{(k)}$ by $\mathscr S_W Q^{(k)} $ and call it the $W$-{\it normal complete step  prolongation} of $Q^{(k)}$.

%

 \begin{proof}
Apply Proposition \ref{prop: step prolongation of order k+1} inductively to get $\mathscr S_W^{(\ell)}Q^{(k)}$ for $\ell \geq k+1$ such that
$$\kappa_{[\ell ]}(z) \in W^{2 }_{\ell} \text{ and } \tau_{[\ell ]}(z) \in \Hom(E_+^{(\ell-1)}, W^{1 }_{\ell-1}). $$
The statements (1) and (2) follows from Proposition \ref{prop: step prolongation isomorphism}.
 \end{proof}



\subsection{Proper geometric structure of finite type}

\begin{definition}
A proper geometric structure $Q^{(k)}$ of type $\frak g[k]=\frak g_- \oplus \frak g_0 \oplus \dots \oplus \frak g_k$ ($k \geq 0$) is said to be {\it of finite type} if the prolongation  $Prol\, (\frak g[k])$ is finite dimensional, that is, there exists an integer $k_0 \geq 0$ such that $\frak g_i=0$ for $i >k_0$, where we set $Prol\, (\frak g[k]) =\oplus \frak g_i$.
\end{definition}

\begin{corollary} \label{cor:finite dimensional absolute parallelism}
The equivalence problem of the proper geometric structure $Q^{(k)}$ of finite type reduces to that of the absolute parallelism $(\mathscr S_WQ^{(k)}, \theta)$ of finite dimension.

\end{corollary}

\begin{corollary} \label{cor:finite dimensional Lie group}
The automorphism group $\Aut(Q^{(k)})$ of a proper geometric structure $Q^{(k)}$ of finite type is a finite dimensional Lie group.  The dimension of $\Aut(Q^{(k)})$ is less than or equal to the dimension of $  Prol\,(\frak g[k])$.

\end{corollary}

It was Cartan who invented the ingenious idea to study the equivalence problem of geometric structures through prolongation and reduction of bundles of frames adapted to geometric structures. It seems that he had a very general and deep ideas, which were, however, of heuristic nature, and carried applications in various concrete problems. Afterwards, the theoretical aspects were developed, to certain extent, rigorously in modern mathematics. A modern formulation of the step prolongation is due to Singer and Sternberg (\cite{SS65}) and the generalization to the nilpotent geometry is due to Tanaka (\cite{T70}). Their main results were Corollary \ref{cor:finite dimensional absolute parallelism} and Corollary \ref{cor:finite dimensional Lie group}.



\section{Fundamental identities} \label{sec:inductive formula} 

From now on, we will  write alternatively $\kappa, \tau, \sigma$ for $\gamma_{I}, \gamma_{II}, \gamma_{III}$.
We denote by the same notation their restrictions to $\mathscr S_W^{(k+1)}Q^{(k)}$.
Recall that we introduce the following notation to keep the information on the degree of structure functions.

\begin{definition}
Set $\tau_{(\ell)[k+1]}$ to be the component of $\tau$ of degree $\ell$ and modified degree $k+1$, so that  $$\tau_{(\ell)[k+1]}:= \gamma_{\ell}|_{\frak g_{k-\ell}\wedge \frak g_-}$$ for $0 \leq \ell \leq k$.
We say that $\tau_{[k+1]}$ is {\it flat} if  $\tau_{(\ell)[k+1]}$ is zero for all $\ell >0$, so that $\tau_{[k+1]} =  \tau_{(0)[k+1]}$.

Similarly, we  set $\sigma_{(\ell)[k+1]}$ to be the component of $\sigma$ of degree $\ell$ and modified degree $k+1$, so that
$$\sigma_{(\ell)[k+1]}:=\sum_{ i+j = k-1-\ell}\gamma_{\ell}|_{\frak g_i \wedge \frak g_j}$$
for any $\ell \in \mathbb Z$.
We say that $\sigma_{[k+1]}$ is {\it flat} if $\sigma_{(\ell)[k+1]} $ is zero for any $\ell \not=0$ and  $ \sigma_{(0)[k+1]}$ is given by the Lie bracket $[\,,\,]$.

We define
\begin{eqnarray*}
\kappa^{[k+1]}&:=&\sum _{0 \leq i \leq k+1} \kappa_{[i]}\\
\tau^{[k+1]}&:=&\sum_{1 \leq i \leq k+1} \tau_{[i]} \\
\sigma^{[k+1]}&:=&\sum_{2 \leq i \leq k+1}\sigma_{[i]}.
\end{eqnarray*}
We say that $\tau^{[k+1]}$ ($\sigma^{[k+1]}$, respectively) is {\it flat} if $\tau_{[i]}$ ($\sigma_{[i]}$,  respectively) is flat for all $i \leq k+1$.
\end{definition}

In this section we will prove a recursive formula for $\kappa_{[k]}, \tau_{[k]}$, and $\sigma_{[k]}$. 


\begin{theorem}[Fundamental identities] $\,$ \label{thm:fundamental identities}

\begin{enumerate}
\item For $X \in \frak g_x$, $Y \in \frak g_y$, and $Z \in \frak g_z$,   where $x,y,z <0$, and for a nonnegative integer $k$, we have
    \begin{eqnarray*}
    &&(\partial \kappa_{[k]})(X,Y,Z) \\
    &=& \mathfrak S_{X,Y,Z} \left\{ \sum_{\substack{d_1+ d_2 =k\\ d_1, d_2 >0}} \left\{\kappa_{[d_1]}(\kappa_{[d_2]}(X,Y)_-,Z) + \tau_{(d_1)[k+x+y+1]}(\kappa_{[d_2]}(X,Y)_+, Z) \right\} - D_X \kappa_{[k+x]}(Y,Z)\right\}.
    \end{eqnarray*}
    \vskip 10 pt

\item For $X \in \frak g_x$, $Y \in \frak g_y$, and $A \in \frak g_a$, where $x,y<0$ and $a \geq 0$, and for  $k=d+a+1$,  where $d$ is a nonnegative  integer,   we have
    \begin{eqnarray*}
    && (\partial \tau_{(d)[k]}(A,\,\cdot\,))(X,Y) \\
    &=&  ( D_A \kappa_{[k-1]} + \rho(A) \kappa_{[d]} )(X,Y) - \mathscr A_{X,Y} \tau_{(d)[k+x]}([A,X]_+, Y) +[\kappa_{[d]}(X,Y)_+,A] -\sigma_0(\kappa_{[d]}(X,Y)_+,A) \\
    && - \sum_{\substack{\delta_1 +\delta_2=d\\-a \leq \delta_1 <d \text{ and } \delta_1 \not=0 \\0 < \delta_2 \leq d+a =k-1}} \left\{ \tau_{(d_1)[k-1]}(\kappa_{[\delta_2]}(X,Y)_-,A) + \sigma_{(\delta_1)[k+x+y+1]}(\kappa_{[\delta_2]}(X,Y)_+,A)\right\} \\
    && + \,\,\mathscr A_{X,Y}\sum_{\substack{ d_1+d_2=d\\d_1,d_2 >0}} \left\{ \kappa_{[d_1]}(\tau_{(d_2)[k-d_1]}(A,X)_-,Y) + \tau_{(d_1)[k+x]}(\tau_{(d_2)[k-d_1]}(A,X)_+,Y)\right\} \\
    &&+ \,\,\mathscr A_{X,Y} D_Y \tau_{(d+y)[k+y]}(A,X).
    \end{eqnarray*}
\vskip 10 pt

\item For $A_a \in \frak g_a$, $B \in \frak g_b$, $X \in \frak g_x$, where $a,b \geq 0$ and $x<0$, and for $k=d+a+b+2$, where     $d $ is an integer,  we have
    \begin{eqnarray*}
    &&[\sigma_{(d)[k]}(A,B),X] \\
    &=& \mathscr A_{A,B} \sum_{\substack{d_1+d_2=d \\0 \leq d_2 \leq d+a 
    }}
    \left\{ \sigma_{(d_1)[k+x-d_2]}(A, \tau_{(d_2)[k-1-(d_1+a)]}(B,X)_+) \right. \qquad \qquad \qquad \qquad \qquad\\[-15 pt]
     && \qquad \qquad \qquad \qquad \qquad \qquad\qquad \qquad \left.+ \,\, \tau_{(d_1)[k-1-(d_2+b)]}(A, \tau_{(d_2)[k-1-(d_1+a)]}(B,X)_-)\right\} \\[10 pt]
    &&+ \sum_{\substack{\delta_1+ \delta_2 = d \\0 < \delta_1 \leq d+ \min(a,b)}} \tau_{(\delta_1)[k-1]}(X, \sigma_{(\delta_2)[k-\delta_1]}(A,B)) \\
    && +\,\, D_X\sigma_{(d+x)[k+x]}(A,B) + \mathscr A_{A,B} D_A \tau_{(d+a)[k-1]}(B,X).
    \end{eqnarray*}
    \vskip 10 pt

\item For $A  \in \frak g_a$, $B  \in \frak g_b$, $C  \in \frak g_c$, where $a,b,c \geq 0$, and    for $k=d+a+b+c+3$, where $d$ is an integer,  we have
    \begin{eqnarray*}
    \mathfrak S_{A,B,C}\left\{ \sum_{\substack{d_1 + d_2=d \\-\min\{a,b\} \leq d_2 \leq d+c }} \sigma_{(d_1)[k-1]}(\sigma_{(d_2)[k-1-c-d_1]}(A,B),C) -D_C \sigma_{(d+c)[k-1]}(A,B)\right\} =0.
    \end{eqnarray*}


\end{enumerate}

\end{theorem}

Here, $\frak S_{X,Y,Z}F(X,Y,Z)$ denotes the cyclic sum $F(X,Y,Z) + F(Y,Z,X) + F(Z,X,Y)$ and $\mathscr A_{X,Y}G(X,Y)$ denotes the alternating sum $G(X,Y) - G(Y,X)$. Also, we define  $\kappa(\,,\,)_+:=\max\{\kappa(\,,\,), 0\}$ and $\kappa(\,,\,)_- := \min\{\kappa(\,,\,), 0\}$, so that $\kappa(\,,\,) = \kappa(\,,\,)_+ + \kappa(\,,\,)_-$,  and define  $\tau(\,,\,)_{\pm}$ and $ \sigma(\,,\,)_{\pm}$ similarly.

\begin{proof}
 %
 Recall the Bianchi identity:
$$\gamma \circ \gamma - D \gamma =0.$$

(1)  Let $X \in \frak g_x$, $Y \in \frak g_y$ and $Z \in \frak g_z$ with $x,y,z <0$. The $\frak g_{k+x+y+z}$-component of $(\gamma \circ \gamma - D \gamma)(X,Y,Z)$, where $k$ is any positive integer, is
\begin{eqnarray*}
&&\mathfrak S_{X,Y,Z}\left\{ \sum_{ i+j=k }  \gamma_i (\gamma_j (X,Y),Z) -  \sum_{-i+j=k}D_X \gamma_j (Y,Z) \right\} \\[10 pt]
&=&  \mathfrak S_{X,Y,Z}\{\gamma_0 (\gamma_k (X,Y),Z)+ \gamma_k (\gamma_0(X,Y),Z)\} \\
&& +\,\,\mathfrak S_{X,Y,Z} \left\{\sum_{\substack{i+j=k\\[2 pt] 0<i,j <k }}  \gamma_i (\gamma_j(X,Y),Z)  -  \sum_{-i+j=k}D_X \gamma_j (Y,Z)\right\}  .
\end{eqnarray*}
The 1st line is
\begin{eqnarray*}
 \mathfrak S_{X,Y,Z}\left\{\gamma_0(\gamma_k(X,Y),Z) + \gamma_k(\gamma_0(X,Y),Z)  \right\}
  &=& \mathfrak S_{X,Y,Z}\left\{[\gamma_k(X,Y),Z] +  \gamma_k([X,Y],Z)  \right\} \\
  &=& -(\partial \kappa_{[k]})(X,Y,Z)
\end{eqnarray*}
and the 2nd line is
\begin{eqnarray*}
\mathfrak S_{X,Y,Z} \left\{ \sum_{\substack{d_1+ d_2 =k\\ d_1, d_2 >0}} \left\{\kappa_{d_1}(\kappa_{d_2}(X,Y)_-,Z) + \tau_{(d_1)[k+x+y+1]}(\kappa_{d_2}(X,Y)_+, Z) \right\} - D_X \kappa_{k+x}(Y,Z) \right\},
\end{eqnarray*}
from which we get the desired identity. \\

(2) Let $X \in \frak g_x$, $Y \in \frak g_y$ and $A \in \frak g_a$ with $x,y <0$ and $a \geq 0$. The $\frak g_{d+x+y+a}$-component of $(\gamma \circ \gamma - D \gamma)(X,Y,A)$, where $d$ is any integer, is
\begin{eqnarray*}
&& \sum_{\delta_1 + \delta_2 =d}\gamma_{\delta_1}(\gamma_{\delta_2}(X,Y),A) + \mathscr A_{X,Y} \sum_{d_1+d_2 =d} \gamma_{d_1}(\gamma_{d_2}(A,X),Y) \\
&& + \left\{ D_A \gamma_{d+a}(X,Y) + \mathscr A_{X,Y} D_Y \gamma_{d+y} (A,X) \right\}  \\ [10 pt]
 &=& \sum_{\substack{\delta_1+ \delta_2=d\\\delta_1\delta_2 =0}} \gamma_{\delta_1}(\gamma_{\delta_2}(X,Y),A)  + \mathscr A_{X,Y} \sum_{\substack{d_1+ d_2=d\\d_1d_2=0}}  \gamma_{d_1}(\gamma_{d_2}(A,X),Y)    \\
 && + \sum_{\substack{\delta_1+ \delta_2=d\\\delta_1\delta_2 \not=0}} \gamma_{\delta_1}(\gamma_{\delta_2}(X,Y),A)  + \mathscr A_{X,Y} \sum_{\substack{d_1+ d_2=d\\d_1d_2\not=0}}  \gamma_{d_1}(\gamma_{d_2}(A,X),Y)    \\
 &&+ \,\, \left\{ D_A \gamma_{d+a}(X,Y) + \mathscr A_{X,Y} D_Y \gamma_{d+y} (A,X) \right\}.
\end{eqnarray*}
Each line can be written as follows.
\begin{eqnarray*}
\bullet \text{ The 1st line}&=& \gamma_0(\gamma_d(X,Y),A) + \gamma_d([X,Y],A) + \mathscr A_{X,Y} [\gamma_d(A,X),Y] + \gamma_d([A,X],Y) \\
&=& \partial \tau_{(d) [\delta_1^{\sharp}]} (A, \,\cdot\,)(X,Y) -(\rho(A)\kappa_{[d]})(X,Y) + \sigma_0(\kappa_{[d]}(X,Y)_+,A) - [\kappa_{[d]}(X,Y)_+, A] \\
&& + \mathscr A_{X,Y} \tau_{(d)[d_1^{\sharp}]}([A,X]_+,Y),
\end{eqnarray*}
where $\delta_1^{\sharp} =d+a+1 =k$ and $d_1^{\sharp} = d+ a+x+1 = k+x$. Here, we use the formula
$$(\rho(A) \kappa_{[k-a-1]})(X,Y)=[A, \kappa_{[k-a-1]}(X,Y)] - \kappa_{[k-a-1]}([A,X]_{-},Y) - \kappa_{[k-a-1]}(X, [A,Y]_{-} ) .$$

\begin{eqnarray*}
 \bullet  \text{ The 2nd line} &=& \sum _{\substack{\delta_1+ \delta_2=d\\-a \leq \delta_1 <0 \text{ and } \delta_1 \not=0\\(0 <\delta_2 \leq d+z=k-1, \delta_2 \not=d)}}
\gamma_{\delta_1}(\gamma_{\delta_2}(X,Y),A)
+ \mathscr A_{X,Y} \sum_{\substack{d_1+ d_2=d\\0 < d_1 <d\\(0 < d_2 < d)}}  \gamma_{d_1}(\gamma_{d_2}(A,X),Y)    \\
&=& \sum _{\substack{\delta_1+ \delta_2=d\\-a \leq \delta_1 <d \text{ and } \delta_1 \not=0}} \left\{ \tau_{(\delta_1)[\delta_1^{\sharp}]}(\kappa_{\delta_2}(X,Y)_-,A) +\sigma_{(\delta_1)[\delta_1^{2\sharp}]}(\kappa_{\delta_2}(X,Y)_+,A) \right\} \\
&& + \mathscr A_{X,Y} \sum_{\substack{d_1 +d_2 =d \\d_1, d_2 >0}}\left\{ \kappa_{d_1}(\tau_{(d_2)[d_2^{\sharp}]}(A,X)_-, Y) + \tau_{(d_1)[d_1^{\sharp}]}(\tau_{(d_2)[d_2^{\sharp}]}(A,X)_+,Y) \right\},
\end{eqnarray*}
where $\delta_1^{\sharp} = \delta_1 +a <d + a=k-1$, $\delta_1^{2\sharp}=\delta_1 + \delta_2 + x +y+ z+2 < k+x+y+1$, $d_2^{\sharp} = d_2 + a +1 <k$.
\begin{eqnarray*}
\bullet \text{ The 3rd line} \,\,= \,\, D_A \kappa_{k-1}(X,Y) + \mathscr A_{X,Y} D_Y \tau_{(d+y)[(d+y)^{\sharp}]}(A,X), \qquad \qquad \qquad
\end{eqnarray*}
where $(d+y)^{\sharp} =d+y+a+1 = k+y$.

Taking the sum of all terms in three lines, we get the identity (2). \\

(3)
Let $A \in \frak g_a, B \in \frak g_b$ and $X \in \frak g_x$, where $a,b \geq 0$ and $x <0$.
For $d \in \mathbb Z$, set   $k=d+a+b+2 $. The $\frak g_{d+a+b+x}$-component of $(\gamma \circ \gamma - D \gamma)(A,B,X)$ is
\begin{eqnarray*}
&& \sum_{\delta_1 + \delta_2 =d}\gamma_{\delta_1}(\gamma_{\delta_2}(A,B),X) + \mathscr A_{A,B} \sum_{d_1+d_2 =d} \gamma_{d_1}(\gamma_{d_2}(B,X),A) \\
&& + \left\{ D_X \gamma_{d+x}(A,B) + \mathscr A_{A,B} D_B \gamma_{d+b} (X,A) \right\}   \\ [10 pt]
  &=& \sum_{\substack{\delta_1+ \delta_2=d\\\delta_1  =0}} \gamma_{\delta_1}(\gamma_{\delta_2}(A,B),X) 
     \\
  && + \sum_{\substack{\delta_1+ \delta_2=d\\\delta_1  \not=0}} \gamma_{\delta_1}(\gamma_{\delta_2}(A,B),X)  + \mathscr A_{A,B} \sum_{ d_1+ d_2=d}  \gamma_{d_1}(\gamma_{d_2}(B,X),A)    \\
  &&+ \,\, \left\{ D_X \gamma_{d+x}(A,B) + \mathscr A_{A,B} D_B\gamma_{d+b} (X,A) \right\}.
\end{eqnarray*}

Each line can be written as follows.
\begin{eqnarray*}
\bullet \text{ The 1st line} 
& =& [\gamma_d(A,B),X]  \\
&=& [\sigma_{(d)[k]}(A,B), X], \text{  where }
 d +b+1=   k-1-a <k  \\
%
\bullet \text{ The 2nd line} &=&  \sum_{\substack{\delta_1+ \delta_2 = d \\0 < \delta_1 \leq d+ \min(a,b)}} \tau_{(\delta_1)[k-1]}(  \sigma_{(\delta_2)[k-\delta_1]}(A,B), X) \\
&& + \mathscr A_{A,B} \sum_{\substack{d_1+d_2=d \\0 \leq  d_2 \leq d+a 
    }}
    \left\{ \sigma_{(d_1)[k+x]}(  \tau_{(d_2)[k-1-(d_1+a)]}(B,X)_+,A) \right. \qquad \qquad \qquad \qquad \qquad\\[-15 pt]
     && \qquad \qquad \qquad \qquad \qquad \qquad\qquad \qquad \left.+ \,\, \tau_{(d_1)[k-1-(d_2+b)]}( \tau_{(d_2)[k-1-(d_1+a)]}(B,X)_-, A)\right\} \\[10 pt]
%
\bullet \text{ The 3rd line} &=& D_X\sigma_{(d+x)[k+x]}(A,B) + \mathscr A_{A,B} D_A \tau_{(d+a)[k-1]}(B,X).
\end{eqnarray*}
Here, we use Proposition \ref{prop:degree of structure function} to get the bounds $0 < \delta_1 \leq d+ \min(a,b)$ and $ 0 \leq  d_2 \leq d+a$. For example,
\begin{enumerate}
\item [$\bullet$] if $-\delta_2=\delta_1 -d >\min(a,b)$, then  $\sigma_{(\delta_2)[k-\delta_1]}(A,B)=0$ by Proposition \ref{prop:degree of structure function} (3). 
\item [$\bullet$] if $-d_1=d_2 -d >a$, then $\sigma_{(d_1)[k+x]}(\,\cdot\,, A) =0$ by Proposition \ref{prop:degree of structure function} (3) and $\tau_{(d_1)[k-1-(d_2+b)]}(\,\cdot\,, A) =0$ by Proposition \ref{prop:degree of structure function} (2). In fact, the latter vanishes if $-d_1>0$.

\end{enumerate}

\noindent
This completes the proof of Theorem \ref{thm:fundamental identities} (3). \\

  (4) Use
$$\mathfrak S_{A,B,C} \left\{\sum_{\substack{d_1 + d_2 =d\\-min\{a,b\} \leq d_2 \leq d+c}} \sigma_{(d_1)[d_1^{\sharp}]}(\sigma_{(d_2)[d_2^{\sharp}]}(A,B),C) - D_C \sigma_{(\delta)[\delta^{\sharp}]}(A,B) \right\}=0.$$
Now for $d \in \mathbb Z$, set $k=d+a+b+c+3$. Put $u=a+b+d_2$ and $w=a+b+c+d$.
Then $d_1 = w - u -c \geq -c$ and  $d_2 = u-a-b \geq -\min\{a,b\}$.
%
Furthermore, we have
\begin{eqnarray*}
 d_1^{\sharp} &=& w+2 =d+a+b+c+2 \\
  &=& k-1 \\
 d_2^{\sharp} &=& u+2 = d_2 +a +b+2 \\
 &=& k-d_1-c-1 \\
 &=& k-1- (c+d_1) \\
 &\leq & k-1 \\
 \delta &=& w-a-b \\
 &=& d+c \\
 \delta^{\sharp}&=& w+2 \\
 &=&k-1.
\end{eqnarray*}
This completes the proof of Theorem \ref{thm:fundamental identities}.
\end{proof}

Let us see more concretely the fundamental identities for lower orders. \\

\noindent {\bf Order 0.} We see that $\kappa_0$ is the bracket of $\frak g_-$ and the identity (1) turns out to be
$$\partial \kappa_0=0$$
which says that $\kappa_0$ satisfies the Jacobi identity.

Nest recall that $\tau_{[k]}$ is the $\Hom(\frak g_+, \Hom(\frak g_-, \frak g)_{k-1})$-component of $\gamma$ where $\frak g_+= \oplus _{i \geq 0} \frak g_i$, so that the identity (2) is trivial. The structure function $\sigma_{[k]}$ is the $\Hom(\wedge \frak g_+, \frak g_{k-1})$-component of $\gamma$, and therefore $\sigma_{[0]}=0$ and the identity (3) is trivial. \\

\noindent{\bf Order 1.} We have
$$\partial \kappa_{1}=0$$
because the term $D\kappa$ vanishes in this case.

The identity (2) becomes
$$\partial \tau_{[1]}(A_0, \,\cdot\,) = D_{A_0}\kappa_0 + \rho(A_0)\kappa_0 \text{ for } A_0 \in \frak g_0.$$
Therefore, $\partial \tau_{[1]}(A_0, \,\cdot\,) =0$. We see also that the right hand side of the above identity vanishes. We see also that $\sigma_{[1]}=0$. \\

\noindent {\bf Order 2.} In this case we have
$$\partial \kappa_2 = \kappa_1 \circ \kappa_1 - \sum_{p<0} D_{\pi_p}\kappa_{2+p}  \qquad \qquad (*1) $$
where
\begin{eqnarray*}
\kappa_1 \circ \kappa_1 (X,Y,Z)&:=& \mathfrak S_{X,Y,Z} \kappa_1(\kappa_1(X,Y),Z) \\
(D_{\pi_p}\kappa_{2+p} )(X,Y,Z)&:=& \mathfrak S_{X,Y,Z}D_{\pi_pX} \kappa_{2+p}(Y,Z).
\end{eqnarray*}

For $\tau_{[2]}$ we have
$$\partial \tau_{[2]}(A_0, \,\cdot\,) =D_{A_0}\kappa_1 + \rho(A_0)\kappa_1 \text{ for } A_0 \in \frak g_0.  \qquad \qquad (*2) $$
We have also
$$\partial \tau_{[2]}(A_1, \,\cdot\,) = D_{A_1}\kappa_1 + \rho(A_1)\kappa_0 \text{ for } A_1 \in \frak g_1.$$
But $\tau_{[2]}(A_1, X)=[A_1,X]$ and thus the both sides vanish. \\

We remark that the identities $(*1)$ and $(*2)$ appeared already in \cite{CII2} in his style. Singer-Sternberg gave its interpretation in \cite{SS65} especially when the structure function is constant. \\

Theorem \ref{thm:fundamental identities} has many important implications. We give here as corollaries some of them which follow immediately from it and will be used in later sections.

 \begin{corollary}  \label{cor:vanishing of kappa} $\,$

 \begin{enumerate}
  \item $\partial \kappa_{[1]}=0$.
  \item

 For $k \geq 2$, 
 if $\kappa^{[k-1]}$ is flat,
 then $\partial \kappa_{[k]}=0$ and $\partial  \tau _{[k]}(A,\,\cdot\,)=0$ for $A \in \frak g_a$ where $0 \leq a \leq k$.
      \end{enumerate}

 \end{corollary}
 \begin{proof}
It follows from  Theorem \ref{thm:fundamental identities} (1) and (2).
  \end{proof}

\begin{corollary} \label{cor: tau flat implies sigma flat}
 If the structure function $\tau^{[k-1]}$ is flat for some integer $k$, then so is $\sigma^{[k]}$.
\end{corollary}

\begin{proof}
Note first that $\sigma^{[1]}$ is flat and $\sigma_{[2]} = \sigma_{(0)[2]}$.
Applying Theorem \ref{thm:fundamental identities} (3) for $d=0$ and $a=b=0$, we have
$$[\sigma_{(0)[2]}(A,B),X] = [A,[B,X]] - [B, [A,X]] =[[A,B],X].$$
Thus $\sigma_{(0)[2]}(A,B) =[A,B]$, which shows that $\sigma^{[2]}$ is flat.

Now assume that $k \geq 3$ and $\tau^{[k+1]}$ is flat. Then by induction, we deduce that $\sigma^{[k-1]}$ is flat.

Let $A \in \frak g_a$, $B \in \frak g_b$ and $X \in \frak g_x$ with $a,b \geq 0$ and $x <0$ and put $d=k-(a+b+2)$, to which we apply Theorem \ref{thm:fundamental identities} (3). Then we see that $[\sigma_{(d)[k]}(A,B),X]$ is written as a sum of the following forms of terms:
$$\sigma_{(d_1)[d_1^{\sharp}]} \circ\tau_{(d_2)[d_2^{\sharp}]},  \quad \tau_{(d_1)[{d_1^{\sharp}}']} \circ  \tau_{(d_2)[{d_2^{\sharp}}']}, \quad \tau_{(\delta_1)[\delta_1^{\sharp}]}\circ \sigma_{(\delta_2)}$$
where $d_1+d_2=\delta_1 +\delta_2 =d$ and $\delta_1 >0$ and $d_1^{\sharp}, d_2^{\sharp},  {d_1^{\sharp}}', {d_2^{\sharp}}', \delta_1^{\sharp} <k$, as well as vanishing terms:
$$D_X\sigma_{[m]}, \quad  D_A\tau_{[i]}, \quad  D_B \tau_{[j]},$$
where $i,j,m <k$.

Among them non-vanishing terms can occur only from $\sigma_{(0)}\circ  \tau_{(0)}$ and $\tau_{(0)} \circ \sigma_{(0)}$, from which it follows that
$$\sigma_{(d)[k]}(A,B) =0 \text{ if } d \not=0$$
and
$$[\sigma_{(0)[k]}(A,B),X] = [A,[B,X]] -[B,[A,X]], $$
that is, $\sigma_{(0)[k]}(A,B) = [A,B]$. Therefore, $\sigma^{[k]}$ is flat.
\end{proof}

\begin{corollary} \label{cor:formula of tau}

If $\tau^{[k-1]}$ is flat,
then
\begin{eqnarray*}
 \partial \tau_{[k]}(A,\,\cdot\,))(X,Y) =  \left(\rho(A)\kappa_{[k-a-1]}  + D_A \kappa_{[k-1]} \right)(X,Y)
 \end{eqnarray*}
 for $A \in \frak g_a$, $X \in \frak g_x$, $v \in \frak g_y$ with $ 0 \leq a \leq k-2$ and $x,y <0$.

\end{corollary}

\begin{proof}
By Theorem \ref{thm:fundamental identities} (2) and 
 %
Corollary \ref{cor: tau flat implies sigma flat} we get the desired result.
\end{proof}

\begin{corollary}  \label{cor:simple observation} Let $k \geq 2$. If $\tau^{[k-1]}$ is constant, then
\begin{enumerate}
\item $\sigma^{[k]}$ is constant;
\item $\sigma^{[k]}_{(d)}=0$ for $d <0$;
\item $\sigma^{[k]}_{(0)}(A,B) = [A,B]$ for $A \in \frak g_a, B \in \frak g_b$, where $a \geq 0, b \geq 0$ and $a+b \leq k-2$.
\end{enumerate}

\end{corollary}

\begin{proof}
The statement (1) follows immediately from the third fundamental identity (Theorem \ref{thm:fundamental identities} (3)). The assertion (2) can be verified by induction as follows. Let $d$ be a negative integer and assume that
\begin{enumerate}
\item [(i)] $\sigma^{[k]}_{d'} =0$ for $d' <d$
\item [(ii)] $\sigma^{[k']}_{(d)} =0$ for $k'<k$.
\end{enumerate}

Now let $a,b$ be non-negative integer such that $k=d+a+b+2$. Let $A \in \frak g_a$ and $B \in \frak g_b$ and $X \in \frak g_x$ ($x<0$). Then by Theorem \ref{thm:fundamental identities} (3), we have
 \begin{eqnarray*}
    &&[\sigma_{(d)[k]}(A,B),X] \\
    &=& \mathscr A_{A,B} \sum_{\substack{d_1+d_2=d \\0 \leq d_2 \leq d+a 
    }}
    \left\{ \sigma_{(d_1)[k+x-d_2]}(A, \tau_{(d_2)[k-1-(d_1+a)]}(B,X)_+) \right. \qquad \qquad \qquad \qquad \qquad\\[-15 pt]
     && \qquad \qquad \qquad \qquad \qquad \qquad\qquad \qquad \left.+ \,\, \tau_{(d_1)[k-1-(d_2+b)]}(A, \tau_{(d_2)[k-1-(d_1+a)]}(B,X)_-)\right\} \\[10 pt]
    &&+ \sum_{\substack{\delta_1+ \delta_2 = d \\0 < \delta_1 \leq d+ \min(a,b)}} \tau_{(\delta_1)[k-1]}(X, \sigma_{(\delta_2)[k-\delta_1]}(A,B)) \\
    && +\,\, D_X\sigma_{(d+x)[k+x]}(A,B) + \mathscr A_{A,B} D_A \tau_{(d+a)[k-1]}(B,X).
    \end{eqnarray*}
By the inductive assumption (i),
\begin{eqnarray*}
\sigma_{(\delta_2)}(A,B) &=& \sigma_{(\delta_2)[\delta_2+a+b+2]}(A,B) \\
&=& \sigma_{(\delta_2)[k-\delta_1]}(A,B)  =
  0 \text{ for } \delta_2<d \\
\sigma_{(d+x)}(A,B) &=& \sigma_{(d+x)[k+x]}=0.
\end{eqnarray*}
Therefore, the 2nd and the 3rd terms of the right hand side of the above identity vanish.

For the 1st term, we have
\begin{eqnarray*}
(\tau_{d_1} + \sigma_{d_1})(\tau_{d_2}(A,X), B) &=& \sigma_{d_1}(\tau_{d_2}(A,X), B) \\
&=& \sigma_{d_1[d_1+d_2+a+b+x]}(\tau_{d_2}(A,X),B) \\
&=& \sigma_{(d)[k-2+x]}([A,X], B) \text{ by (i)} \\
&=& 0 \text{ by (ii)}.
\end{eqnarray*}
Hence the 1st term vanishes.

The last term is
\begin{eqnarray*}
D_A\tau_{d+a}(B,X) &=& D_A\tau_{(d+a)[d+a+b+1]}(B,X) \\
&=& D_A\tau_{(d+a)[k-1]}(B,X).
\end{eqnarray*}
But by the assumption  $\tau^{[k-1]}$ is constant, therefore, the last term also vanishes. Hence $\sigma_{(d)}^{[k]}=0$, which proves (2).

By a similar argument we have
\begin{eqnarray*}
[\sigma^{[k]}_0(A,B), X] =\sigma_0^{[k]}([A,X], B) + \sigma^{(k)}_0(A, [B,X]),
\end{eqnarray*}
from which follows the assertion (3).
\end{proof}

\begin{remark}
In general $\sigma_d^{[k]}$ does not necessarily vanish even for $d <0$. We have $\gamma_{11}^1=$ and $\gamma_{22}^2=0$ but $\gamma_{21}^2$ does not always vanish.
\end{remark}






%










\section{Invariants} \label{sec:Invariants}

On the basis of Section \ref{sect:normal step prolongation} and Section \ref{sec:inductive formula}, we show how to obtain
the invariants of an arbitrarily given proper geometric structure.

\subsection{Fundamental system of invariants} \label{sec:complete system of invariants} $\,$

Let $Q^{(k)}$ be a proper geometric structure of type $(\frak g_-, G_0, \cdots, G_k)$. Choose complementary subspaces $W=\{W^1_{\ell}, W^2_{\ell +1}\}_{\ell \geq k}$ as in Section \ref{sect:normal step prolongation} and let $\mathscr S_{W}Q^{(k)}$ be the $W$-normal complete prolongation of  $Q^{(k)}$ and $\gamma$ be its structure function, which is a $\Hom(\wedge^2E, E)$-valued function on $\mathscr S_WQ^{(k)}$. Define $\Hom(\otimes^mE, \Hom(\wedge^2 E, E))$-valued function $D^m \gamma$ by
\begin{eqnarray*}
D^0\gamma &=& \gamma \\
d(D^{m-1}\gamma) &=& (D^m\gamma) \circ \theta \text{ for } m >0.
\end{eqnarray*}
Then $\{D^m\gamma\}_{m \geq 0}$ forms a set of invariants of $Q^{(k)}$, that is, if two proper geometric structures $Q^{(k)}$ and $\overline{Q}^{(k)}$ are isomorphic, then $D^m\gamma(z) = \overline{D}^m \overline{\gamma}(\overline{z})$ for any $z \in \mathscr S_WQ^{(k)}$ and for the corresponding point $\overline{z} \in \mathscr S_W\overline{Q}^{(k)}$.

 Let us call $\{D^m\gamma\}_{m \geq 0}$ the {\it complete system of invariants} of $Q^{(k)}$ determined by the $W$-normal prolongation since these invariants $\{D^m\gamma\}_{m \geq 0}$ completely determine the formal equivalence class of geometric structure $ Q^{(k)}$ (Theorem \ref{thm:Theorem 6.1 formal equivalence}).


Now Theorem \ref{thm:fundamental identities}  gives a recursive formula  how to determine $\gamma$. To see it more closely, we give the following.

 \begin{definition}
Let $Q^{(\ell)}$ be a proper geometric structure of type $\frak g[\ell]$ and let $\frak g$ be the prolongation of $\frak g[\ell]$. We say that
$Q^{(\ell)}$ is {\it quasi-involutive} if
$$H^1_r(\frak g_-, \frak g) =0 \text{ and } H^2_{r+1}(\frak g_-, \frak g)=0 \text{ for } r \geq \ell.$$
 \end{definition}

By Theorem \ref{vanishing cohomology of higher degree} there is an integer $\nu$ ($\geq k$) such that $\mathscr S_W^{(\ell)}Q^{(k)}$ is quasi-involutive for $\ell \geq \nu$.

By the fundamental identities (Theorem \ref{thm:fundamental identities}), if $\ell > \nu $, then $\kappa_{[\ell]}$, $\tau_{[\ell]}$, $\sigma_{[\ell]}$ and all their covariant derivatives
are uniquely determined by $\{\kappa_{[j]}, \tau_{[j]}, \sigma_{[j]}; j \leq \nu\}$ and their covariant derivatives
through polynomial functions determined only by the graded Lie algebra $\mathfrak g$.
Indeed, 
 if  $H^2_{\ell}(\frak g_-, \frak g)=0$,   then the restriction of $\partial$ to $W_{\ell}^2$  is injective. Since $\kappa_{[\ell]}$ has its value in $W_{\ell}^2$, $\partial \kappa_{[\ell]}$ determines $\kappa_{[\ell]}$. Similarly, if  $H^1_{\ell-1}(\frak g_-, \frak g)=0$,  then $\partial \tau_{[\ell]}(A,\,\cdot\,)$ determines  $\tau_{[\ell]}(A,\,\cdot\,)$.

 In this regard, we say that $\{ D^m\gamma^{[\nu]}\}_{m \geq 0}$ 
 forms {\it a fundamental system of invariants} of
 $Q^{(k)}$.

More specifically, we define the {\it set of essential invariants} of $Q^{(k)}$ as the  set
\begin{eqnarray*}
 \{D^m \tau_{[i]}: i \in I^1, m \geq 0 \}
 \cup \{D^m \kappa_i: i \in I^2, m \geq 0\}
\end{eqnarray*}
where
\begin{eqnarray*}
I^{1} &:=& \{ i \in \mathbb Z_{\geq 0} : H^1_{i-1} (\frak g_-,\frak g) \not=0 \} \\
I^{2}&:=& \{ i \in \mathbb Z_{\geq 0} : H^2_i (\frak g_-, \frak g) \not=0 \}.
\end{eqnarray*}

Summarizing the above discussion we have the following.

\begin{theorem} \label{thm:complete system of invariants}
Let $Q^{(k)}$ be a proper geometric structure and $\mathscr S_WQ^{(k)}$ its $W$-normal prolongation. If $\mathscr S_W^{(\ell)}Q^{(k)}$ is quasi-involutive,
then $\{D^m\gamma^{[\ell]} \}_{m\geq 0}$ forms a fundamental system of invariants of $Q^{(k)} $.
 Furthermore, the set of essential invariants of $Q^{(k)}$ forms a fundamental system of invariants of $Q^{(k)}$.
\end{theorem}

We remark   that, though
$\kappa _{\ell}, \tau_{[\ell]} $ are functions on
$\mathscr S^{(\ell )}_W Q^{(k)} $,
their covariant derivatives
$ D^i \kappa _{\ell},  D^i\tau_{[\ell]} $
are functions on
 $\mathscr S^{(\ell  + \mu i)}_W Q^{(k)} $
 and
the essential invariants
$\{
D^i \kappa _{\ell},  D^i\tau_{[\ell]}
\} $
should be regarded as functions on
$\mathscr S_W Q^{(k)}$.

Let us see what the above theorem says, in particular, for a geometric structure
$Q^{(0)} $ of order 0.

\begin{corollary}
Let $Q^{(0)}$ be a geometric structure of order 0.
If $\mathscr S_W^{(\ell)}Q^{(0)}$ is quasi-involutive, then $$ \kappa_1, \cdots, \kappa _{\ell}$$
form  a fundamental system of invariants.
\end{corollary}

\begin{proof}
 In this case it holds that
$H^1_r (\mathfrak g _-, \mathfrak g) = 0 $
for $r > 0$. Therefore the   essential invariants are  composed of
$$ \kappa_0, \kappa_1, \dots , \kappa_{\ell} , \tau _{[0]}, \tau_{[1]} $$
and their covariant derivatives.
Now  $\tau_{[0]} = 0$
and $\kappa_0, \tau_{[1]} $
comes from the bracket of the graded Lie algebra $\mathfrak g$.
Therefore non-trivial invariants can appear only from  $ \kappa_1, \cdots, \kappa _{\ell}$.
\end{proof}

This is a well-known phenomenon  for   the case of Cartan connection. It is remarkable that the same folds for the step prolongation  $\mathscr S _W Q^{(0)}$.


\subsection{Involutive geometric structures} $\,$

Again from Theorem \ref{thm:fundamental identities} we have the following important consequence:

\begin{proposition}
Let $Q^{(k)}$ be a proper quasi-involutive geometric structure and $Q$ be $W$-normal complete step prolongation of $Q^{(k)}$. If the structure function $\gamma^{[k]}$ is constant, then the structure function $\gamma$ is constant.
\end{proposition}

\begin{definition} \label{def:involutive}
A proper geometric structure $Q^{(k)}$ is called {\it involutive} if

\begin{enumerate}
\item $Q^{(k)}$ is quasi-involutive.
\item The structure function $\gamma^{[k]}$ is constant.
\end{enumerate}
\end{definition}


If $Q^{(k)}$ is involutive, we have seen that the structure function $\gamma$ of $Q$ is constant, but not only this, we see by Corollary \ref{cor:simple observation} that $\sigma_{(d)}=0$ for $d<0$ and that $\tau_{(0)}$ and $\sigma_{(0)}$ coincide with the bracket of the graded Lie algebra $ \frak g$. Then we see that $(\frak g, \gamma)$, where $\gamma \in \Hom(\wedge^2 \frak g, \frak g)$ is the structure function of $Q$, becomes a transitive filtered Lie algebra $L$ such that $\gr L \cong \frak g$.
To summarize,  we have the following proposition.

\begin{proposition}
If a proper geometric structure $Q^{(k)}$ of type $\frak g[k]$ is   involutive, then the structure function $\gamma$ of the $W$-normal complete step prolongation $Q$ of  $Q^{(k)}$ defines on the prolongation $\frak g$ of $\frak g[k]$ a structure of transitive filtered Lie algebra $L=(\frak g, \gamma)$ such that $\gr L \simeq \frak g$.
\end{proposition}

For a proper geometric structure $Q$, let $\mathrm{LocAut}(Q)$ be the Lie pseudo-group of all local isomorphisms $\varphi$ of $Q$, $\varphi$ being defined on an open set of the base space $M$ of $Q$, and let $\mathrm{InfAut}(Q)$ be the Lie algebra sheaf on $M$ of all infinitesimal automorphisms $X$ of $Q$, $X$ being a local vector field on $M$ such that $\mathrm{Exp} X \in \mathrm{LocAut}(Q)$.

If $Q$ is analytic and $\gamma$ is constant, then $\mathrm{LocAut}(Q)$ and $\mathrm{InfAut}(Q)$ are transitive on $M$ provided that $M$ is connected (Theorem \ref{thm:involutive local equivalence}), so that   each stalk $\mathcal L_x$ of $\mathrm{InfAut}(Q)$ at $x \in M$  is isomorphic to each other. We see that the completion $L_x$ of $\mathcal L_x$ with respect to the natural filtration $\{\mathcal L^p_x\}$ turns out to be   a transitive filtered Lie algebra, which we call the {\it formal algebra} of $\mathrm{InfAut}(Q)$. 


\subsection{Transitive models of geometric structures} $\,$

In the previous subsection we have seen that an involutive geometric structure (or a geometric structure of constant structure function) yields a transitive filtered Lie algebra. In this subsection we  consider the converse.

Let $(L, \{L^p\})$ be a transitive filtered Lie algebra and $\frak g=\oplus \frak g_p$ its associated graded Lie algebra $\gr L$. Choose complementary subspaces $\{\underline{\frak g}_p\}_{p \in \mathbb Z}$ such that
$$L^p = \underline{\frak g}_p \oplus L^{p+1},$$
and then define  $\underline{\gamma} \in \Hom(\wedge^2 \underline{\frak g}, \underline{\frak g})$ by
$$ \underline{\gamma}(X,Y): =[X,Y] \text{ for } X,Y \in \underline{\frak g}=\oplus \underline{\frak g}_p.$$
Then identifying $\underline{\frak g}$ with $\frak g$, we define $\gamma \in \Hom(\wedge^2 \frak g, \frak g)$ to be the bilinear map corresponding to $\underline{\gamma}$.

\begin{theorem} [The 3rd fundamental theorem of Lie] \label{thm:Cartan-Lie} Let $L$ be a transitive filtered Lie algebra and $\frak g=\gr L$ and $\gamma \in \Hom(\wedge^2 \frak g,\frak g)$ as above. Then there exists   an involutive geometric structure  $Q_L$ with constant structure function equal to $\gamma$ such that the (formal) Lie algebra of $\mathrm{InfAut}(Q_L)$ is $L$

\end{theorem}

This is a local version of Lie's third fundamental theorem that for a Lie algebra $\frak g$ there exists a Lie group $G$ corresponding to $\frak g$, which is well-known if $\frak g$ is finite dimensional.
In the infinite dimensional case, however, we have to state it locally as in the above theorem because we have no satisfactory formulation of infinite dimensional Lie groups. The involutive geometric structure plays a role of local Lie group.

In \cite{C04}  and then in \cite{CII2}, Cartan proved the following theorem in a classical form: \\

\noindent {\bf Theorem (C).} {\it Let $(V,\frak g_0, c)$ be a triple consisting of a finite dimensional vector space $V$, a Lie subalgebra $\frak g_0 \subset \frak{gl}(V)$ and $c \in \Hom(\wedge^2 V, V)$ which satisfy
\begin{enumerate}
\item[(i)] $c \circ c \in \partial \Hom(\wedge^2 V, \frak g_0)$
\item[(ii)] $\rho(A)c \in \partial \Hom(V, \frak g_0)$ for $A \in \frak g_0$
\item[(iii)] $\frak g_0$ is involutive, that is, $H_r^1(V, \frak g_0) =H_{r+1}^2(V, \frak g_0) =0$ for $r >0$,
\end{enumerate}
where the coboundary operator $\partial$ and the cohomology group $H^p_i(V,\frak g_0)$ are those defined by the Spencer complex associated to the prolongation of $(V, \frak g_0)$. Then there exists an analytic manifold $\widetilde M$ equipped with a $V$-valued 1-form $\theta$ and $\frak g$-valued 1-form $\pi$ on $\widetilde M$ satisfying
\begin{enumerate}
\item $d \theta + \frac{1}{2} c(\theta, \theta) + \pi \wedge \theta =0$
\item $(\theta + \pi)_z: T_z \widetilde M \rightarrow V \oplus \frak g$ is an isomorphism for all $z \in \widetilde M$.
\end{enumerate}
}
\vskip 10 pt

The above formulation of Theorem (C) and its rigorous proof based on Cartan-K\"ahler Theorem, may be found in Part 1 of \cite{M69}. It is also proved there Theorem \ref{thm:Cartan-Lie} in the case where the depth of the filtered Lie algebra $(L, \{L^p\})$ is one as follows.

Let $k$ be an integer such that
$$H^1_r((\gr L)_-, \gr L) = H^2_{r+1}((\gr L)_-, \gr L) =0 \text{ for } r >k.$$
Let $\frak g $ be $\gr L$ and  $E^{(k-1)}$ be $\frak g/F^k$. Denote by $\gamma^{(k-1)} \in \Hom(\wedge ^2E^{(k-1)}, E^{(k-1)})$  the truncated bracket of $L$ with respect to an identification of $L/F^k$ with $E^{(k-1)}$.
Then $(E^{(k-1)}, \frak g_k, \gamma^{(k-1)})$ satisfies the assumptions (i) -- (iii) of Theorem (C).
  By Theorem (C), there exist  a manifold $U^{(k)}$ equipped with an $E^{(k)}$-valued 1-form $\theta^{(k-1)}$ and a $\frak g_k$-valued 1-form $\omega_k$ satisfying the conditions (1) and (2) in the theorem.
  The filtration of $E^{(k-1)}$ then induces a filtration of $U^{(k)}$ and it is shown that $U^{(k)}$ can be realized as an open set of an involutive geometric structure of order $k$. \\

\begin{proof} [Proof of Theorem  \ref{thm:Cartan-Lie}]
  For a transitive filtered Lie algebra $(L, \{L^p\} )$ of arbitrary depth, we define a filtration $\{\overline{L}^p\}$ of $L$ of depth one by setting
  \begin{enumerate}
  \item[$\bullet$] $\overline{L}^p = L$ for $p<0$ and $\overline{L}^0 = L^0$;
  \item [$\bullet$] for $p >0$, $\overline{L}^{p+1} = \{ X \in \overline{L}^p: [X,L] \subset \overline{L}^p \}$.
  \end{enumerate}
Then we see that
\begin{enumerate}
\item [(a)] $[\overline{L}^p, \overline{L}^q] \subset \overline{L}^{p+q}$;
\item [(b)] for any $L^p$ there exists $q$ such that $L^p \supset \overline{L}^q$;
\item [(c)] for any $\overline{L}^q$ there exists $p$ such that $\overline{L}^q \supset L^p$.
\end{enumerate}
Thus $(L, \{\overline{L}^p\})$ is a transitive filtered Lie algebra of depth one.
Applying  the result in the above to $(L, \overline{L}^p)$, we get  $(U^{(k)}, \theta^{(k-1)}, \omega_k)$.

Let $\mathrm{LocAut}(U^{(k)}, \theta^{(k-1)})$ be the pseudo-group of local diffeomorphisms of $U^{(k)}$ which preserve $\theta^{(k-1)}$.
Note that it is transitive on $U^{(k)}$ by Theorem \ref{thm:involutive local equivalence}.
On account of the structure equation for $\theta$ and the filtration on $U^{(k)}$, we may assume that there is a fibration $U^{(k)} \stackrel{\pi}{\longrightarrow} M$ with $\dim M = \dim L/L^0$ and $\mathrm{LocAut}(Q^{(k)}, \theta^{(k-1)})$ to $M$ to define a pseudo-group $\mathcal P$ on $M$ isomorphic to $\mathrm{LocAut}(U^{(k)}, \theta^{(k-1)})$.
Moreover we see that there is a filtration $\{F^p\}$ on $M$ invariant by $\mathcal P$ defined from $L^p$ ($p \leq 0$).

Now consider $\mathscr S^{(0)}(M,F)$ and lift $\mathcal P$ to $\mathscr S^{(0)}(M,F)$ to get a pseudo-group $\mathcal P_{\mathscr S^{(0)}}$ on $\mathscr S^{(0)}(M,F)$. Then take one orbit $Q^{(0)}$ and let $\mathcal P_{Q^{(0)}}$ be the restriction of $\mathcal P_{\mathscr S^{(0)}}$ to $Q^{(0)}$. Next consider $\mathscr S^{(1)}Q^{(0)}$ and the lift $\mathcal P_{\mathscr S^{(1)}}$ of $\mathcal P_{Q^{(0)}}$ and take an orbit $Q^{(1)}$. Repeating this process, we obtain an involutive geometric structure $Q^{(\ell)}$ after a finite number of steps. Then by step prolongation of $Q^{(\ell)}$ we get an involutive geometric structure $Q^{(\infty)} = \mathscr S_WQ^{(\ell)}$.

In each of the above construction we can choose an orbit $Q^{(i)}$ so that the structure function of $Q^{(\infty)}$ coincides with the prescribed $\gamma$.
This completes the proof of Theorem  \ref{thm:Cartan-Lie}.
\end{proof}

The involutive geometric structure $Q_L=Q^{(\infty)}$ thus constructed may be called the {\it transitive model corresponding to} $(L, \gamma)$. Then by Theorem \ref{thm:involutive local equivalence} which we will prove in Section \ref{sect:equivalence},  we have:

\begin{theorem} The transitive model $Q_L$ is uniquely determined up to local analytic isomorphism.
\end{theorem}

\section{Equivalence problems} \label{sect:equivalence}

In the previous section, we have studied how to obtain the invariants of a geometric structure. In this section we study the converse, that is, we consider whether the invariants that we have found in  Section \ref{sec:Invariants} is sufficient to determine the equivalence.


%

\subsection{Formal equivalence} $\,$

 Let $X$ be a differentiable or analytic manifold and $(x^1, \dots, x^n)$ be a local coordinate system. Given a differentiable function $f$ in a neighborhood of a point $p \in X$, the series of derivatives $\left\{ \frac{\partial^{i_1 + \dots + i_n} f}{\partial (x^1)^{i_1} \dots \partial (x^n)^{i_n}} (p)\right\}$ determines a {\it formal structure at} $p$, which we denote by $f[[X,p]]$ or $f[[p]]$. 

Let $Y$ be another manifold with a coordinate system $(y^1, \dots, y^n)$ and $h[[Y,q]]$ be a formal structure at $q$. We say $f[[X,p]]$ and $h[[Y,q]]$ are {\it formally equivalent} and write $f[[X,p]] \cong h[[Y,q]]$ if there exists a formal (power series) invertible map $\varphi:X \rightarrow Y$ given by a formal power series $x^i=\varphi^i(y^1, \dots, y^n)$ ($i=1, \dots, n$) such that $\varphi^*h[[Y,q]] = f[[X,p]]$.

Let $Q^{(k)}$ be a geometric structure of type $\frak g[k]$. 
Denote by $Q^{(\ell)}$ ($\ell >k$) and $Q$ the $W$-normal step prolongation $\mathscr S^{(\ell)}_W Q^{(k)} $ and $\mathscr S_W Q^{(k)} $.
%
Then the formal structure $f[[ z]]$ at $z \in Q$ may be given by the series of covariant derivatives evaluated at $z$:
$$\{ (D^if)(z), i=0,1,2, \dots \}. $$

Furthermore, all the relations that the derivatives $(D^if)(z)$ should satisfy are uniquely determined from $\{ (D^i\gamma)(z)\}$. Indeed, let $\{\theta^1, \dots, \theta ^n\}$ be an absolute parallelism with
$$d \theta^k + \frac{1}{2} \sum_{i,j} \gamma^k_{ij} \theta^i \wedge \theta^j =0, $$
and, for a function $f$, define
\begin{eqnarray*}
df&=& \sum_i f_i \theta^i \\
df_i &=& \sum_{i,j} f_{ij} \theta^j \qquad etc.
\end{eqnarray*}
Then all the relations among $f_{ij \dots k}$ come  from $d^2f=0, d^2f_i=0, ...$.

From $d^2f=0$ it follows
$$f_{ij} - f_{ji} =\sum_k \gamma^k_{ij} f_k.$$
Similarly all the relations among $f_{ij\dots k}$ are determined by $\gamma^k_{ij}$ and their derivatives.

\begin{theorem}\label{thm:Theorem 6.1 formal equivalence}
Let $Q^{(k)}$ and $\overline{Q}^{(k)}$ be proper geometric structures of the same type, and let $\mathscr S_WQ^{(k)}$ and $\mathscr S_W\overline{Q}^{(k)}$ be their $W$-normal complete prolongations. 
Let $(z, \overline{z}) \in \mathscr S_WQ^{(k)} \times \mathscr S_W\overline{Q}^{(k)}$.  If $(D^i \gamma)(z) = \overline{D}^i \overline{\gamma}(\overline{z})$ for all $i \geq 0$, then there exists a formal isomorphism
$$\Phi^{(k)}:[[Q^{(k)}, z^{(k)}]] \rightarrow [[\overline{Q}^{(k)}, \overline{z}^{(k)}]]$$
where $(z^{(k)}, \overline{z}^{(k)}) \in Q^{(k)} \times \overline{Q}^{(k)}$ is the projection of $(z, \overline{z})$.
\end{theorem}

\begin{proof}
Let $Q$ denote $\mathscr S_WQ^{(k)}$ and $[[Q, z]]$ the formal power series ring of $Q$. Define
$$\varepsilon:f \in [[Q,z]] \mapsto   \left((D^if)(z)\right)_{i \geq 0}  \in
\otimes E^*, $$ 
where $E$ is the vector space in which $\theta$ takes values.
All the relations that the derivatives $(D^if)(z)$ should satisfy are uniquely determined from $\{ (D^i\gamma)(z)\}$  and so is the image of $\varepsilon$, which we denote by $(\otimes E^*)_{\gamma}$.

With a similar notation for $\mathscr S_W\overline{Q}^{(k)}$, we have an isomorphism
$$\overline{\varepsilon}:[[\overline{Q}, \overline{z}]] \rightarrow (\otimes E^*)_{\overline{\gamma}}.$$
But, by the assumption that $(D^i \gamma)(z) = (\overline{D}^i\overline{\gamma})(\overline{z})$ for $i \geq 0$, we see that
$(\otimes E^*)_{\gamma} = (\otimes E^*)_{\overline{\gamma}}$.

Now define an isomorphism $\Psi$ by the following commutative diagram:
\begin{eqnarray*}
\xymatrix{
[[Q,z]] \ar[r]^{\Psi} \ar[d]^{\varepsilon} & [[\overline{Q}, \overline{z}]] \ar[d]^{\overline{\varepsilon}} \\
(\otimes E^*)_{\gamma} \ar[r]^{id} & (\otimes E^*)_{\overline{\gamma}}.
}
\end{eqnarray*}
Then the ring isomorphism satisfies
$$\Psi^* \overline{\theta}= \theta.$$

Take a coordinate system $x=(x^1, x^2, \dots)$ of $Q$ and let $\overline{x}$ be the corresponding coordinate system of $\overline{Q}$. Then
$$dx = (Dx)\circ \theta \text{ and } d \overline{x} = (\overline{D} \overline{x}) \circ \overline{\theta}$$
and $dx = \Phi^*d \overline{x}$ and $Dx = \Phi^*\overline {D}\overline{x}$. Furthermore,  $Dx$ and $\overline{D}\overline{x}$ are invertible. Hence we conclude $\Psi^* \overline{\theta} = \theta$.
Thus we have a formal isomorphism
$$\Psi:[[\mathscr S_WQ^{(k)}, \theta, z]] \rightarrow [[\mathscr S_W\overline{Q}^{(k)}, \overline{\theta}, \overline{z}]].$$

On the other hand, we have the following categorical isomorphisms:
$$(Q^{(k)}, geom. str.) \simeq (\mathscr S_WQ^{(k)}, geom. str.) \simeq (\mathscr S_W Q^{(k)}, \theta).$$
Therefore, the formal isomorphism $\Psi$ induces a formal isomorphism
$$\Psi: [[Q^{(k)}, geom. str., z^{(k)}]] \rightarrow [[\overline{Q}^{(k)}, geom. str., \overline{z}^{(k)}]].$$
\end{proof}

\begin{corollary}
Let $Q^{(k)}$ and $\overline{Q}^{(k)}$ be proper geometric structures of the same type, and let $\mathscr S_WQ^{(k)}$ and $\mathscr S_W\overline{Q}^{(k)}$ be their $W$-normal complete prolongations. 
If $ \chi(z) = \overline{\chi}(\overline{z})$
for any essential invariant $\chi $ of $ Q^{(k)}$ and the corresponding one $ \overline{\chi}$ of $ \overline {Q}^{(k)}$
for a pair $(z, \overline{z}) \in \mathscr S_WQ^{(k)} \times \mathscr S_W\overline{Q}^{(k)}$,
then there exists a formal isomorphism
$$\Phi^{(k)}:[[Q^{(k)}, z^{(k)}]] \rightarrow [[\overline{Q}^{(k)}, \overline{z}^{(k)}]]$$
where $(z^{(k)}, \overline{z}^{(k)}) \in Q^{(k)} \times \overline{Q}^{(k)}$ is the projection of $(z, \overline{z})$.
\end{corollary}


\subsection{Involutive geometric structures} $\,$

The quasi-involutivity is not enough to solve the equivalence problem for geometric structure of infinite type. It is    the involutive geometric structure that we can solve the equivalence problem perfectly.

\begin{theorem} \label{thm:involutive local equivalence}
Let $Q^{(k)}$ and $\overline{Q}^{(k)}$ be proper geometric structures of type $\frak g[k]$ with structure functions $\gamma^{[k]} $ and $\overline{\gamma}^{[k]} $. Assume that $Q^{(k)}$ and $\overline{Q}^{(k)}$ are involutive and $\gamma^{[k]}  =\overline{\gamma}^{[k]} $.
\begin{enumerate}
\item If, moreover, $\frak g(k)$ is of finite type, then in the $C^{\infty}$-category, there exists a $C^{\infty}$-local isomorphism $(Q^{(k)}, z^{(k)}_Q) \rightarrow (\overline{Q}^{(k)}, z^{(k)}_{\overline{Q}})$ for any $(z^{(k)}_Q, z^{(k)}_{\overline{Q}}) \in Q^{(k)} \times \overline{Q}^{(k)}$.

\item In the analytic category, there exists an analytic local isomorphism $(Q^{(k)},z^{(k)}_Q) \rightarrow (\overline{Q}^{(k)}, z^{(k)}_{\overline{Q}})$ for any $(z^{(k)}_Q, z^{(k)}_{\overline{Q}}) \in Q^{(k)} \times \overline{Q}^{(k)}$.
\end{enumerate}
\end{theorem}

\begin{proof}
The assertion (1) is easily obtained by applying Frobenius Theorem. In fact, it follows immediately from the second fundamental theorem of Lie: Two absolute parallelism $(Q, \theta )$ and $(\overline{Q}, \overline{\theta})$ with constant structure functions $\gamma $ and $\overline{\gamma}$ are locally isomorphic if and only if $\gamma  = \overline{\gamma} $.

The assertion (2) is more involved because we have to deal with general system of partial differential equations.

In the case where the filtration of the base manifold is trivial, it was first proved by Cartan (\cite{C04}, \cite{C05}, \cite{C08}) by using nowadays called Cartan-K\"ahler theorem, and rigorously settled by Singer-Sternberg (\cite{SS65}).

In the case of general filtered manifold of depth $>1$, we have two proofs: One is to reduce the proof to the case of trivial  filtration by geometric consideration. This kind of proof was given in Theorem 3.6.2 of  \cite{M93}. 

 Another proof uses, instead of the classical Cartan-K\"ahler theorem, its nontrivial generalization to a setting of nilpotent analysis (\cite{M90}, Theorem 3.3 of \cite{M02}).
 This proof has the advantage that it makes distinct and clear geometric,
algebraic, and analytic essences which are mixed in the classical Cartan -K\"ahler theorem.
%

However, because of the nature of this theory, we have to impose an additional assumption that the base filtered manifold satisfy the H\"ormander condition,  that is, $\frak g_-$ is generated by $\frak g_{-1}$.

 Now let us give the second proof.






Let $Q$ and $\overline{Q}$ be geometric structures on filtered manifolds $(M,F)$ and $(\overline{M}, \overline{F})$ respectively. Assume that they are of the same type $\frak g = \oplus \frak g_i$ and both involutive at order $k$ having the same structure constants $\gamma^{[k]}=\overline{\gamma}^{[k]}$.

Let $M \times \overline{M} \rightarrow M$ be a fibred manifold, whose local smooth section $\sigma: x \in U \mapsto (x,\varphi(x)) \in U \times \overline{U}$ is identified with a smooth map $U \rightarrow \overline{U}$, where $U$, $\overline{U}$ are open sets of $M$, $\overline{M}$, respectively.

Let $\mathscr J^{(\ell)}(M \times \overline{M})$ denote the set of all weighted $\ell$-jets $j^{\ell}_x\varphi$ of local smooth maps $\varphi: M \rightarrow \overline{M}$ with respect to the filtered manifolds $(M,F)$ and $(\overline{M}, \overline{F})$.
Let $\Gamma^0(M\times \overline{M})$ denote the set of local sections of $M \times \overline{M} \rightarrow M$ consisting of local diffeomorphisms
$$\varphi: U (\subset M) \rightarrow \overline{U}(\subset \overline{M}).$$

Recall that  if  $\varphi \in \Gamma^0(M \times \overline{M})$ preserves the filtrations, that is, $\varphi_*F^p \subset \overline{F}^p$, then it induces the lift
$$\mathscr S^{(\ell)}\varphi: \mathscr S^{(\ell)}(M)|_U \rightarrow \mathscr S^{(\ell)}(\overline{M})|_{\overline{U}}. $$
Consider the equation
$$(\dag) \quad \quad \quad \quad \quad (\mathscr S^{(\ell)}\varphi)(Q^{(\ell)}) \subset \overline{Q}^{(\ell)},$$
where we regard $Q^{(\ell)}$ as a submanifold of $\mathscr S^{(\ell)}(M)$ and $\overline{Q}^{(\ell)}$ as a submanifold of $\mathscr S^{(\ell)}(\overline{M})$ by fixing the complementary subspaces necessary to the embeddings commonly to $Q$ and $\overline{Q}$.

Since the relation $(\dag)$ depends only on $\ell$-jet $j^{(\ell)}\varphi$ of $\varphi$, it defines a submanifold
$$\mathcal R^{(\ell)} \subset \mathscr J^{(\ell)}(M \times \overline{M}),$$ which may be  a system of differential equations for $\varphi$ to define an  isomorphism  $\mathscr S^{(\ell)}\varphi: Q^{(\ell)} \rightarrow \overline{Q}^{(\ell)}$.
If $\mathscr S^{(\ell)}\varphi:(Q^{(\ell)}, z^{\ell}) \rightarrow (\overline{Q}^{(\ell)}, \overline{z}^{\ell})$ is a local or formal isomorphism, then $j^{\ell}_x \varphi $ belongs to $\mathcal R^{(\ell)}_x$, where $x \in M$ is the projection of $z^{\ell}$.

If we denote by $p^{(\ell)}\mathcal R^{(k)}$ ($\ell \geq k$) the prolongation of $\mathcal R^{(k)}$ to $\mathscr J^{(\ell)}(M\times \overline{M})$, we see that $p^{(\ell)}\mathcal R^{(k)} \supset \mathcal R^{(\ell)}$. Moreover, since $Q^{(\ell)}$ and $\overline{Q}^{(\ell)}$ are prolongations of $Q^{(k)}$ and  $\overline{Q}^{(k)}$, respectively, we see that $p^{(\ell)}\mathcal R^{(k)} = \mathcal R^{(\ell)}$ for $\ell \geq k$.

Now take $z \in Q$ and set $z^{\ell}$ and $x$ its projections to $Q^{(\ell)}$ and $M$. Define a map
$$\mathcal R^{(\ell)}_{x} \stackrel{\iota}{\longrightarrow}\overline{Q}^{(\ell)}$$
by $\iota(j^{\ell}_x\varphi) = (\mathscr S^{(\ell)}\varphi)(z^{\ell})$. Then clearly $\iota$ is injective. Furthermore, $\iota$ is surjective for $\ell \geq k $. Indeed, since $Q^{(k)}$ and $\overline{Q}^{(k)}$ are involutive with the same structure constant, for any $\ell \geq k$ and $\overline{z}^{\ell} \in \overline{Q}^{\ell}$, there exist $\overline{z} \in \overline{Q}$ which projects to $\overline{z}^{\ell}$ and  formal isomorphisms $\mathscr S \varphi : (Q, z) \rightarrow (\overline{Q}, \overline{z})$ and $\mathscr S ^{(\ell)}\varphi : (Q^{(\ell)}, z^{\ell}) \rightarrow (\overline{Q}^{(\ell)}, \overline{z}^{\ell})$  (Theorem \ref{thm:Theorem 6.1 formal equivalence}) which make the following diagram commutative:   
\begin{eqnarray*}
\xymatrix{
(Q, z) \ar[r]^{\mathscr S \varphi} \ar[d] & (\overline{Q}, \overline{z}) \ar[d] \\
(Q^{(\ell)}, z^{\ell}) \ar[r]^{\mathscr S^{(\ell)}\varphi} & (\overline{Q}^{(\ell)}, \overline{z}^{\ell})
}
\end{eqnarray*}
with  $(\mathscr S^{(\ell)}\varphi)(z^{\ell}) = \overline{z}^{\ell}$.

We say, according to Malgrange, that a $u^k \in \mathcal R^{(k)}$ is {\it strongly prolongable} if for $\ell \geq k$, $u^{\ell} \in \mathcal R^{(\ell)}$ projects to $u^k$, then there is $u^{\ell+1} \in \mathcal R^{(\ell+1)}$ which projects to $u^{\ell}$.
Since $\iota(\mathcal R^{(\ell)}) = \overline{Q}^{(\ell)}$ and  $\overline{Q}^{(\ell+1)}  \rightarrow \overline{Q}^{(\ell)}$ is surjective for $\ell \geq k$, for our $\mathcal R$, any $u^{\ell} \in \mathcal R^{(\ell)}$ ($\ell \geq k$) is strongly prolongable.

Now by a generalized Cartan-K\"ahler Theorem (\cite{M90}) if the given data are all analytic, then for any $\ell \geq k$ and $u^{\ell} \in \mathcal R^{\ell}_{x}$, there exists a formal solution $\varphi \in \mathcal R_{x}$ which projects to $u^{\ell}$ and satisfies the formal Geverey estimate, that is, $\varphi$ is a formal map $(M,x) \rightarrow (\overline{M}, \overline{x})$ satisfying the following estimate: There exist $C$, $\rho >0$ such that
$$|(X_I \varphi)(0) | \leq C w(I)! \rho^{w(I)} \text{ for all } I$$
where $\{X_1, \dots, X_m\}$ is an admissible local basis of the filtered vector bundle $(TM, F)$ and
$X_I = X_{i_1} \dots X_{i_r}$ for $I=(i_1, \dots, i_r)$, $r \geq 0$, $1 \leq i_1, \dots, i_r \leq n$ and $w(I)$ is the associated weight of $I$.

Then under the assumption that $\frak g_-$ is generated by $\frak g_{-1}$, we conclude that the formal Gevrey map $\varphi:(M,x) \rightarrow (M, \overline{x})$ is analytic by virtue of Theorem 2 in \cite{M95} 
 which may be stated equivalently in the following form:

\begin{theorem} \label{thm:estimate}
Let $\{X_1, \dots, X_r\}$ be independent analytic vector fields defined in a neighborhood of the origin $0$ in $\mathbb R^n$ or $\mathbb C^n$ satisfying the H\"ormander condition, that is,
$$X_1, \dots, X_r, [X_i,X_j], [[X_i,X_j],X_k], \dots$$ generate the tangent space. If $F$ is a formal function at $0$ satisfying:
$$\exists C, \rho >0 \text{ such that } |(X_{i_1} \dots X_{i_{\ell}}F)(0)| \leq C \ell ! \rho^{\ell} \text{ for all } \ell >0, 1 \leq i_1, \dots, i_{\ell} \leq r,$$
then $F$ is analytic.
\end{theorem}

This completes the proof of Theorem \ref{thm:involutive local equivalence}.
\end{proof}

\begin{remark}
The equivalence problem is very delicate in the $C^{\infty}$-category for the geometric structure of infinite type. A general theory for $C^{\infty}$-integrability was initiated by Kumpera-Spencer and Goldschmidt (cf. \cite{KS72}, \cite{G81}).  As shown by J. Conn (\cite{Co77}), Theorem \ref{thm:involutive local equivalence} (2) does not hold in general in the $C^{\infty}$-category for a geometric structure of infinite type.
\end{remark}

\subsection{Analytic geometric structures of finite type} $\,$

We have the following notable theorem which applies to all geometric structures of finite type in the analytic category.

\begin{theorem} \label{thm: analytic and of finite type}
Let $Q^{(k)}$ and $\overline{Q}^{(k)}$ be analytic proper geometric structures of the same type and of finite type. If $(D^i\gamma)(z) = (\overline{D}^i\overline{\gamma})(\overline{z})$ for all $i \geq 0$ for a pair $(z, \overline{z}) \in \mathscr S_WQ^{(k)} \times \mathscr S_W\overline{Q}^{(k)}$, then there exists an  analytic  local isomorphism
$$\varphi^{(k)}:(Q^{(k)}, z^{(k)}) \rightarrow (\overline{Q}^{(k)}, \overline{z}^{(k)}).$$
\end{theorem}

Theorem \ref{thm: analytic and of finite type}  follows from  
Theorem \ref{thm:Theorem 6.1 formal equivalence} and the following fact.

\begin{proposition} \label{prop:weighted estimate}
Let $X_1, \dots, X_n$ be $n$-independent analytic vector fields defined on a neighborhood of $0$ in $\mathbb R^n$ or $\mathbb C^n$. A formal function $f$ at $0$ is convergent in a neighborhood of $0$ if and only if there exists positive constants $C$ and $\rho$ such that
$$ |(X_{i_1} X_{i_2}\dots X_{i_{\ell}} f)(0)| \leq C \ell!\rho^{\ell} \text{ for } \ell=0,1,2, \dots, (i_1, i_2, \dots, i_{\ell}) \in \{1,2,\dots, n\}^{\ell}.$$
\end{proposition}


This proposition is a special case of Theorem \ref{thm:estimate}.
 We, however, give an elementary proof in {\bf Appendix} at the end of this paper.

\begin{proof} [Proof of Theorem \ref{thm: analytic and of finite type}]
We keep the same notation as in the proof of Theorem \ref{thm:Theorem 6.1 formal equivalence}. If $(D^i\gamma)(z) = (\overline{D}^i\overline{\gamma})(\overline{z})$ for $(z, \overline{z}) \in \mathscr S_WQ^{(k)} \times \mathscr S_W\overline{Q}^{(k)}$, then, by the proof of  Theorem \ref{thm:Theorem 6.1 formal equivalence}, there is a ring isomorphism $$\Psi:[[Q, z]] \rightarrow [[\overline{Q}, \overline{z}]]$$ such that $\Psi^*\overline{\theta} = \theta$.   Now we assume that $Q$ and $\overline{Q}$ are finite dimensional analytic manifolds.

Suppose that $\overline f \in [[\overline{Q}, \overline{z}]]$ is convergent at $\overline{z}$. Then by Proposition \ref{prop:weighted estimate}, $\{(D^i \overline{f})(\overline{z})\}$ satisfies the estimate in this proposition. Then so does $\{(D^if)(z) \}$, where $f = \Psi^*\overline{f}$. This implies that $\Psi$ is a local analytic isomorphism
$$\Psi:(\mathscr S_WQ^{(k)}, z) \rightarrow (\mathscr S_W\overline{Q}^{(k)}, \overline{z})$$
satisfying $\Psi^*\overline{\theta} =\theta$, from which immediately follows the theorem.
\end{proof}

\begin{corollary}
Let $Q^{(k)}$ and $\overline{Q}^{(k)}$ be analytic proper geometric structures of the same type and of finite type.
If $ \chi(z) = \overline{\chi}(\overline{z})$
for any essential invariant $\chi $ of $ Q^{(k)}$ and the corresponding one $ \overline{\chi}$ of $ \overline {Q}^{(k)}$
for a pair $(z, \overline{z}) \in \mathscr S_WQ^{(k)} \times \mathscr S_W\overline{Q}^{(k)}$,
 then there exists an analytic local isomorphism
$$\varphi^{(k)}:(Q^{(k)}, z^{(k)}) \rightarrow (\overline{Q}^{(k)}, \overline{z}^{(k)}).$$
\end{corollary}


\section{Cartan connections from the view-point of step prolongation} \label{sect:Cartan connection}

In this section
  we investigate conditions for our $W$-normal complete prolongation $\mathscr S_W Q^{(0)}$ with the canonical Pfaff form $\theta$ to become a Cartan connection of type $G/G^0$.

\subsection{L'espace g\'en\'eralis\'e} $\,$

In 1922 Cartan introduced the notion of espace g\'en\'eralis\'e as a curved version of Klein geometry (\cite{C22}). It has been playing an important role in geometry, especially in the equivalence problem of geometric structures. The following formulation is a standard one.

\begin{definition} 
A {\it Cartan connection of type} $L/L^0$ is a principal $L^0$-bundle $P$ on a manifold $M$ with a $\frak l$-valued 1-form $\theta$ on $P$ satisfying the following properties.
\begin{enumerate}
\item $\theta: T_zP \rightarrow \frak l$ is an isomorphism for all $z  \in P$
\item $R_a^* \theta = \mathrm{Ad} (a)^{-1}\theta$ for $a \in L^0$
\item $\theta(\widetilde A) =A $ for $A \in \frak l^0$.
\end{enumerate}
\end{definition}

In the above definition it is not the Lie group $L$ but the Lie algebra $\frak l$ that actually appear, so that we may refer to it a Cartan connection of type $(\frak l, L^0)$.
%
%
We may slightly weaken the condition on $\frak l$ not assuming that $\frak l$ has a Lie algebra structure.


 \begin{definition}
 Let $L^0$ be a filtered Lie algebra and $E$ an $L^0$-module which contains the Lie algebra $\frak l^0$ of $L^0$ as an $L^0$-adjoint submodule.  We call $(E, L^0)$ a {\it transitive filtered pre-Lie algebra} if the module $E$ is endowed with a filtration $\{F^pE\}_{p \in \mathbb Z}$ such that
\begin{enumerate}
\item $F^pE \supset F^{p+1}E$, $F^0E = \frak l^0$
\item $\rho(F^i\frak l^0) F^pE \subset F^{i+p}E$ for $i \geq 0$ and $p \in \mathbb Z$
\item $\gr E$ is a transitive graded Lie algebra
\end{enumerate}


\end{definition}

\begin{definition} Let $(E,L^0)$ be a transitive filtered pre-Lie algebra.
A {\it pre-Cartan connection of type} $(E, L^0)$ is a principal $L^0$-bundle $P$ on a manifold $M$ with a $E$-valued 1-form $\theta$ on $P$ satisfying the following properties.
\begin{enumerate}
\item $\theta: T_zP \rightarrow E$ is an isomorphism for all $z  \in P$
\item $R_a^* \theta = \rho(a)^{-1}\theta$ for $a \in L^0$
\item $\theta(\widetilde A) =A $ for $A \in \frak l^0$.
\end{enumerate}

 \end{definition}

Note that if there is a pre-Cartan connection $\overset{\circ}{P}$ of type $(E, L^0)$ which has a constant structure function $\overset{\circ}{\gamma}$, then $\overset{\circ}{\gamma}$ defines a Lie algebra structure on $E$ extending that of $\frak l^0$, so that any pre-Cartan connection $P$ of type $(E, L^0)$ is a Cartan connection.





\subsection{Principal geometric structures and principal prolongations} \label{sect:principal geometric structure} $\,$


In this subsection we define a principal geometric structure and its universal frame bundle, and introduce the category of the principal geometric structures
which is a little more restrictive than that of geometric structures introduced in Section \ref{sect:geometric structure of order k}, and which is well adapted to study the Cartan connections.

 To do this we have only to replace ``step wise principal fiber bundle" in Definition \ref{def:universal frame bundle}
  by ``principal fiber bundle".

\begin{definition} \label{def: universal principal frame bundles}
    A {\it principal geometric structure  $\breve{Q}^{(k)}$ of order $k$}
    and its {\it universal principal frame bundle $\mathscr R^{(k+1)}\breve{Q}^{(k)}$ of order $k+1$} are defined inductively for $k \geq -1$ by the following properties.

\begin{enumerate}

\item
 A principal geometric structure $\breve{Q}^{(k)}$
 of order $k \geq -1$   (of type $(\frak g_-, \breve{G}^{(k)})$) on a filtered manifold $(M,F)$ (of type $ \frak g_-$)
  is a principal fiber bundle over $M$, of which the structure group $\breve{G}^{(k)}$
  is endowed with a filtration 
  $$ 1 =F^{k+1}\breve{G}^{(k)}\subset F^k\breve{G}^{(k)} \subset\cdots\subset F^0\breve{G}^{(k)}=\breve{G}^{(k)}$$

 consisting of closed normal subgroups  satisfying: 

\begin{enumerate}
\item
A principal geometric structure $\breve{Q}^{(-1)}$
 of order $-1$
 on $(M,F)$
 is the filtered manifold $(M,F)$ itself.
\item
If $k \geq 0$,  $\breve{Q}^{(i-1)}:=\breve{Q}^{(k)}/F^{i+1}$ is a principal geometric structure of order $i-1$ (of type $(\frak g_-, \breve{G}^{(i-1)}:=\breve{G}^{(k)}/F^{i+1})$) on $(M,F)$ for $0 \leq i \leq k$.
\item
If $k \geq 0$, $\breve{Q}^{(i)} \rightarrow M $ is a principal subbundle of the universal principal frame bundle
 $\mathscr R^{(i)}\breve{Q}^{(i-1)}\rightarrow M $ of $ \breve{Q}^{(i-1)}$ of order $i$  for $0\leq i \leq k$.

\end{enumerate}


 \item
To every principal geometric structure $\breve{Q}^{(k)}\stackrel{\breve{G}^{(k)}}
             {\rightarrow }
             M$ of order $ k \geq 0 $ there is associated a principal fiber bundle
$\widehat{\mathscr R}^{(k+1)}\breve{Q}^{(k)} \stackrel{\widehat G^{(k+1)} }{\rightarrow} M$ which is called the universal linear frame bundle of $\breve{Q}^{(k)}$.
Its structure group $\widehat {\breve G}^{(k+1)} $ is endowed with a filtration $\{F^{\ell} \}$.
The quotient
\begin{eqnarray*}
{\mathscr R}^{(k+1)}\breve{Q}^{(k)}&:  =&\widehat{\mathscr R}^{\,(k+1)}\breve{Q}^{(k)} /F^{k+2}
\end{eqnarray*}
  is a principal fiber bundle on $M$  with structure group
\begin{eqnarray*}
\breve G^{(k+1)} &: =&\widehat {\breve G}^{(k+1)}/F^{k+2},
\end{eqnarray*}
 which is called the universal frame bundle of $\breve  Q^{(k)}$ of order $ k+1 $.


\item
The universal linear frame bundle $\widehat{\mathscr R}^{(k+1)}\breve{Q}^{(k)}$   for a principal geometric structure $\breve{Q}^{(k)}$ of order $k\geq -1$ on $M$  is
  defined as follows.

  For (principal) geometric structure of order $-1$ $\breve Q^{(-1)} ={\bf Q}^{(-1)}$, we set
  $$\widehat{\mathscr R}^{(0)} \breve Q^{(-1)} =\widehat{\mathscr S}^{(0)}  {\bf Q}^{(-1)} $$

  For $k\geq 0$ it is defined as follows:

\begin{enumerate}

\item
The bundle $\widehat{\mathscr R}^{(k+1)}\breve{Q}^{(k)}$.
  We set
$$\breve E^{(-1)}:= \frak g_- \text{ and } \breve{E}^{(k)}:= \frak g_- \oplus \breve{\frak g}^{(k)}\text{ if } k \geq 0,$$
where $\breve{\frak g}^{(k)}$ is the Lie algebra of $\breve{G}^{(k)}$.
We regard $\breve E^{(k)}$ as a filtered vector space.
We note also that the tangent space $T_{z^k}\breve Q^{(k)}$ has a canonical filtration
$\{F^{\ell}T_{z^k}\breve{Q}^{(k)}\}_{\ell \in \mathbb Z}$.

The fiber $\widehat{\mathscr R}^{(k+1)}_{z^k}$ over $z^k$ for $z^k \in \breve Q^{(k)}$,  is the set of all filtration preserving isomorphisms
  $$\zeta^{k+1}:\breve E ^{(k)} \rightarrow T_{z^k} \breve{Q}^{(k)}$$
  which satisfy the following conditions:


  $\zeta^{k+1}(A) = \widetilde{A}_{z^k} \text{ for } A \in \breve{\frak g}^{(k)}  $ and
   $[\zeta^{k}]=z^{k}$,
  where $\zeta^k$ is the truncation of $\zeta^{k+1}$, defined by the following commutative diagram
  \begin{eqnarray*}
  \xymatrix{
  \breve E^{(k)} \ar[d]^{\zeta^{k+1}} \ar[r] & \breve E^{(k-1)} \ar[d]^{\zeta^k}   \\
   T_{z^k}\breve{Q}^{(k)}   \ar[r] & T_{z^{k-1}}\breve Q^{(k-1)}  ,
  }
  \end{eqnarray*}
and $z^{k-1}$ is the image of $z^k$ by the projection map $\breve Q^{(k)} \rightarrow \breve Q^{(k-1)}$.

\item
 The structure group $\widehat{\breve G}^{(k+1)} $  of $\widehat{\mathscr R}^{(k+1)}\breve{Q}^{(k)}$
     consists of all filtration preserving linear isomorphisms $\alpha^{k+1}$ which make the following diagram commutative:
      \begin{eqnarray*}
 \xymatrix{
  \breve E^{(k)} \ar[d]^{\alpha^{k+1}} \ar[r] & \breve E^{(k-1)} \ar[d]^{\alpha^k}   \\
  \breve E^{(k)}   \ar[r] & \breve E^{(k-1)}
 }
 \end{eqnarray*}
with  $[\alpha^k] =a^k \in \breve{G}^{(k)}$, and which satisfy
$$\alpha^{k+1}(A) = Ad(a^k)A \qquad \text{ for } A \in \breve{\frak g}^{(k)}.$$

\item The action of $\widehat{\breve G}^{ k+1  } $ on $\widehat{\mathscr R}^{(k+1)}(\breve{Q}^{(k)})$.
    For $\alpha^{k+1} \in \widehat{\breve G}^{ (k+1)  } $ and $\zeta^{k+1} \in \widehat{\mathscr R}^{(k+1)}(\breve Q^{(k)})$ we define the right action $\zeta^{k+1}\cdot \alpha^{k+1}$ by the following commutative diagram:
\begin{eqnarray*}
 \xymatrix{
 \breve E^{(k)} \ar[d]^{\alpha^{k+1}}\ar[r]^{\zeta^{k+1}\cdot \alpha^{k+1}} & T_{z^ka^k }\breve Q^{(k)} \\
 \breve E^{(k)} \ar[r]^{\zeta^{k+1}}& T_{z^{k}}\breve Q^{(k)} \ar[u]_{ (R_{a^k})_* }
 }
 \end{eqnarray*}
where $z^{k} =[\zeta^{k}]$ is the projection  of $\zeta^{k+1}$   to $\breve  Q^{(k)}$ and $a^k=[\alpha^k]$ similarly.  


\end{enumerate}

\end{enumerate}
\end{definition}


\noindent {\bf Notations.}
In   Definition \ref{def: universal principal frame bundles}, the filtered vector space $\breve{E}^{(k)}$ and the filtered group $\breve G^{(k+1)}$ depend on the structure group $\breve {G}^{(k)}$ of $\breve Q^{(k)}$. When we need to emphasize it, we write  $\breve{E}^{(k)}(\breve {\frak g}^{(k)})$ and  $\breve G^{(k+1)}(\breve{G}^{(k)})$ for $\breve{E}^{(k)}$ and  $\breve G^{(k+1)}$. \\



Let $\breve{Q}^{(k)} \rightarrow M$ be a principal geometric structure of order $k$.
By setting for $\ell >k$
$$\mathscr R^{(\ell)}(\breve Q^{(k)}) = \mathscr R^{(\ell)}(\mathscr R^{(\ell-1)}\breve Q^{(k)}), \qquad \breve {G}^{(\ell)}(\breve G^{(k)})=\breve G^{(\ell)}(\breve G^{(\ell-1)}(\breve G^{(k)}))$$
and by passing to the limit
$$\mathscr R \breve{Q}^{(k)} = \lim_{\leftarrow \ell} \mathscr R^{(\ell)}\breve Q^{(\ell)}, \qquad \breve G(\breve G^{(k)}) =\lim_{\leftarrow \ell} \breve{G}^{(\ell)}(\breve G^{(k)}),$$
we obtain a principal fiber bundle $\mathscr R \breve Q^{(k)} \stackrel{\breve G(\breve G^{(k)})}{\longrightarrow} (M,F)$ called the {\it complete universal principal frame bundle} of $\breve Q^{(k)}$, which is endowed with an absolute parallelism defined by the  canonical form $\breve{\theta}$.

The canonical form $\breve{\theta}$ is an $\breve E  (\breve {\frak g}^{(k)})$($=\frak g_- \oplus \breve {\frak g} (\breve {\frak g}^{(k)})$)-valued 1-form on $\mathscr R \breve Q^{(k)}$. Write $ \breve E(\breve{\frak g}^{(k)})$ simply as $\breve E$. We then have the structure function
$$\breve {\gamma}:\mathscr R \breve{Q}^{(k)} \rightarrow \Hom(\wedge ^2 \breve E, \breve E)$$ defined by $d \breve{\theta} + \frac{1}{2} \breve{\gamma}(\breve{\theta}, \breve{\theta})=0$.

The canonical form $\breve{\theta}$ and the structure function $\breve{\gamma}$ of $\mathscr R\breve Q^{(k)}$ enjoy the same property as those of $\mathscr S Q^{(k)}$. More rigidly, we have
$$R_a^* \breve {\theta} = \rho(a)^{-1} \breve{\theta} \qquad \text{ for } a \in \breve{G} (\breve G^{(k)})$$
and
$$\breve {\gamma}(A,X) = \rho(A)X \qquad \text{ for } A \in \breve {\frak g}  (\breve{\frak g}^{(k)}) \text{ and } X \in \breve E (\breve {\frak g}^{(k)})$$
where $\rho$ denotes the action of $\breve G (\breve G^{(k)})$ or $\breve{\mathfrak g} (\breve{\frak g}^{(k)})$ on $\breve{E} (\breve{\frak g}^{(k)})$.
We have thus seen the similarity and the difference between $\{Q^{(k)}, \mathscr SQ^{(k)}\}$ and $\{\breve Q^{(k)}, \mathscr R \breve Q^{(k)}\}$.



\begin{remark}

The notions of principal geometric structure and universal principal frame bundle were introduced  in \cite{M83} and \cite{M93}
 under different terminologies as non-commutative frame bundle, Cartan bundle, or tower,  and there played a fundamental role
 to study the equivalence problem of geometric structures and, in particular, to construct Cartan connections.

It is in order to study the step prolongation that we have introduced in this paper in Section \ref{sect:geometric structure of order k} the larger category of geometric structures by extending the category of principal geometric structures.

This distinction and similarity of two categories will then make clear the relation between step prolongation and Cartan connection
as shown in the next sub section.

%
\end{remark}

Parallel to Definition \ref{def:normal geometric structure}, we give the following:

\begin{definition}
A principal geometric structure $\breve{Q}^{(k)} \stackrel{\breve{G}^{(k)}}{\longrightarrow}  M $ of type $(\frak g_-,  \breve{G}^{(k)})$ is called {\it proper} if $\frak g_- \oplus \gr \breve{\frak g}^{(k)}$ has a compatible structure of a transitive truncated graded Lie algebra, where $\gr \breve{\frak g}^{(k)}$ is the graded Lie algebra  associated with the Lie algebra $\breve{\frak g}^{(k)}$ of $\breve G^{(k)}$.
\end{definition}

\subsection{The structure function $\tau$ and Cartan connections} $\,$


In this subsection we first show that a principal geometric structure can be viewed as a geometric structure whose structure function $\tau$ is constant.

To see this more precisely, let $\breve{Q}^{(k)} \stackrel{\breve{G}^{(k)}}{\longrightarrow}M$ be a principal geometric structure. For $0 \leq i \leq k$, let $\breve{\frak g}^{(i)}$ be the Lie algebra of $\breve{G}^{(i)}=\breve{G}^{(k)}/F^{i+1}$ and set $\frak g^{(i)}$($=\oplus_{a=0}^i \frak g_a$)$=\gr \breve{\frak g}^{(i)}$, $\breve{E}^{(i)} = \frak g_- \oplus \breve{\frak g}^{(i)}$ and $E^{(i)}=\frak g_- \oplus \frak g^{(i)}$.

Choose complementary subspace $\{W^1_i\}_{1 \leq i \leq k}$ such that
$$\Hom(\frak g_-, E^{(i-1)})_i = \frak g_i \oplus W^1_i.$$
Then the choice of $\{W^1_i\}$ determines in a natural manner identifications
$$
\phi_W^{(i)}: \frak g ^{(i)} \stackrel{\cong}{\longrightarrow} \breve{\frak g}^{(i)} \text{ and hence } \phi_W^{(i)} :E^{(i)} \stackrel{\cong}{\longrightarrow} \breve{E}^{(i)}$$
and further constants $\overset{\circ}{\tau}_{[i+1]} \in \Hom(\frak g^{(i)}, W^1_i)$.



Indeed, the construction can be done by induction as follows.
Suppose that we have the isomorphism $\phi_W^{(i-1)}: \frak g^{(i-1)} \rightarrow \breve{\frak g}^{(i-1)}$ and consider the following diagram:
\begin{eqnarray*}
\xymatrix{
& 0 \ar[d] & 0 \ar[d] & & \\
0 \ar[r] &\quad \frak g_i\quad  \ar[r] \ar[d] & \quad \breve{\frak g}^{(i)} \quad \ar[r] \ar[d]^{\iota_{\mathscr R}} &\quad  \breve{\frak g}^{(i-1)}\quad     \ar@{=}[d] \ar[r] & 0  \\
0 \ar[r]&\quad  \Hom(\frak g_-, \breve E^{(i-1)})_i \quad \ar[d]^{\pi_W} \ar[r] &\quad \breve{\frak g}_{\mathscr R}^{(i)}\quad     \ar[r] &\quad \breve{\frak g}^{(i-1)} \quad \ar[r] & 0 \\
& W_i^1 \ar[d] &&& \\
&0,&&&
}
\end{eqnarray*}
where the first column is induced from the isomorphism $\phi^{(i-1)}_W: E ^{(i-1)} \rightarrow  \breve{E }^{(i-1)}$ and the split exact sequence:
$$0 \rightarrow \frak g_i \rightarrow \Hom(\frak g_-, E^{(i-1)})_i \rightarrow W_i^1\rightarrow 0 $$
and the second column is the inclusion map $\iota_{\mathscr R}:\breve{\frak g}^{(i)} \hookrightarrow \breve{\frak g}^{(i)}_{\mathscr R}:=\breve{\frak g}^{(i)}(\breve{\frak g}^{(i-1)})$.

Recalling the definition of the principal prolongation $\breve{\frak g}_{\mathscr R}^{(i)}$ of $\breve{\frak g}^{(i-1)}$ and taking into account the isomorphism $\phi_W^{(i-1)}:E ^{(i-1)} \rightarrow  \breve{E }^{(i-1)}$, we see that the second row has a natural splitting (as vector spaces).
Since the first column and the second row split so does the first row, which then yields an isomorphism $\phi^{(i)}_W:\frak g ^{(i)} \rightarrow  \breve{\frak g} ^{(i)}$.

A diagram chasing gives a well-defined map
$$w_{[i+1]}:\breve{\frak g}^{(i-1)} \rightarrow W^1_i$$
such that
$$A + w_{[i+1]}(A) \in \breve{\frak g}^{(i)} \qquad \text{ for }  A \in \breve{\frak g}^{(i-1)}$$
according to the identifications:
$$\breve{\frak g}^{(i)} \subset \breve{\frak g}_{\mathscr R}^{(i)} = \breve{\frak g}^{(i-1)} \oplus W_i^1 \oplus \frak g_i.$$
We then define $\overset{\circ}{\tau}_{[i+1]}:\frak g^{(i)} \rightarrow \Hom(\frak g_-, E^{(i-1)})_i$
by
$$\overset{\circ}{\tau}_{[i+1]}|_{\frak g^{(i-1)}} = w_{[i+1]}\circ \phi^{(i-1)}_W, \qquad \overset{\circ}{\tau}_{[i+1]}|_{\frak g_i} = id_{\frak g_i}.$$

The notation being as above, we have

  \begin{proposition} \label{prop:W-normal and principal prolongation}
  Let $\breve{Q}^{(k)} \stackrel{\breve{G}^{(k)}}{\longrightarrow} M$ be a  principal geometric structure of type $(\frak g_-, \breve{G}^{(k)})$. Then a choice of  $\{W^1_i\}_{1 \leq i \leq k}$  determines a natural identification of $\breve {Q}^{(k)}$ with a geometric structure  $Q^{(k)}$ with the structure function $\tau^{[k]} =  \overset{\circ}{\tau} ^{[k]}$ as follows:
    There exists a  geometric structure
  $$Q^{(i)} \stackrel{G_i}{\longrightarrow} Q^{(i-1)} \longrightarrow \dots \longrightarrow Q^{(0)} \rightarrow M$$
  of type $(\frak g_-, G_0, \dots, G_i)$
  and bundle isomorphisms $\iota_W^{(i)}:\breve{Q}^{(i)} \rightarrow Q^{(i)}$ and $\mathscr R \iota ^{(i)}_W:\mathscr R^{(i+1)}\breve{Q}^{(i)} \rightarrow \mathscr S_{W^1}^{(i+1)}Q^{(i)}$   for $0 \leq i \leq k  $ such that
  \begin{enumerate}
  \item the structure function $\tau_{[i+1]}$ of $\mathscr S_{W^1}^{(i+1)}Q^{(i)}$ equals to $\overset{\circ} {\tau}_{[i+1]} $,
      and
  \item the following diagram is commutative
  \begin{eqnarray*}
   \xymatrix{
   \breve{Q}^{(i+1)} \ar[r]^{\iota_W^{(i+1)}} \ar@{^{(}->}[dr] & Q^{(i+1)}\ar@{^{(}->}[dr] & & \\
   & \mathscr R^{(i+1)}\breve{Q}^{(i)} \ar[r]^{\mathscr R \iota^{(i)}_W} \ar[d] & \mathscr S_{W^1}^{(i+1)} Q^{(i)} \ar@{^{(}->}[r]\ar[d] & \mathscr S^{(i+1)}Q^{(i)} \\
   &  \breve{Q}^{(i)} \ar[r]^{\iota_W^{(i)}} & Q^{(i)}. &
   }
  \end{eqnarray*}
  Here, we ignore the first row if $i=k$.

  \end{enumerate}

  \end{proposition}

  \begin{proof}
 The isomorphism $\phi_W^{(i)}: E^{(i)} \rightarrow \breve{E}^{(i)}$ induces a map $\mathscr R \iota^{(i)}_W: \mathscr R^{(i+1)}\breve{Q}^{(i)} \rightarrow \mathscr S ^{(i+1)}Q^{(i)}$ as follows. Let $\breve{z}^{i+1} =[\breve{\zeta}^{i+1}] \in \mathscr R^{(i+1)}\breve Q^{(i)}$. Define $z^{i+1} =[\zeta^{i+1}] \in \mathscr S^{(i+1)}Q^{(i)}$ by
  \begin{eqnarray*}
  \xymatrix{
  \breve{E}^{(i )} \ar[r]^{\breve{\zeta}^{i+1}} & T_{\bullet} \breve{Q}^{(i)} \ar[d]^{{\iota_W^{(i)}}_*} \\
  E^{(i )} \ar[u]^{\phi_W^{(i)}} \ar[r]^{\zeta^{i+1}} & T_{\bullet} Q^{(i)},
  }
  \end{eqnarray*}
  and  take the reduction $\mathscr S^{(i+1)}_{W^1} Q^{(i)}$ of $\mathscr S^{(i+1)}Q^{(i)}$ as the inverse image of $\overset{\circ}{\tau}_{[i+1]} $. Then $\mathscr S_{W^1}^{(i+1)}Q^{(i)}$ coincides exactly with the image of $\mathscr R^{(i+1)}\breve{Q}^{(i)}$.
%
 % ...
  \end{proof}

\begin{theorem} \label{thm: constant tau implies pre Cartan}
Let $Q$ be a proper geometric structure of type $(\frak g_-, G_0, \dots, G_k, \dots)$. If the structure function $\tau$ of $Q$ is constant and the structure groups $G_i$ are connected for $i\geq 0$, then $Q$ is a principal proper geometric structure and $(Q, \theta)$ is a pre-Cartan connection.
\end{theorem}

\begin{proof}
Let $Q$ be a proper geometric structure of type $(\frak g_-, G_0, \dots, G_k, \dots)$ with constant structure function $\tau$. Let $\frak g= \oplus_p \frak g_p$ be the transitive graded Lie algebra associated to $Q$, which we also denote by $E$. We also write $\frak g^0=\oplus_{p \geq 0} \frak g_p$, the graded subalgebra of $\frak g$.

From the fundamental identity it follows that the structure function $\sigma$ is also constant (Corollary \ref{cor:simple observation}). The Bianchi identity implies that the structure function $\sigma \in \Hom(\wedge^2 \frak g^0, \frak g^0 )$ satisfies the Jacobi identity. Hence $(\frak g^0, \sigma)$ defines a filtered Lie algebra $(\breve{\frak g}, \{F^p \breve{\frak g}\})$ such that $[F^p\breve{\frak g}, F^q \breve{\frak g}] \subset F^{p+q} \breve{\frak g}$ and $\gr \breve{\frak g}^0$ is isomorphic to $\frak g^0$. Again by the Bianchi identity, we see that the structure function $\tau + \sigma \in \Hom(\frak g^0 \otimes E, E)$ determines  an action of $\breve{\frak g}^0$ on $E$, which acts on $\frak g^0 \subset E$ as the adjoint action, so we also write $E=\breve E = \frak g_- \oplus \breve{\frak g}^0$. Then we have $F^p\frak g^0 \cdot F^p \breve{E}^q \subset F^{p+1} \breve{E}$, and thus we have an injective Lie algebra homomorphism
$$\iota: \breve{\frak g}^{(k)} \rightarrow F^0\frak{gl}(\breve{E}^{(k-1)})/F^{k+1}$$
where $\breve{\frak g}^{(k)} = \breve{\frak g}/F^{k+1}$ and $\breve{E}^{(k-1)}=\breve{E}/F^k$. We then identify $\breve{\frak g}^{(k)}$ with its image $\iota(\breve{\frak g}^{(k)})$.

Now we take the connected Lie subgroup $\breve{G}^{(k)}$ of $F^0GL(\breve{E}^{(k-1)})/F^{k+1}$ with Lie algebra $\breve{\frak g}^{(k)}$. Let $\breve{G} =\lim_{\leftarrow k} G^{(k)}$, which then satisfies: $F^kG/F^{k+1} \cong G_k$.
Then via the canonical form $\theta$ we can easily see that $\breve{Q}^{(k)} = Q^{(k)} = Q/F^{k+1}$ admits $\breve{G}^{(k)}$-action, which together with the canonical Pfaff class $[\theta^{(k-1)}]$ makes $\breve{Q}^{(k)}$ a principal proper geometric structure of order $k$. The rest of assertions of the theorem then follows immediately.
\end{proof}

\begin{corollary} An involutive    proper  geometric structure $Q^{(k)}$   of order $k$ is a principal geometric structure, provided that the structure group $G_i$ of $Q^{(i)} \rightarrow Q^{(i-1)}$ is connected for $0 \leq i \leq k$.
\end{corollary}

\begin{proof}
The step prolongation $\mathscr S_WQ^{(k)}$ has constant structure function. Hence $\tau$ is constant.
\end{proof}

\begin{remark}
The definition of involutive geometric structure given in  \cite{M93}  and in the present paper thus infinitesimally coincide. In \cite{M93} the geometric structures are assumed a priori to be principal.
\end{remark}


Now let us consider the case where the structure function $\tau$ is flat and demonstrate a slightly stronger statement than that of Theorem  \ref{thm: constant tau implies pre Cartan}.

Let $\frak g_-=\oplus_{p<0}\frak g_p$ be a graded Lie algebra. Let $G_0$ be a connected subgroup of the connected component $\Aut(\frak g_-)^0$  of the  automorphism group of the graded Lie algebra and $\frak g_0$ be its Lie algebra. Let $\frak g= \oplus_{p \in \mathbb Z} \frak g_p$ be the prolongation of $\frak g_- \oplus \frak g_0$. Let $G^{(k)}, G_k$ be the Lie groups with Lie algebra $\frak g^{(k)}=\oplus _{i=0}^k\frak g_i$ and $\frak g_k$, respectively, constructed as follows.
Write $G^{(\infty)}$ also as $G^0$.


We construct $G^{(k)}$ inductively. Suppose that we have constructed $G^{(k-1)}$ in such a way that
$$G^{(k-1)} =G_0 \times N(\frak g_1 \oplus \dots \oplus \frak g_{k-1}) \subset F^0GL(E^{(k-2)})/F^k.$$
Here, $G_0$ is regarded as a closed subgroup of $F^0GL(E^{(k-2)})/F^k $ and $N(\frak g_1 \oplus \dots \oplus \frak g_{k-1})$ is a closed normal connected subgroup of  $F^0GL(E^{(k-2)})/F^k $  with Lie algebra $\frak g_1 \oplus \dots \oplus \frak g_{k-1}$, defined by
$$N(\frak g_1 \oplus \dots \oplus \frak g_{k-1}):=\mathrm{Exp}(\frak g_1 \oplus \dots \oplus \frak g_{k-1})/F^k$$
where $\mathrm{Exp}:\frak{gl}(E^{(k-2)}) \rightarrow GL(E^{(k-2)}) $ is the exponential map.

Let $G^{(k)}_{\mathscr R}=G^{(k)}(G^{(k-1)})$ be the principal prolongation of $G^{(k -1)}$. Then $G_0$ is regarded as a closed subgroup of $G^{(k)}_{\mathscr R}$ and
$$G^{(k)}_{\mathscr R} = G_0 \times N(\frak g_1 \oplus \dots \oplus \frak g_{k-1} \oplus \frak q_{k})$$
where $\frak q_{k}$ is $\Hom(\frak g_-, E^{(k-1)})_k$ and $N(\frak g_1 \oplus \dots \oplus \frak g_{k-1} \oplus  \frak q_k)$, define by
$$N(\frak g_1 \oplus \dots\oplus \frak g_{k-1}\oplus \frak q_k):= \mathrm{Exp}(\frak g_1 \oplus \dots\oplus \frak g_{k-1}\oplus \frak q_k)/F^{k+1},$$
 is a closed normal connected subgroup with Lie algebra $\frak g_1 \oplus \dots \oplus \frak g_{k-1}\oplus \frak q_k$. Here, $\mathrm{Exp}: \frak {gl}(E^{(k-1)}) \rightarrow GL(E^{(k-1)})$ is the exponential map.
Since $\frak g_1 \oplus \dots\oplus \frak g_{k-1}\oplus \frak q_k$ is contained in the lower triangular matrices, the exponential map
 $\mathrm{Exp}: \frak {gl}(E^{(k-1)}) \rightarrow GL(E^{(k-1)})$
maps $\frak g_1 \oplus \dots\oplus \frak g_{k-1}\oplus \frak q_k$ onto a closed subgroup (homeomorphic to an Euclidean space), and the quotient map
$$\mathrm{Exp}(\frak g_1 \oplus \dots\oplus \frak g_{k-1}\oplus \frak q_k) \rightarrow \mathrm{Exp}(\frak g_1 \oplus \dots\oplus \frak g_{k-1}\oplus \frak q_k)/F^{k+1}$$
is bijective.

For the subalgebra $\frak g_1 \oplus \cdots \oplus \frak g_{k-1} \oplus \frak g_k \subset \frak g_1 \oplus \cdots \oplus \frak g_{k-1} \oplus \frak q_k$, the image   $N(\frak g_1 \oplus \dots \oplus \frak g_{k-1} \oplus \frak g_k)$  of the map
$$\frak g_1 \oplus \dots \oplus \frak g_{k-1} \oplus \frak g_k \rightarrow \mathrm{Exp}(\frak g_1 \oplus \dots \oplus \frak g_{k-1} \oplus \frak g_k) \rightarrow \mathrm{Exp}(\frak g_1 \oplus \dots \oplus \frak g_{k-1} \oplus \frak g_k)/F^{k+1}$$
is a closed normal subgroup of $G_{\mathscr R}^{(k)}$.
Set $G^{(k)}=G_0 \times N(\frak g_1 \oplus \dots \oplus \frak g_{k-1} \oplus \frak g_k)$. Then there is a filtration $\{F^iG^{(k)}\}$ on $G^{(k)}$ so that  $G^{(k)}/F^{k+1}=G^{(k-1)}$, and there is an embedding $\frak g^{(k)} \rightarrow \frak g^{(k)}(\frak g^{(k-1)})$, where $\frak g^{(k-1)}$ is the Lie algebra of $G^{(k-1)}$. \\

Fix subspaces $\{W^1_i\}_{i \geq 0}$ and $\{W^2_{ i+1}\}_{i \geq 0}$ such that
\begin{eqnarray*}
 \Hom (\frak g_-,  \frak g )_{i} &= &W_{i}^1 \oplus \partial  \frak g_{i}\\
 \Hom(\wedge^2\frak g_-,   \frak g )_{i+1} &= &W_{i+1}^2 \oplus \partial \Hom(\frak g_-,  \frak g )_{i+1} .
 \end{eqnarray*}

\begin{theorem} \label{thm:vanishing of tau implies} Let $Q^{(0)} \stackrel{G_0}{\longrightarrow} (M,F) $ be a proper geometric structure  of order $0$ with $G_0$ being connected. Let $Q^{(k)}$ be the $W$-normal step prolongation  of $Q^{(0)}$  of order $k$ and let $Q$ be the $W$-normal complete step  prolongation of $Q^{(0)}$.
If $\tau^{[k+1]}$ is flat on $Q^{(k+1)}$, then $Q^{(k)} \rightarrow M$ is a proper principal geometric structure of order $k$.

 If $\tau$ is flat,   i.e., $ \tau_{[i]} $ is flat for all $i$,
 then $Q \rightarrow M$ is a Cartan connection of type $G/G^0$.

\end{theorem}

\begin{proof}
Let $Q^{(0)} \stackrel{G_0}{\longrightarrow}(M,F)$ be a geometric structure of order 0 of type $(\frak g_-, G_0)$ with $G_0$ connected and let $Q^{(k)}$ be the $W$-normal step prolongation  of $Q^{(0)}$  of order $k$ for $k \geq 1$.
We will prove by induction on $\ell$ that the following holds: \\

(*) \quad If $\tau^{[\ell+1]}$ is flat on $Q^{(\ell+1)}$, then $Q^{(\ell)} \rightarrow M$ is 
a principal subbundle of $\mathscr R^{(\ell)}Q^{(\ell-1)}$
with structure group $G^{(\ell)}$. \\

For the case when $\ell=0$, there is noting to prove: $\tau^{[1]}$ is flat and $Q^{(0)}\rightarrow M$ is a principal subbundle of $\mathscr R^{(0)}(M)=\mathscr S^{(0)}(M)$.
Furthermore, all three bundles $\mathscr S_{W^1}^{(1)}Q^{(0)}$,  $\mathscr S^{(1)}Q^{(0)}$, $\mathscr R^{(1)}Q^{(0)}$ are the same.

Assuming that  the statement (*) holds for $\ell <k$, we will prove that the statement $(*)$ holds for $\ell=k$. If $\tau^{[k+1]}$ is flat, then $\tau^{[k]}$ is flat. By the inductive assumption, $Q^{(k-1)} \rightarrow M$ is a proper  principal geometric structure of order $k-1$ with structure group $G^{(k-1)}$. The isomorphism from the Lie algebra $\frak g^{(k-1)}$ with the graded Lie algebra $\oplus_{i=0}^{k-1}\frak g_i$ induces an embedding of $\mathscr R^{(k)}Q^{(k-1)}$ into $\mathscr S^{(k)}Q^{(k-1)}$ whose image is $\mathscr S_{W^1}^{(k)}Q^{(k-1)}$ as in the proof of Proposition \ref{prop:W-normal and principal prolongation}.  We identify $\mathscr S_{W^1}^{(k)}Q^{(k-1)}$ with $\mathscr R^{(k)}Q^{(k-1)}$
which is a principal bundle over $M$ with structure group $G^{(k)}(G^{(k-1)})$.

Note that the group $G^{(k)}(G^{(k-1)})$ contains $G^{(k)}$ as a subgroup. Therefore, $G^{(k)}$ acts on $\mathscr R^{(k)}Q^{(k-1)}$, and the structure function $\kappa^{[k]}$ on $\mathscr R^{(k)}Q^{(k-1)}$ satisfies
\begin{eqnarray*}
R_a^*\kappa^{[k]} = \rho(a)^{-1}\kappa^{[k]} \quad \text{ for } a \in G^{(k)}.
\end{eqnarray*}
Thus
\begin{eqnarray} \label{eq:kappa k on principal bundle}
\kappa^{[k]}(z^{(k)}a) = \rho(a)^{-1}\kappa^{[k]}(z^{(k)})
\end{eqnarray} for any $z^{(k)} \in \mathscr R^{(k)}Q^{(k-1)}$ and  $a \in G^{(k)}$.

Now $Q^{(k)}=\mathscr S_W^{(k)}Q^{(k-1)} \rightarrow Q^{(k-1)}$ is a subbundle of $\mathscr S_{W^1}^{(k)}Q^{(k-1)} \rightarrow Q^{(k-1)}$ defined by
  \begin{eqnarray*}
Q^{(k)}= \left\{z \in \mathscr S_{W^1} ^{(k )}Q^{(k-1)}: \kappa_{[k ]}(z) \in W_{k }^2  \right\}.
\end{eqnarray*}
We will show that the action of $G^{(k)}$ on $\mathscr R^{(k)}Q^{(k-1)}=\mathscr S^{(k)}_{W^1}Q^{(k-1)}$ leaves invariant $Q^{(k)}$. \\

Let $Q$ be the $W$-normal complete step prolongation of $Q^{(0)} \rightarrow M$ and let $\theta$ denote the canonical Pfaff form on $Q$.
Take a section $\sigma$ of $Q \rightarrow Q^{(k)}$ and set $\overline{\theta}:=\sigma^*\theta$ and $\overline{\theta}^{(k)}:=\sigma^*\theta^{(k)}$, where $\theta^{(k)}$ is the $E^{(k)}$-component of $\theta$, that is, $\theta^{(k)}=\pi^{(k)}\theta$, where $\pi^{(k)}:E \rightarrow E^{(k)}=E/ F^{k+1}E$.

Let $V_{z^{(k)}}Q^{(k)}$ denote the vertical tangent space at $z^{(k)}$. Then
$$\overline{\theta}^{(k)}_{z^{(k)}}:V_{z^{(k)}}Q^{(k)} \rightarrow F^0 E^{(k)}=\frak g^{(k)}$$
is an isomorphism of filtered vector space
and
$$\gr \overline{\theta}^{(k)}:\gr V \rightarrow \gr F^0 E^{(k)}$$
is the canonical isomorphism.

For $A \in \frak g^{(k)}$, let $A^*$ denote the vertical vector field on $Q^{(k)}$ determined by $\overline{\theta}^{(k)}(A^*)=A$ and let $a(t)=\exp(tA) \in G^{(k)}$.

Let $z^{(k)}$ be an element of $Q^{(k)}$ and let $z^{(k)}(t)$ be the integral curve of $A^*$ with $z^{(k)}(0) =z^{(k)}$. By Corollary \ref{cor:formula of tau}, the flatness of $\tau^{[k+1]}$ implies
$$A^* \kappa^{[k]} = - \rho(A) \kappa^{[k]}.$$
Integrating this, we have
\begin{eqnarray}   \label{eq:tau k+1 flat and fundamental identity}
\kappa^{[k]}(z^{(k)}(t))=\rho(a(t))^{-1}\kappa^{[k]}(z^{(k)}).
\end{eqnarray}

 By (\ref{eq:kappa k on principal bundle}) and (\ref{eq:tau k+1 flat and fundamental identity}), we have
 $$\kappa^{[k]}(z^{(k)}(t)) = \kappa^{[k]}(z^{(k)}a(t)).$$
Note that $\{z(t)^{(k)}\}$ is a curve in $Q^{(k)}$ and $\{z^{(k)}a(t)\}$ is a curve in $\mathscr R^{(k)}Q^{(k-1)}=\mathscr S_{W^1}^{(k)}Q^{(k-1)}$,
and that both $Q^{(k)}$ and $\mathscr R^{(k)}Q^{(k-1)}$ are principal bundles on $Q^{(k-1)}$.
Furthermore, both $z^{(k)}(t)$ and $z^{(k)}a(t)$ project to the same point in $Q^{(k-1)}$.
%

Recall that  $Q^{(k)}=\mathscr S_W^{(k)}Q^{(k-1)}$ is defined by imposing a condition
 on the value of $\kappa_{[k]}$:
  \begin{eqnarray*}
Q^{(k)}=\left\{z \in \mathscr S_{W^1} ^{(k )}Q^{(k-1)}: \kappa_{[k ]}(z) \in W_{k }^2  \right\}.
\end{eqnarray*}
Since $z^{(k)}(t)$ is contained in $Q^{(k)}$,  so is $z^{(k)}a(t)$. Hence the action of $G^{(k)}$ on $\mathscr R^{(k)}Q^{(k-1)}$ leaves  invariant $Q^{(k)}$. It follows that $Q^{(k)}\rightarrow M$ is a principal subbundle of $\mathscr R^{(k)}Q^{(k-1)}$ with structure group $G^{(k)}$. 

If $\tau$ is flat, then $Q^{(k)} \rightarrow M$ is a proper principal geometric structure for any $k$ and thus its projective limit $Q \rightarrow M$ with the canonical Pfaff form $\theta$ is a Cartan connection.
\end{proof}


\subsection{Condition (C)}  \label{sect:condition C} $\,$



Now assume that $G_0$ satisfies the condition (C), that is, there exists a  $G^0$-invariant subspace $W^2 = \oplus_{i \geq 1} W_i^2$ such that
$$\Hom(\wedge^2 \frak g_-, \frak g)_i = \partial \Hom(\frak g_-, \frak g)_i \oplus W_i^2 \quad \text{ for } i \geq 1.$$
Then by Theorem 3.10.1 of \cite{M93} to each geometric structure $P^{(0)} \rightarrow M$ of order 0 of type $(\frak g_-, G_0)$ there is associated a series of principal bundles $P^{(k)} \stackrel{G^{(k)}}{\rightarrow}M$ defined by
$$P^{(k )} = \{z^{k } \in \mathscr R^{(k )}P^{(k-1)}: \kappa_{[k ]}(z^{k }) \in W_{k }^2\}$$
and the projective limit $P:=\lim_{\leftarrow k} P^{(k)}$ is equipped with an absolute parallelism $\theta_P$ which defines a Cartan connection of type $(\frak g_-, G^0)$. The principal $G^{(k)}$-bundle $P^{(k)} \rightarrow M$ is a proper principal geometric structure in this paper. We call $P^{(k)}$ the {\it $W^2$-normal principal   prolongation of order $k$} of $P^{(0)}$  and call $(P, \theta)$ the {\it $W^2$-normal principal complete prolongation} of $P^{(0)}$ and denote it by $\mathscr R_{W^2}P^{(0)}$.

On the other hand, keeping $W^2= \oplus_{i \geq 1}W_i^2$ as above, we choose complementary subspaces $W^1=\oplus_{i \geq 1} W_i^1$ such that
$$\Hom(\frak g_-, \frak g)_i = \partial \frak g_i \oplus W_i^1 \quad \text{ for } i \geq 1.$$
Then to each geometric structure $Q^{(0)} \rightarrow M$ of order 0 of type $(\frak g_-, G_0)$, there is associated a series of principal bundles $Q^{(k)} \stackrel{G_k}{\rightarrow} Q^{(k-1)}$ and the projective limit $\mathscr S_WQ^{(0)}$,   the $W$-normal complete prolongation of $Q^{(0)}$ constructed in Section \ref{sect: definition of kappa and tau}.

\begin{theorem} \label{thm:Cartan connection} 
Assume that $G_0$ satisfies the condition (C). Let $Q^{(0)}$ be a geometric structure of type $(\frak g_-, G_0)$. Then there is an isomorphism
$$\iota_W: \mathscr R_{W^2} Q^{(0)} \rightarrow \mathscr S_WQ^{(0)}.$$

\end{theorem}

\begin{proof} Let   $Q^{(k)}$ denote the $W$-normal step prolongation of order $k$ of $Q^{(0)}$ and let $P^{(k)}$ denote the  $W^2$-normal principal   prolongation of order $k$  of $P^{(0)}=Q^{(0)}$. We will show that there is an isomorphism $\iota_W: P^{(k)} \rightarrow Q^{(k)}$ for any $k \geq 1$. 

As in the proof of Theorem \ref{thm:vanishing of tau implies}, we define inductively,
\begin{eqnarray*}
\xymatrix{
\mathscr R^{(k+1)}P^{(k)} \ar[r]^{\mathscr R\iota_{W}^{(k)}} \ar[d] & \mathscr S_{W^1}^{(k+1)}Q^{(k)} \ar[d]\\
P^{(k)} \ar[r]^{\iota_W^{(k)}} & Q^{(k)}.
}
\end{eqnarray*}
%
Furthermore, $P^{(k+1)}$ and $Q^{(k+1)}$ are defined by the same condition:
\begin{eqnarray*}
P^{(k+1)}=\{z^{k+1} \in \mathscr R^{(k+1)}P^{(k)} : \kappa_{[k+1]}(z^{k+1}) \in W^2_{k+1} \} \\
Q^{(k+1)}=\{z^{k+1} \in \mathscr S^{(k+1)}_{W^1}Q^{(k)} : \kappa_{[k+1]}(z^{k+1}) \in W^2_{k+1} \}.
\end{eqnarray*}
Thus $\mathscr R\iota_{W}^{(k)}$ sends $P^{(k+1)} $ onto $Q^{(k+1)}$. 
This completes the induction and gives the isomorphism
$$\iota_W: \mathscr R_{W^2}P^{(0)} \rightarrow \mathscr S_WQ^{(0)}.$$
\end{proof}

We remark that in Theorem \ref{thm:Cartan connection}  $\tau^{[k+1]}$ on $\mathscr S_{W^1}Q^{(k)}$ is flat for any $k \geq 0$ by Proposition \ref{prop:W-normal and principal prolongation}. \\

For the sake of consistence of our description, let us prove again directly not relying on Theorem 3.10.1 of \cite{M93} that the $G_{k+1}$-principal bundle $Q^{(k+1)} \rightarrow Q^{(k)}$ constructed as above turns out to be a principal $G^{(k+1)}$-bundle $Q^{(k+1)} \rightarrow M$.

Suppose that $Q^{(k)} \rightarrow M$ is a principal $G^{(k)}$-bundle. Construct $\mathscr R^{(k+1)}Q^{(k)}$ and $\mathscr S_{W^1}^{(k+1)}Q^{(k)}$, and identify $\mathscr R^{(k+1)}Q^{(k)}$ with $\mathscr S_{W^1}^{(k+1)}Q^{(k)}$.  Then $Q^{(k+1)}$ is defined by
$$Q^{(k+1)}=\{z^{k+1} \in \mathscr S_W^{(k+1)}Q^{(k)} = \mathscr R^{(k+1)}Q^{(k)}: \kappa_{[k+1]}(z^{k+1}) \in W^2_{k+1}\}.$$
Let $\mathscr RQ^{(k)}$ be the complete universal principal frame bundle of $Q^{(k)}$. Denote by $G_{\mathscr R}$ and $G_{\mathscr R}^{(k+1)}$ the structure group of $\mathscr RQ^{(k)}$ and $\mathscr R^{(k+1)}Q^{(k)}$ respectively. Then the structure function $\gamma$ of $\mathscr RQ^{(k)}$ satisfies
$$\gamma(za) = \rho(a)^{-1} \gamma(z) \quad \text{ for } z \in \mathscr RQ^{(k)} \text{ and } a \in G_{\mathscr R},$$
which descends to the structure function $\gamma^{[k+1]}$ of $\mathscr R^{(k+1)}Q^{(k)}$ satisfying
$$\gamma^{[k+1]}(z^{k+1}a^{k+1}) = \rho(a^{k+1})^{-1} \gamma^{[k+1]}(z^{k+1}) \quad \text{ for } z^{k+1} \in \mathscr R^{(k+1)}Q^{(k)} \text{ and } a^{k+1} \in G_{\mathscr  R}^{(k+1)}. $$
Since $\tau^{[k+1]}$ is flat, so in $\sigma^{[k+2]}$ by Corollary \ref{cor: tau flat implies sigma flat}, and thus we can write
$$\gamma^{[k+1]} = \kappa_+^{[k+1]} + \gamma_0^{[k+1]}$$
where $\kappa_+^{[k+1]} = \sum_{1 \leq i \leq k+1}\kappa_i^{[k+1]}$ and $\gamma_0^{[k+1]}$ is the bracket of $\frak g_- \oplus \frak g^{(k+1)}$.
Now that $G^0$ preserves the bracket of $\frak g$, $G^{(k+1)}$ preserves $\gamma_0^{[k+1]}$. It then follows that
$$\rho(a^{k+1})^{-1}(\kappa_+^{[k+1]} + \gamma_0^{[k+1]}) = \rho(a^{k+1})^{-1} \kappa_+^{[k+1]} + \gamma_0^{[k+1]}.$$
By the assumption that $W^2$ is $G^0$-invariant, $W^{2,[k+1]}=\oplus_{1 \leq i \leq k+1}W^2_i$ is $G^{(k+1)}$-invariant. Hence we deduce that if $z^{k+1} \in Q^{(k+1)}$ and $a^{k+1} \in G^{(k+1)}$, then
$$\kappa_+ ^{[k+1]}(z^{k+1}a^{k+1}) = \rho(a^{k+1})^{-1} \kappa_+^{[k+1]}(z^{k+1}) \in W^{2,[k+1]}$$
and thus $z^{k+1}a^{k+1} $ belongs to $ Q^{(k+1)}$. This proves that $Q^{(k+1)}$ is a principal $G^{(k+1)}$-bundle. \\

By Theorem \ref{thm:Cartan connection}, if the condition (C) is satisfied, then the $W$-normal complete step prolongation $\mathscr S_WQ^{(0)}$   produces a Cartan connection.
We remark that $\mathscr S_WQ^{(0)}$ does not depend on the choice of $\{W^1_i\}$ because $\tau$ is flat in this case.


\subsection{Inductive vanishing  of the structure function $\kappa$}

\begin{proposition} \label{prop: curvature vanishing}   
Let  $Q^{(0)} \stackrel{G_0}{\rightarrow} (M,F) $ be a geometric structure  of order $0$ of type $(\frak g_-, G_0)$ with $G_0$ being connected.
Let $Q^{(k)}$ be the $W$-normal step prolongation of order $k$ of   $Q^{(0)} \stackrel{G_0}{\rightarrow} (M,F) $.
If $\kappa^{[k]}$ is flat, 
then
 \begin{enumerate}
 \item $\tau^{[k+1]}$ is flat  and
%
 \item
$Q^{(k )}$ is a principal bundle on $M$ with structure group $G^{(k )} $ and
 \item
$\kappa_{[k+1]}$ induces a section $s_{\kappa_{[k+1]}}$ of   the associated vector bundle
  $$\mathcal H^2_{k+1}:=
  Q^{(0)} \times_{G_0}H^2(\frak g_-, \frak g )_{k+1}$$ on $M$, the vanishing of which implies the vanishing of $\kappa_{[k+1]}$.
 \end{enumerate}
\end{proposition}


 \begin{proof}
Note that $\kappa_{[1]}$ is a function on $Q^{(1)}$ with values in $\Hom(\wedge^2 \frak g_-, \frak g)_{1}$.
 From Corollary \ref{cor:vanishing of kappa}
 we get $\partial \kappa_{[1]} =0$.
By Proposition \ref{prop:structure equation along fiber},  $\kappa_{[1]}(z(1+A)) = \kappa_{[1]}(z) + \partial A$ for $A \in \frak g_1$  and thus $\kappa_{[1]}(za) \equiv \kappa_{[1]}(z) \mod \partial \frak g_1$ for $a \in G_1$. Therefore, $\kappa_{[1]}$ induces a function
 $$\overline{\kappa} _1: Q^{(0)} \rightarrow H^2(\frak g_-, \frak g )_1.   $$
Since $\overline{\kappa}_{ 1}(za) = a^{-1}\overline{\kappa}_{ 1}(z) $ for $a \in G_0$,
  the function $\overline{\kappa} _1$ induces a section $s_{\kappa_1}$ of the vector bundle $\mathcal H^2_1:=Q_0 \times_{G_0} H^2(\frak g_-, \frak g )_1$ on $M$. If  the section $s_{\kappa_1}$ vanishes, then the function  $\overline{\kappa} _1$ vanishes and thus  $\kappa_1 = \partial \chi$ for some $\chi:Q_0 \rightarrow \Hom(\frak g_-, \frak g)_1$. Since $\kappa_{[1]}$ has values in $W_1^2$ and $W_1 ^2\cap \partial \Hom(\frak g_-, \frak g)_{ 1}  =0$, we have $\kappa_{[1]}=0$.  \\

Now assume that $\kappa^{[k]}$ is flat, i.e.,  $\kappa _{[i]} =0$ for all $1 \leq i \leq k$. Then by the induction hypothesis, $ \tau^{[k]} $ is flat.
By Corollary \ref{cor:vanishing of kappa}, we have
 \begin{enumerate}
 \item [(i)] $\partial \kappa_{[k+1]} =0$;
 \item [(ii)] $\partial (\tau_{[k+1]}(A,\,\cdot\,))=0$ for any $A \in \frak g_{a}$, where $0 \leq a \leq k-1$.
 \end{enumerate}

By (ii)  $\tau_{[k+1]}(A, \,\cdot\,)$ is contained in  $\Ker \partial$.
Since $\frak g$ is the prolongation of $(\frak g_-, \frak g_0)$, we have $H^1_+(\frak g_-, \frak g)=0$. Thus    $\tau_{[k+1]}(A, \,\cdot\,)$ is contained in $\partial \frak g_{k }$. From  $\tau_{[k+1]}(A, \,\cdot\,) \in W_{k }^1$ and $W_{k }^1 \cap \partial \frak g_{k } =0$, it follows that  $\tau_{[k+1]}(A, \,\cdot\,) =0$ for any $A \in \frak g_{a }$, where $0 \leq a \leq k-1$. 
By Theorem  \ref{thm:vanishing of tau implies}
 $Q^{(k)}\rightarrow M$ is a proper principal geometric structure with structure group $G^{(k)}$. 
Furthermore, we have
  $$R_a^*(\kappa_0 +\kappa_1 + \dots +\kappa_{[k+1]}) =\rho(a)^{-1}(\kappa_0+ \kappa_1 + \dots +\kappa_{[k+1 ]}) \text{ for } a \in G^{(k)}.$$
  Hence we have
  \begin{eqnarray*}
  R_{a_0}^* \kappa_{[k+1]} &=& \rho(a_0)^{-1} \kappa_{[k+1]} \text{ for } a_0 \in G_0\\
  \widetilde{A} \kappa_{[k+1]} &=& -\rho(A) \kappa_{[k+1-a]} \text{ for }  A \in \frak g_a.
  \end{eqnarray*}
 By the assumption that $\kappa_{[i]}=0$ for all $1 \leq i \leq k$, for $a \in G^{(k)}$ we have
 $$R_a^*\kappa_{[k+1]} = \rho(a_0)^{-1} \kappa_{[k+1]}$$
where $a_0$ is the $G_0$-component of $a$.

By  Proposition \ref{prop:structure equation along fiber}
  $\kappa_{[k+1]}(z(1+A)) = \kappa_{[k+1]}(z) + \partial A$ for $A \in \frak g _{k+1}$, and thus    $\kappa_{[k+1]}(za) \equiv \kappa_{[k+1]}(z) \mod \partial \frak g _{k+1}$ for $a \in G_{k+1}$.   Together with  (i), it follows that   $\kappa_{[k+1]}:Q^{(k+1)} \rightarrow \Hom(\wedge^2 \frak g_-, \frak g)_{k+1}$ induces  a function
$$\overline{\kappa }_{k+1}: Q^{(k)} \rightarrow H^2(\frak g_-, \frak g )_{k+1},     $$
which again  induces a section $s_{\kappa_{k+1}}$ of the vector bundle $\mathcal H^2_{k+1}:=Q^{(k)} \times_{G^{(k)}} H^2(\frak g_-, \frak g )_{k+1}$ on $M$. By the same arguments as in the case when $k=0 $, if $s_{\kappa_{k+1}}$ vanishes, then $\kappa_{k+1}$ vanishes. This completes the proof of Proposition \ref{prop: curvature vanishing}
%
 \end{proof}

\begin{theorem} \label{thm: sections all vanishing}
 Let $Q^{(0)} \stackrel{G_0}{\rightarrow} (M,F) $ be a geometric structure of order $0$ with $G_0$ being connected. Assume that there is no nontrivial section  of $\mathcal H_k=Q^{(0)} \times_{G_0} H^2(\frak g_-, \frak g)_k$ for all $k\geq 1$.
 Then  $(\mathscr S_W Q^{(0)}, \theta)$ is a Cartan connection of type $G/G^0$  which is flat, i.e., whose  curvature vanishes.
\end{theorem}

\begin{proof}
It follows from Proposition \ref{prop: curvature vanishing}  and Theorem \ref{thm:vanishing of tau implies}.
\end{proof}

In the smooth category, 
there always exists a nontrivial smooth section of the smooth vector bundle $\mathcal H_k=Q^{(0)} \times_{G_0} H^2(  \frak g_-, \frak g)_k$ unless  $H^2(\frak g_-, \frak g)_k$ is zero. Thus Theorem \ref{thm: sections all vanishing} cannot be applied, and we need an explicit formula for structure functions to show their flatness or to compare them.   Relations among structure functions as in Theorem \ref{thm:fundamental identities} will play a more important role in this category. However, in the complex analytic category, certain conditions imposed on the geometry of the base manifold enforce the vanishing of holomorphic sections of the holomorphic vector bundle  $\mathcal H_k$, which will be explained  with a concrete example in subsection \ref{sect:complex geometry}.


\section{Applications to subriemannian geometry and complex geometry}


\subsection{Subriemannian geometry} \label{sec:subriemannian}


\subsubsection{}
A {\it subriemannian manifold}
$(M, D, s)$
 is a smooth manifold
$M$ equipped with
  a subbundle $ D \subset TM$
and a fibre metric $s$ of $D$, that is,
 a smooth section $s$
of the symmetric tensor product
$S^2D^* $ of the dual vector bundle $D^*$.
Since around 1980 subriemannian geometry has been studied
from many different domains, nilpotent geometry and analysis, riemannian and symplectic
geometry, control theory, etc.(\cite{AS04}, \cite{Mont02}, \cite{LS95}).

Subriemannian geometry, as a generalization of riemannian geometry, inherits from it
to some extent main notions and properties.
For instance, we may speak of a length minimizing curve $\gamma$ joining two points of a subriemannian manifold $(M, D, s)$ among the integral curves of $D$, and we may also speak of the Hamiltonian flow of the energy function associated with  a subriemannian metric.

A big surprise was brought by the discovery that
there  exists a minimizing curve of a subriemannian manifold  called
{\it abnormal} geodesic which appears   depending only on the distribution and does not satisfy the usual geodesic equation.

Another  difference between two geometries
comes from the fact that the first order approximation at
 a point of a riemannian manifold is nothing but an Euclidean vector space, while
the first order approximation of a subriemannian manifold
is a pair $(\mathfrak g _- , \sigma )$
of a nilpotent graded Lie algebra
$ \mathfrak g_- = \bigoplus _{p < 0}
\mathfrak g _p$
and an inner product
$\sigma $ on $\mathfrak g_{-1}$,
which has a great deal of variety.
The distributions $D$ are highly locally non trivial.
This gives to subriemannian geometry much more variety than to riemannian geometry.

   Moreover, it had not been clear how to define
   the curvatures of a subriemannian manifold.
   In the next paragraph, as an application of our general method,  we show how
   to define the curvature of a subriemannian manifold of constant symbol along with
   a general algorithm to compute it.

   \subsubsection{ } 

We define a {\it subriemannian filtered manifold} $(M,F,s)$  to be a filtered manifold $(M, F)$
equipped with a fibre metric
$s$ on $F^{-1}$.

When we consider a subriemannian  manifold $(M,D,s)$, the distribution $D$ is usually
assumed {\it bracket generating}.
There is a bijective correspondence  between the subriemannian manifolds $(M,D,s)$ with regularly bracket
generating   and the subriemannian filtered manifolds $(M,F,s)$ with $ F$ being  generated by $F^{ -1 }$.


Here after we consider subriemannian
manifolds $(M, F, s) $ whose metric $s$ is non-degenerate,
but not necessarily positive definite.

For each $ x \in M$ there is associated
a pair $(gr F_x,s_x)$
called the  {\it subriemannianan  symbol} of $(M, F, s)$ at $x$,
 where $s_x$ is viewed as an inner product
on $gr_{-1} F_x$.
We say that $ (M, F, s)$ has {\it constant symbol of type}
$(\mathfrak g_{-}, \sigma )$,
where
$ \mathfrak g _{-}
= \bigoplus _{p < 0 } \mathfrak g _p$
is a nilpotent graded Lie algebra and
$\sigma $
an inner product on $ \mathfrak g _{-1} $,
if there exists for every $x $
$$ \text{ a graded Lie algebra isomorphism } z :
\mathfrak g _- \to gr F_x \text{  such that  }
z|_{\frak g_{-1}}^*s_x = \sigma.
$$


Now consider a subriemannian filtered manifold $(M,F,s)$ of constant symbol $(\frak g_-, \sigma)$. Let $\mathscr S^{(0)}(M,F)$ be the universal frame bundle of $(M,F)$ of order 0 and let $ Q^{(0)}(M,F,s)$ be a principal subbundle of $\mathscr S^{(0)}(M,F)$ defined by
$$Q^{(0)}(M,F,s)_x:=\{ z: \frak g_- \rightarrow \gr F_x \mid \text{ graded Lie algebra isomorphism satisfying } z|_{\frak g_{-1}}^* s_x = \sigma \}.$$
The structure group of $Q^{(0)}(M,F,s)$ is
$$G_0(\frak g_-,\sigma):=\{ a: \frak g_- \rightarrow \frak g_- \mid \text{ graded Lie algebra automorphism satisfying } a|_{\frak g_{-1}}^* \sigma = \sigma \}$$
and we denote its Lie algebra by  $\frak g_0(\frak g_-, \sigma)$.

Consider the prolongation of the truncated transitive graded Lie algebra $\frak g_- \oplus \frak g_0(\frak g_-, \sigma)$ and denote it by 
$\frak g(\frak g_-, \sigma)$. 
$$\frak g(\frak g_-, \sigma) = \oplus \frak g_p(\frak g_-, \sigma)  .$$

\begin{proposition}
[\cite{M08}] If $\frak g_-$ is generated by $\frak g_{-1}$ and $\sigma$ is positive definite, then $\frak g_p(\frak g_-, \sigma)$ vanishes for $p >0$.
\end{proposition}

For the proof we use Yatsui's result on completely reducible graded Lie algebra (\cite{Y02}).

If $\sigma$ is neither positive definite nor negative definite, the above proposition does not hold. We know examples of indefinite $(\frak g_-,\sigma)$ for which $\frak g_1(\frak g_-,\sigma)$ does not vanish.
However, we have:

\begin{proposition}
If $\frak g_-$ is generated by $\frak g_{-1}$, then $(\frak g_-, \sigma)$ is of finite type, that is, the prolongation $\frak g_p(\frak g_-, \sigma)$ vanishes for large enough $p$.
\end{proposition}

This follows from Tanaka's criterion that $\frak g_- \oplus \frak g_0$ ($\frak g_-$ being generated by $\frak g_{-1}$) is of finite type if and only if $\frak f_{-1} \oplus \frak f_0$ is of finite type, where $\frak f_{-1}=\frak g_{-1}$ and $\frak f_0 =\{A \in \frak g_0:[A,X_p] =0 \text{ for } X_p \in \frak g_p \text{ and } p<-1 \}$ (\cite{T70}). 

To study the equivalence problem of subriemannian structures, the following is fundamental.

\begin{theorem} [\cite{M08}] Let $\frak g_- = \bigoplus_{p<0} \frak g_p$ be a graded Lie algebra generated by $\frak g_{-1}$ and $\sigma$ a positive definite inner product on $\frak g_{-1}$. Then for every subriemannian filtered manifold $(M,F,s)$ having constant symbol of type $(\frak g_-, \sigma)$, there exists canonically a Cartan connection $(P, G_0, \theta)$ of type $(\frak g_- \oplus \frak g_0(\frak g_-,\sigma), G_0(\frak g_-, \sigma))$.

\end{theorem}


The construction of the canonical Cartan connection  is made according to the general construction (\cite{M93}) by applying the criterion, the condition (C) to this case.
Hence for positive definite case, the curvature of the Cantan connection associated with a subriemannian structure gives a generalization of riemannian curvature.



 While
in the case of indefinite metrics,  the condition (C) being not assured,  we have no
 hope to have  Cartan connection, in general,
associated  with subriemannian $G_0$-structure $Q^{(0)}$,
  but it is
  our step prolongation $\mathscr S_WQ^{(0)}$
  and fundamental identities
  that enable us to define the subriemannian curvature
  and give  a theoretical basis for studying  the equivalence problem of the indefinite
 subriemannian structures.
 Indefinite subriemannian structures are not yet studied much but we know interesting examples in subriemannian contact
 geometry and in subriemannian geometry associated with Clifford modules (\cite{FGMMV}).

 As we see even in the simplest example of subriemannian contact structure, there are a great variety of subriemannian symbols, and it  is rather restrictive to treat only subriemannian structure of constant symbols, and we are naturally led to consider subriemannian structures of nonconstant symbol.

 This is one of motivations for us to extend the present scheme to geometric structures of nonconstant symbol. 


\subsection{Complex geometry} \label{sect:complex geometry}

\subsubsection{ }  One of the import geometric structures of order 0 is $G_0^{\sharp}$-structure of type $\frak g_-$, which was studied by Tanaka in \cite{T79}.  Let $\frak g=\oplus_{i=-\mu}^{\mu} \frak g_i$ be a simple graded Lie algebra over $\mathbb R$ or $\mathbb C$.  Let $\frak g_-$ be the negative part of $\frak g$ and let $G_0^{\sharp}$ be the closed subgroup $G_0\cdot N^0$, where $G_0$ is  a Lie subgroup of  the automorphism group of $\frak g$ and $N^0$ is the subgroup of $GL(\frak g_-)$ consisting of all $a \in GL(\frak g_-)$ such that
$$a X\equiv X \mod \sum_{j= p+1}^{-1}\frak g_j \text{ for all } X \in \frak g_p, \text{ where } p<0.$$
A $G_0^{\sharp}$-structure of type $\frak g_-$ on a manifold $M$ is a reduction $P^{\sharp}$ of the  linear  frame bundle of $M$, whose structure group is $G_0^{\sharp}$.

To the simple graded Lie algebra $\frak g$, there is associated a homogeneous space $G/G^0$, where $G$ is a Lie group with Lie algebra $\frak g$ and $G^0$ is the subgroup of $G$ corresponding to the nonnegative part $\frak g^0=\oplus _{i \geq 0}^{\mu} \frak g_i$. Under the assumption that $\frak g$ is the prolongation of $(\frak g_-, \frak g_0)$, the equivalence problem for $G_0^{\sharp}$-structure $P^{\sharp}$ of type $\frak g_-$ can be solved by associating a normal Cartan connection $(P, \theta)$ of type $G/G^0$ to $P^{\sharp}$ (Theorem 2.7 of \cite{T79}). Furthermore, the harmonic part $H(K)$ of the curvature $K$ of $(P, \theta)$ gives a fundamental system of invariants, i.e., the vanishing of $H(K)$ implies the vanishing of $K$, and vice versa (Theorem 2.9 of \cite{T79}).

A $G_0^{\sharp}$-structure of type $\frak g_-$ can be defined via a fiber subbundle of the projective tangent bundle $\mathbb P(TM)$. Let ${\bf S} \subset \mathbb P(\frak g_{-})$ denote the closed $G_0$-orbit in $\mathbb P(\frak g_{-1})$.
  A fiber subbundle $\mathcal S $ of the  projective  tangent bundle $\mathbb P(TM)$ defines a $G_0^{\sharp}$-structure if  the embedding $\mathcal S_x \subset \mathbb P(T_xM)$ is projectively equivalent to the embedding $ {\bf S} \subset \mathbb P(\frak g_-)$.   

  The theory of Tanaka applied to give a characterization of the homogeneous space $G/G^0$ in the complex analytic category as follows.

  \begin{theorem} \cite{HH08} \label{thm:rational homogeneous} Let $X$ be a homogeneous space $G/G^0$ associated to a long simple root.
  Let $M$ be a Fano manifold of Picard number one. Assume that the variety $\mathcal C \subset \mathbb P(TM)$ of minimal rational tangents of $M$ defines a $G_0^{\sharp}$-structure on $M$, that is, the embedding $\mathcal C_x \subset \mathbb P(T_xM)$ is projectively equivalent to the embedding $ {\bf S} \subset \mathbb P(\frak g_-)$ for a general point $x \in M$.
  Then $M$ is biholomorphic to $G/G^0$.
  \end{theorem}

\subsubsection{ }
 The theory of Tanaka on simple graded Lie algebra $\frak g$ is generalized to the case when $(\frak g_-, \frak g_0)$ satisfies the condition (C) in \cite{M93}.
Another development on the construction of Cartan connection in the complex analytic category is given as follows.

\begin{theorem} [\cite{HL19}] \label{thm:symplectic grass}

Let $X=G/G^0$ be a symplectic Grassmannian $Gr_{\omega}(k,V)$   and let ${\bf S} \subset \mathbb P(\frak g_-)$ be the variety of minimal rational tangents of $X$ at a general point $o$. Let $M$ be a Fano manifold of Picard number one and $\mathcal C \subset \mathbb P(TM)$ be the variety of minimal rational tangents associated to a choice of minimal rational component.
If  the embedding $\mathcal C_x \subset \mathbb P(T_xM)$ is projectively equivalent to ${\bf S} \subset \mathbb P(\frak g_-)$ for a general point $x \in M$, then $M$ is biholomorphic to $X$.

\end{theorem}

In fact, Theorem \ref{thm:symplectic grass} also hold for an odd symplectic Grassmannian $X$ (\cite{HL19}). Here, by an odd symplectic Grassmannian, we mean the space of all isotropic $k$-dimensional subspace of vector space $V$ of dimension $2n+1$ with a skew-symmetric form $\omega$ of maximal rank. Then the odd symplectic Grassmannian $Gr_{\omega}(k,V)$ with $\dim V=2n+1$ is no longer a homogeneous space, while the symplectic Grassmannian $Gr_w(k,V)$ with $\dim V=2n$ is a homogeneous space $G/G^0$.

An odd symplectic Grassmannian  is one of the examples of almost homogeneous manifolds, a compact complex manifold on which its automorphism group has an open orbit. There are only a few results on equivalence problem associated with an almost homogeneous manifold. The next simplest example of an almost homogeneous manifold  is a smooth horospherical variety of Picard number one, classified by Pasquier (\cite{Pa}). An odd symplectic Grassmannian is one of this kind of examples.
 We can now apply the theory developed in this paper to study the rigidity of horospherical varieties.  According to Theorem  \ref{thm: sections all vanishing}, a manifold with a geometric structure $Q^{(0)}$ modeled on $X^0=G/G^0 \subset X$ is locally  equivalent to the model space if there is no nontrivial holomorphic section of $\mathcal H_k=Q^{(0)} \times _{G_0} H^2(\frak g_-, \frak g)_k$ for all $k\geq 1$, as in the case of $G_0^{\sharp}$-structure of type $\frak g_-$ studied by Tanaka.
 %
 Instead of trying to confirm the condition (C), by showing that any holomorphic section of $\mathcal H_k=Q^{(0)} \times _{G_0} H^2(\frak g_-, \frak g)_k$ vanishes for all $k\geq 1$, we get a characterization of
 smooth horospherical varieties of Picard number one  as in Theorem \ref{thm:rational homogeneous} and Theorem \ref{thm:symplectic grass}
  (\cite{HK}).

\begin{theorem} [\cite{HK}]

Let $X $ be a smooth horospherical variety of Picard number one  and let ${\bf S} \subset \mathbb P(\frak g_-)$ be the variety of minimal rational tangents of $X$ at a general point $o$. Let $M$ be a Fano manifold of Picard number one and $\mathcal C \subset \mathbb P(TM)$ be the variety of minimal rational tangents associated to a choice of minimal rational component.
If  the embedding $\mathcal C_x \subset \mathbb P(T_xM)$ is projectively equivalent to ${\bf S} \subset \mathbb P(\frak g_-)$ for a general point $x \in M$, then $M$ is biholomorphic to $X$.

\end{theorem}





\section*{Appendix}

 In this Appendix we give an elementary proof of the following proposition: \\

\noindent{\bf Proposition \ref{prop:weighted estimate}.} {\it
Let $X_1, \dots, X_n$ be $n$-independent analytic vector fields defined on a neighborhood of $0$ in  $\mathbb R^n$ or $\mathbb C^n$. A formal function $f$ at $0$ is convergent in a neighborhood of $0$ if and only if there exists positive constants $C$ and $\rho$ such that}
$$(*) \qquad \qquad |(X_{i_1} X_{i_2}\dots X_{i_{\ell}} f)(0)| \leq C \ell!\rho^{\ell} \text{ for } \ell=0,1,2, \dots, (i_1, i_2, \dots, i_{\ell}) \in \{1,2,\dots, n\}^{\ell}.$$


\begin{proof} [Proof of  Proposition \ref{prop:weighted estimate}]
Note first that it is easy to see the proposition holds if $\{X_1, \dots, X_n\}$ is a coordinate frame $\{\frac{\partial}{\partial x^1}, \dots, \frac{\partial}{\partial x^n}\}$, where $(x^1, \dots, x^n)$ is a coordinate system of $\mathbb R^n$ or $\mathbb C^n$. We will therefore show that if $(*)$ holds for $u$ with respect to $X_1, \dots, X_n$, then a similar estimate for  $u$ holds with respect to  $\{\frac{\partial}{\partial x^1}, \dots, \frac{\partial}{\partial x^n}\}$.

Let $V$ be an $n$-dimensional vector space with a basis $\{e_1, \dots, e_n\}$. Write
$$X=e_1 \otimes X_1 + \dots + e_nX_n =\left(\begin{array}{c}X_1\\ \vdots\\X_n\end{array} \right) \qquad D=e_1 \otimes D_1 + \dots + e_n \otimes D_n = \left( \begin{array}{c}D_1\\ \vdots\\D_n\end{array}\right)$$
where $D_i = \frac{\partial}{\partial x^i}$ for $i=1, \dots, n$.
Then we can write
$$D= AX$$
with $GL(V)$-valued holomorphic function $A=A(x^1, \dots, x^n)$. By successive differentiation we have
\begin{eqnarray*}
Du &= & AXu \\
D^2 u &=& (DA) Xu + A^2 X^2 u \\
D^3u &=& (D^2A) Xu + ((DA)A + D(A^2)) X^2 u + A^3 X^3 u
\end{eqnarray*}
and in general
$$D^k u = \Phi^k_1 X u + \Phi^k_2 X^2 u + \dots + \Phi^k_k X^k u $$
with
\begin{eqnarray*}
\Phi^k_i &=& \Phi^{k-1}_{i-1} A + D \Phi^{k-1}_i \\
\Phi^k_i &=& 0 \text{ for } i \leq 0 \text{ or } i >k.
\end{eqnarray*}
In the above formula $D^k=\otimes ^k D$ and $X^k = \otimes ^k X$ should be regarded as sections of $\otimes ^k V \otimes \mathscr D$, where $\mathscr D$ denotes the sheaf of differential operators on $\mathbb C^n$.

Consider a directed graph whose vertex set is $\{(i,j): i \geq j, 1 \leq i,j \leq k\}$ and whose edge set consists of all arrows $(i,j) \rightarrow (i',j')$ satisfying either ($\searrow$) $i+1=i'$ and $j+1=j'$  or ($\downarrow$) $i+1=i'$ and $j=j'$.

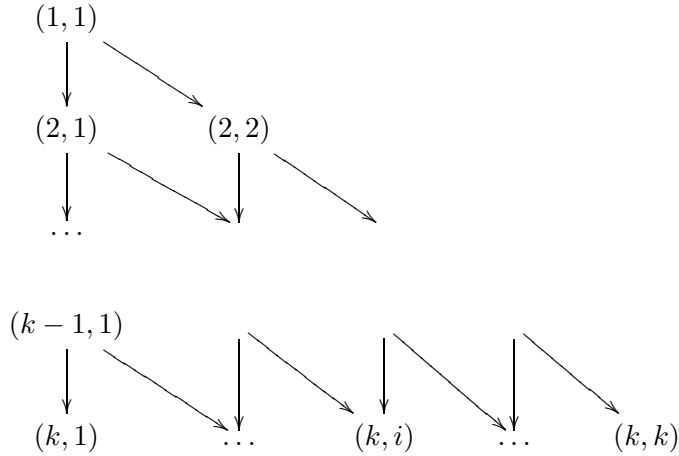
\begin{eqnarray*}
\xymatrix{
(1,1) \ar[d] \ar[dr] &&&& \\
(2,1) \ar[d] \ar[dr] & (2,2)\ar[d] \ar[dr] &&&\\
\dots &&&& \\
(k-1,1)\ar[d] \ar[dr]& \ar[d] \ar[dr]  & \ar[d]\ar[dr] &\ar[d]\ar[dr]& \\
(k,1) & \dots & (k,i) & \dots &  (k,k)
}
\end{eqnarray*}

\noindent
Thus any path $\Gamma$ from $(1,1)$ to $(k,i)$ has $(i-1)$ edges of the first type $\searrow$ and $(k-i)$ edges of the second type $\downarrow$.
 By the recursion formula we see that $\Phi^k_i$ is the sum of all elements $\varphi^k_{i;\Gamma}$    for every path $\Gamma$ from $(1,1)$ to $(k,i)$, which is created from $A$ by multiplying $A$ on the right when the path $\Gamma$ passes an edge of the first type $\searrow$ and by applying  differentiation  $D$ when it passes an edge of the second type $\downarrow$. \\

We endow $\otimes^kV$ with a norm defined by
$$||\sum a_{i_1 \dots i_k} e_{i_1} \otimes \dots \otimes e_{i_k}||:={\rm sup} |a_{i_1 \dots i_k}|.$$
Then $(*)$ may be written as $||X^k u(0)|| \leq C k! \rho^k$, and we want to show
$$ ||D^k u(0) || \leq \overline{C} k! \overline{\rho}^k$$
for some positive constants $\overline{C}$ and $\overline{\rho}$. For that we will show that
$$|\Phi^k_i(0)| \leq C_1 k(k-1) \dots(i+1){ \rho_1}^k \text{ for all } k$$
where $C_1$ and $\rho_1$ are positive constants independent of $k$. \\

To prove this we use the following norm $| \,\cdot\,|_r$ with $r=(r_1, \dots, r_n) \in \mathbb R^n_{> 0}$ for formal power series $F=\sum_{k=0}^nF_k$, $F_k = \sum_{|\alpha|=k} f_{\alpha} x^{\alpha}$, $f_{\alpha} \in \mathbb C $:
\begin{eqnarray*}
|F_k|_r &:=& \mathrm{sup}_{|\alpha|=k} \left( \frac{\alpha !}{|\alpha|!} |f_{\alpha}|r^{\alpha} \right) \\
|F|_r&:=& \sum |F_k|_r.
\end{eqnarray*}

\medskip  \noindent {\bf Lemma A.} {\it
\begin{enumerate}
\item If $F$ is convergent, then there is $r$ such that $|F|_r < \infty$.
\item $|F_k G_{\ell}|_r \leq |F_k|_r |G_{\ell}|_r$, where $F_k$ and $ G_{\ell}$ are homogeneous polynomials of degree $k$ and $\ell$, respectively.
\item $|FG|_r \leq |F|_r |G|_r$.
\item $|DF_k|_r \leq \frac{k}{r_{\min}}|F_k|_r$ for homogeneous polynomial of degree $k$,
where $r_{\min}:=\min\{r_1, \dots, r_n\}$.
\end{enumerate}
}

\begin{proof} Let us prove (2). Let $F_k = \sum_{|\alpha|=k}f_{\alpha}x^{\alpha}$, $G_{\ell} = \sum_{|\beta|=\ell} g_{\beta}x^{\beta}$, and $H_m   = \sum_{|\gamma|=m} h_{\gamma}x^{\gamma}$, where $m=k+\ell$ and $H_m=F_k G_{\ell}$. Recall that $|H_m|_r = {\rm sup}_{|\gamma|=m}\left( \frac{\gamma !}{|\gamma|!} |h_{\gamma}| r^{\gamma}\right)$. But for any $\gamma$ with $|\gamma|=m$ we have
\begin{eqnarray*}
\frac{\gamma !}{|\gamma|!} |h_{\gamma}| r^{\gamma} &=& \frac{\gamma !}{|\gamma|!} \left|\sum_{\alpha+ \beta=\gamma}f_{\alpha}g_{\beta}\right| r^{\gamma} \\
&\leq & \frac{\gamma !}{|\gamma|!}  \sum_{\substack{\alpha+\beta=\gamma\\[2 pt]|\alpha|=k, |\beta|=\ell }}  \frac{|\alpha|! |\beta|!}{\alpha ! \beta !} \left( \frac{\alpha !}{|\alpha|!} |f_{\alpha}|r^{\alpha}\right) \left( \frac{\beta !}{|\beta|!} |g_{\beta}|r^{\beta}\right) \\
&\leq & \frac{k ! \ell !}{m !} \left(
\sum_{\substack{\alpha+\beta = \gamma\\[2 pt] |\alpha|=k, |\beta|=\ell  } } \frac{ \gamma !} {\alpha ! \beta !}
\right) |F_k|_r |G_{\ell}|_r \\
&=& |F_k|_r |G_{\ell}|_r
\end{eqnarray*}
because of the identity
$$\sum_{\substack{\alpha+\beta = \gamma\\[2 pt] |\alpha|=k, |\beta|=\ell  }} \frac{ \gamma !} {\alpha ! \beta !} = \frac{m !} {k! \ell !} $$
which is derived from the binary expansion.

The other assertions (1), (3), (4) are easy to verify.
\end{proof}

Now let us return to the proof of Proposition \ref{prop:weighted estimate}. Take $r=(r_1, \dots, r_n)$ such that $|A|_r < \infty$.
We see that
$$|\Phi A|_r \leq K |\Phi|_r |A|_r.$$
Denote by $Trun^{(\ell)}F:=\sum_{i \leq \ell}F_i$ for $F=\sum F_{\ell}$ and set
$$^{k} \Psi^{\ell}_i=Trun ^{(k-\ell+1)}\Phi^k_i.$$
Then $\Phi^k_i(0)= ^{k}\Psi^{k}_i(0)$.  By Lemma A 
(4)
we have
$$| D (^{k} \Psi^{\ell}_i)|_r \leq \frac{ k-\ell+1 }{r_{\min}}|^{k} \Psi^{\ell}_i|_r.$$
Thus we see that
$$|\varphi^k_{i;\Gamma}(0)| \leq L k(k-1) \dots (i+1) {\rho''}^k.$$
Since the number of path $\Gamma$ from $(1,1)$ to $(k,i)$ is less than $2^k$, we have finally
$$|\Phi^k_i(0)| \leq C_1 k(k-1) \dots(i+1){ \rho_1}^k $$
which completes the proof of Proposition \ref{prop:weighted estimate}.
\end{proof}




\end{document}